\documentclass[12pt, reqno]{amsart}
\usepackage{latexsym,amsmath,amsopn,amssymb,amsthm,amsfonts}
\usepackage[T2A]{fontenc}
\usepackage[cp1251]{inputenc}
\usepackage[english]{babel}
\usepackage[backref=page, breaklinks=true,colorlinks=true,linkcolor=blue,citecolor=blue,urlcolor=blue]{hyperref}

\usepackage{graphicx}

\renewenvironment{proof}[1][Proof]{\textbf{#1.} }{\ \rule{0.5em}{0.5em}}

\DeclareMathOperator{\GT}{GT}
\DeclareMathOperator{\DT}{DT}
\DeclareMathOperator{\DF}{DF}
\DeclareMathOperator{\BFT}{BFT}
\DeclareMathOperator{\CFT}{CFT}
\DeclareMathOperator{\Cyl}{Cyl}
\DeclareMathOperator{\sgn}{sgn}

\renewenvironment{proof}[1][Proof]{\textbf{#1.} }
{\ \rule{0.5em}{0.5em}}
\newtheorem{theorem}{Theorem}
\newtheorem{prop}{Proposition}
\newtheorem{lemma}{Lemma}
\newtheorem{corollary}{Corollary}

\theoremstyle{definition}
\newtheorem{definition}{Definition}
\newtheorem{remark}{Remark}
\newtheorem{example}{Example}

\newtheorem{problem}{Problem}

\begin{document}
\title
[On face angles of tetrahedra with a given base]
{On face angles of tetrahedra with a given base}
\author{E.V.~Nikitenko, Yu.G.~Nikonorov}

\address{Nikitenko Evgeni\u\i\  Vitalievich\newline
Rubtsovsk Industrial Institute of\newline
Altai State Technical University \newline
after I.I.~Polzunov \newline
Rubtsovsk, Traktornaya st., 2/6, \newline
658207, Russia}
\email{evnikit@mail.ru}

\address{Nikonorov Yuri\u{\i} Gennadievich \newline
Southern Mathematical Institute of \newline
the Vladikavkaz Scientific Center of \newline
the Russian Academy of Sciences, \newline
Vladikavkaz, Vatutina st., 53, \newline
362025, Russia}
\email{nikonorov2006@mail.ru}

\begin{abstract}
Let us consider the set $\Omega (\triangle ABC)$ of all tetrahedra $ABCD$ with a given non-degenerate base $ABC$ in
$\mathbb{E}^3$ and $D$ lying outside the plane $ABC$. Let us denote by
$\Sigma(\triangle ABC)$ the set $\left\{\Bigl(\cos \overline{\alpha},\cos \overline{\beta},\cos \overline{\gamma} \Bigr)\in
\mathbb{R}^3\,|\, ABCD \in \Omega (\triangle ABC)\right\}$,
where  $\overline{\alpha}=\angle BDC$, $\overline{\beta}=\angle ADC$, and $\overline{\gamma}=\angle ADB$.
The paper is devoted to the problem of determining the closure of $\Sigma(\triangle ABC)$ in $\mathbb{R}^3$ and its boundary.

\vspace{2mm}
\noindent
2020 Mathematical Subject Classification:
21M20, 51M16, 51M25, 53A04, 53A05, 57N35

\vspace{2mm} \noindent Key words and phrases: angle, base, Cayley--Menger determinant, Euclidean space, extremal problems, inscribed quadrilateral, tetrahedron, triangle.
\end{abstract}

\maketitle

\section{Introduction}

\begin{figure}[t]
\center{\includegraphics[width=0.5\textwidth]{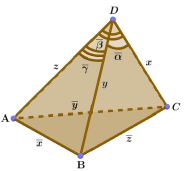}}\\
\caption{Tetrahedron $ABCD$.}
\label{jfp00}
\end{figure}

This paper is devoted to the study of some properties of tetrahedra in $3$-dimensional Euclidean space $\mathbb{E}^3$.
In what follows, $d(L,M)$ denotes the Euclidean distance between the points $L,M \in \mathbb{E}^3$.

Let us fix a non-degenerate triangle $ABC$ in $\mathbb{E}^3$ with $d(A,B)=\overline{x}$, $d(B,C)=\overline{z}$,  and $d(A,C)=\overline{y}$.
For a tetrahedron $ABCD$ in $\mathbb{E}^3$ with the base $ABC$, we will use the following notation:
$d(D,A)=z$, $d(D,B)=y$, $d(D,C)=x$, $\angle ADB=\overline{\gamma}$, $\angle ADC=\overline{\beta}$, and $\angle BDC=\overline{\alpha}$
(see Fig. \ref{jfp00}).
It is easy to check that the numbers $x=d(D,C)$, $y=d(D,B)$, and  $z=d(D,A)$ characterize the point $D$,
up to the symmetry with respect to the plane $ABC$.
\smallskip

Let us consider the set $\Omega (\triangle ABC)$ of all tetrahedra $ABCD$ with a given non-degenerate base $ABC$ in $\mathbb{E}^3$ and $D$ lying outside the plane $ABC$.
Obviously, $\Omega (\triangle ABC)$ is naturally parameterized by points $D \in  \mathbb{E}^3 \setminus \Pi^0$, where $\Pi^0$ is the plane $ABC$.
Moreover, if $\Pi^{-}$ and $\Pi^{+}$ are two distinct half-spaces determined by the plane $\Pi^0$,
then it is sufficient to consider points $D$ in only one of these half-spaces.
\smallskip

The main problem studied in this paper is as follows:

\begin{problem} \label{prob.1}
What is the set
\begin{equation}\label{eq.sigma.1}
\Sigma(\triangle ABC):=\left\{\Bigl(\cos \overline{\alpha},\cos \overline{\beta},\cos \overline{\gamma} \Bigr)\in
\mathbb{R}^3\,|\, ABCD \in \Omega (\triangle ABC)\right\}\,,
\end{equation}
where  $\overline{\alpha}=\angle BDC$, $\overline{\beta}=\angle ADC$, and $\overline{\gamma}=\angle ADB$\,?
\end{problem}

Since the angles $\overline{\alpha}, \overline{\beta}, \overline{\gamma} \in (0,\pi)$, then they are completely determined by their cosines.
We take the cosines instead of the angles in \eqref{eq.sigma.1} only for technical reason (in order to simplify calculations).
\newpage

Let us recall some known results in this direction. It is well known that
\begin{equation}\label{eq.angle.1}
\overline{\alpha} <\overline{\beta} +\overline{\gamma}, \quad \overline{\beta}<\overline{\alpha}+\overline{\gamma}, \quad \overline{\gamma}<\overline{\alpha}+ \overline{\beta}, \quad \overline{\alpha}+\overline{\beta}+\overline{\gamma} <2\pi
\end{equation}
for all non-degenerate tetrahedra $ABCD$. Moreover, recently M.Q.~Rieck obtained the following result.

\begin{prop}[\cite{Rieck}]\label{prop.rieck}
For any non-degenerate tetrahedron $ABCD$ with an acute-angled base $ABC$, we have the following statements {\rm(}in the above notations{\rm)}:
\begin{enumerate}
\item $\angle BAC+\overline{\beta}+\overline{\gamma} <2\pi$,
\item $\overline{\alpha} \leq \angle BAC$ implies $\overline{\beta} \leq \max \{\angle ABC, \angle ACB +\overline{\alpha}\}$,
\item $\overline{\alpha} \leq \angle BAC$ implies $\cos \angle ACB \cdot \cos \overline{\beta} +\cos \angle ABC \cdot \cos \overline{\gamma} >0$,
\end{enumerate}
as well as all statements obtained from these three ones by simultaneous cyclic permutations
$\overline{\alpha} \to \overline{\beta} \to \overline{\gamma} \to \overline{\alpha}$ and $\angle BAC \to \angle ABC \to \angle ACB \to \angle BAC$.
\end{prop}
\medskip

It is interesting and helpful to compare Problem \ref{prob.1} with the following one:

\begin{problem} \label{prob.2}
What is the set
\begin{equation}\label{eq.gamma.1}
\Upsilon(\triangle ABC):=\left\{\Bigl(x,y,z \Bigr)\in \mathbb{R}^3\,|\, ABCD \in \Omega (\triangle ABC)\right\}\,,
\end{equation}
where $x=d(D,C)$, $y=d(D,B)$, and $z=d(D,A)$\,?
\end{problem}

The solution of this problem is well known:
$(x,y,z)\in \Upsilon(\triangle ABC)$ if and only if the Cayley--Menger determinant
\begin{equation}\label{eq.CM.1}
 \mathbb{D}=\begin{vmatrix} 0 & x^2 & y^2 & z^2 & 1\\
                                   x^2 & 0 & \overline{z}^2 & \overline{y}^2 & 1\\
                                   y^2 & \overline{z}^2 & 0 & \overline{x}^2 & 1\\
                                   z^2 & \overline{y}^2 & \overline{x}^2 & 0 & 1\\
                                   1 & 1 & 1 & 1 & 0
        \end{vmatrix}
\end{equation}
is positive.
Moreover, the following result is well known (see e.g. P.~163 in \cite{WirDr2009}):

\begin{prop}\label{pr.CM.1}
A sextuple $(x,y,z,\overline{x},\overline{y},\overline{z})$ of positive numbers are the edge lengths of some tetrahedra in $\mathbb{E}^3$
if and only if the following two conditions are fulfilled:
\begin{enumerate}
\item $\mathbb{D}$ is positive,
\item $\mathbb{D}_1$ is negative, where $\mathbb{D}_1$ is the $(1,1)$-minor of $\mathbb{D}${\rm :}
 $$
 \mathbb{D}_1=\begin{vmatrix}  0 & \overline{z}^2 & \overline{y}^2 & 1\\
                                       \overline{z}^2 & 0 & \overline{x}^2 & 1\\
                                       \overline{y}^2 & \overline{x}^2 & 0 & 1\\
                                        1 & 1 & 1 & 0
        \end{vmatrix}.
 $$
 \end{enumerate}
\end{prop}
Note that the condition (2) in Proposition \ref{pr.CM.1} is equivalent to the fact that the positive numbers
$\overline{x}$, $\overline{y}$, $\overline{z}$ are the side lengths of some non-degenerate triangle in $\mathbb{E}^2$.

Let us recall some classical formulas for the volume of a tetrahedron, which will be very useful to us.
If $V$ is the volume of the tetrahedron $ABCD$, then we have
\begin{eqnarray}\label{eq.vol.1}
V^2\!\!\!&=\!\!\! &\frac{1}{36} x^2y^2z^2 \Bigl(1 + 2 \cos \overline{\alpha}\cdot  \cos \overline{\beta}\cdot \cos \overline{\gamma}
- \cos^2 \overline{\alpha} - \cos^2 \overline{\beta} - \cos^2 \overline{\gamma} \Bigr)\\
\label{eq.vol.2}
\!\!\!&=\!\!\!&\frac{1}{9} x^2y^2z^2 \Bigl(\sin \frac{\overline{\alpha} + \overline{\beta}
+\overline{\gamma}}{2} \cdot \sin \frac{\overline{\alpha} + \overline{\beta} -\overline{\gamma}}{2} \cdot \sin \frac{\overline{\alpha}
+\overline{\gamma}-\overline{\beta} }{2}
\cdot \sin \frac{\overline{\beta} +\overline{\gamma}-\overline{\alpha}}{2}\Bigr),
\end{eqnarray}
see Problems IV and VI in Note 5 (pp. 249, 251) of \cite{Legendre}. Note also that $\mathbb{D} = 288V^2$ for the Cayley--Menger determinant
\eqref{eq.CM.1}, see, e.g., \cite{WirDr2014}.
\smallskip

In order to get a solution of Problem \ref{prob.1}, we are going to use the following idea:
We consider the closure of the set $\Sigma(\triangle ABC)$ in $\mathbb{R}^3$ and get an explicit
description of the boundary of this closure.
After this step, the solution of the original problem can already be obtained by standard methods.

The following result is the starting point for our further research.

\begin{theorem}\label{theo.1}
For any fixed base $ABC$, the closure of the set $\Sigma(\triangle ABC)$ is contained in the set
$$
\mathbb{P}=\left\{(x,y,z)\in [-1,1]^3\,|\,1+2xyz-x^2-y^2-z^2 \geq 0 \right\}.
$$
\end{theorem}

\begin{proof}
Since $\mathbb{P}$ is a closed subset in $\mathbb{R}^3$, it suffices to prove that for any tetrahedron $ABCD$ we have
\begin{equation*}
1 + 2 \cos \overline{\alpha} \cdot  \cos \overline{\beta} \cdot \cos \overline{\gamma}
- \cos^2 \overline{\alpha} - \cos^2 \overline{\beta} - \cos^2 \overline{\gamma} \geq 0,
\end{equation*}
where
$\cos \overline{\alpha}=\cos \angle BDC$, $\cos \overline{\beta}=\cos \angle ADC$, and $\cos \overline{\gamma}=\cos \angle  ADB$.
Now, it suffices to note that the inequality we need easily follows from \eqref{eq.vol.1}.
\end{proof}
\smallskip

As will be shown below, the closure of a set  $\Sigma(\triangle ABC)$ is a proper subset of $\mathbb{P}$
for any given base. However, the structure of this closure depends significantly on the chosen base.
In particular, there is a fundamental difference between acute-angled, rectangular, and obtuse-angled bases (see details in the last section).

The results of this paper could potentially be of practical use in areas such as
molecular geometry, tetrahedral finite-element meshes, camera-tracking, and the so-called ``perspective $3$-point problem'' \cite{HLON, Rieck15, RW, WMMS}.

\begin{figure}[t]
\begin{minipage}[h]{0.41\textwidth}
\center{\includegraphics[width=0.97\textwidth]{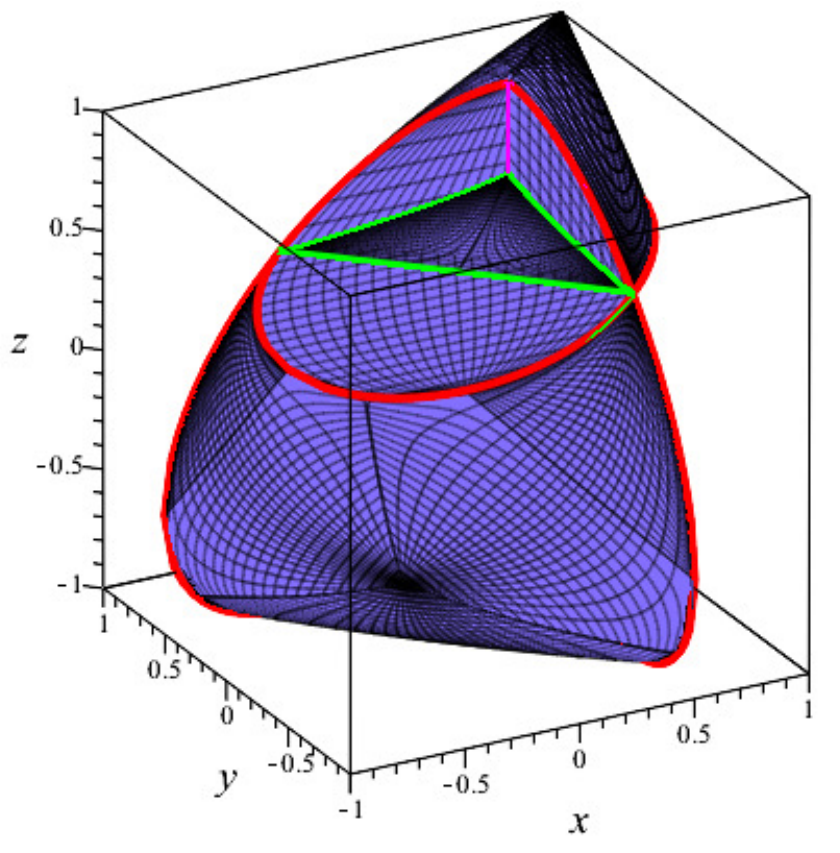}} \\ a)  \\
\end{minipage}
\quad
\begin{minipage}[h]{0.46\textwidth}
\center{\includegraphics[width=0.88\textwidth]{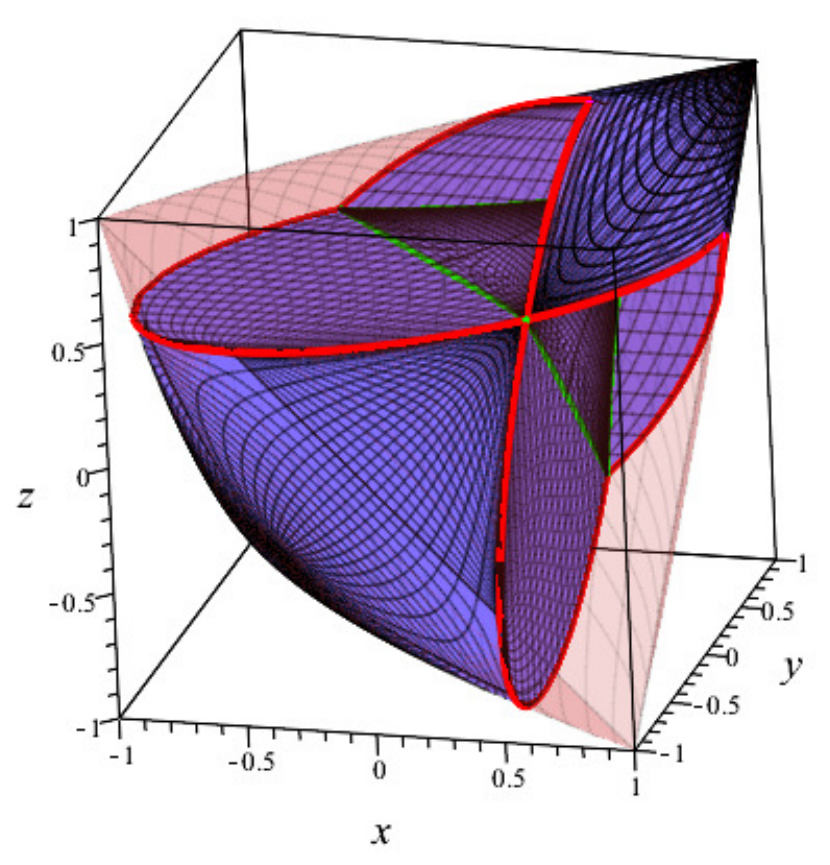}}   \\ b) \\
\end{minipage}
\begin{minipage}[h]{0.41\textwidth}
\center{\includegraphics[width=0.97\textwidth]{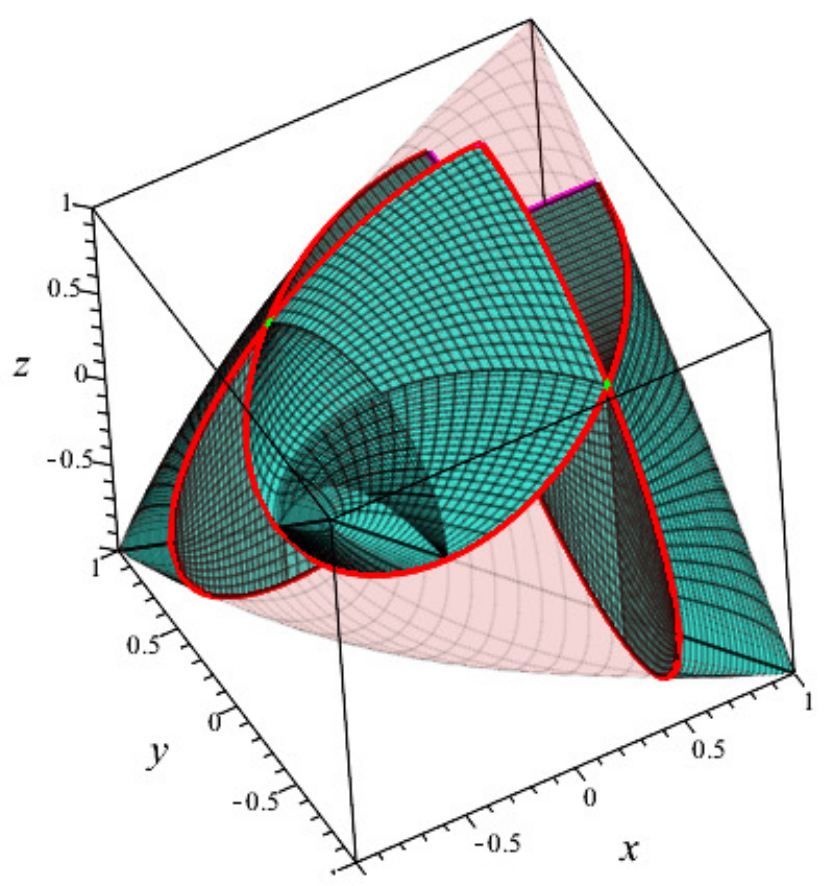}} \\ c)  \\
\end{minipage}
\quad
\begin{minipage}[h]{0.46\textwidth}
\center{\includegraphics[width=0.87\textwidth]{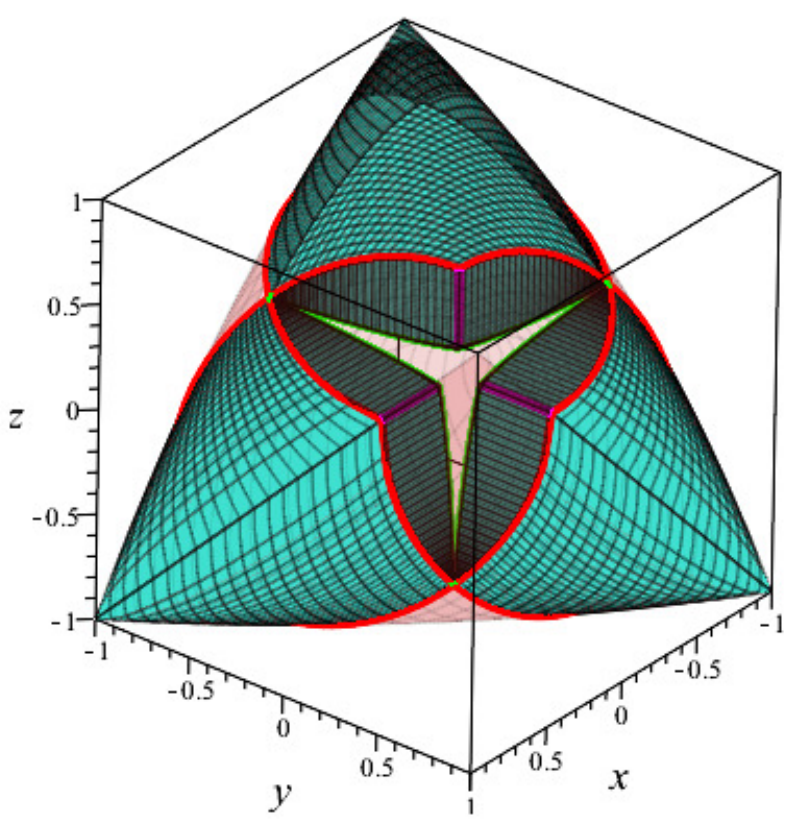}}   \\ d) \\
\end{minipage}
\caption{
The closure of the set $\Sigma(\triangle ABC)$ (in purple) and the closure of its complement $CS$  (in turquoise) in $\mathbb{P}$ (in orange)
for the case, when the base $ABC$ is a regular triangle,
for different points of view (angles):
a) the closure of $\Sigma(\triangle ABC)$, b) $\mathbb{P}$ and the closure of $\Sigma(\triangle ABC)$, c) and d) $\mathbb{P}$ and
the closure of $CS$. Red, green, and magenta curves show non-smoothness points of the boundary of $\Sigma(\triangle ABC)$
as well as of the boundary of the closure of $CS$.}
\label{firstimage}
\end{figure}

In our study, one subset  of $\mathbb{E}^3$ plays an important role,
namely, the right circular cylinder $\operatorname{Cyl}$ over the circle circumscribed around the base of $ABC$
(it suffices to consider only the part of this cylinder lying in $\Pi^{+}$).

The paper is organized as follows.
In Section \ref{sec.dist.coord}, we consider a formalization of Problem \ref{prob.1} in terms of distance coordinates.
Some applications of such formalization are obtained.

Another formalization of Problem \ref{prob.1} in terms of Cartesian coordinated is considered in
Section \ref{sec.cart.coord}. For some reasons, we prefer this formalization and it is used in all the following sections.
Note that in this case, $\Sigma(\triangle ABC)$ is the image of a special smooth map $F: \Pi^{+} \bigcup \Pi^{0}\setminus \{A,B,C\}\rightarrow \mathbb{R}^3$.

In Section \ref{sec.degenerate}, all degenerate points of the above function $F$ are found. They are exactly points of the  set
$\left(\operatorname{Cyl}\bigcap \Pi^{+}\right) \bigcup \Pi^{0}\setminus \{A,B,C\}$. This observation implies
Theorem~\ref{theo.nondeg}, which contains serious restriction on the boundary of the set $\Sigma(\triangle ABC)$.

Section \ref{sec.image.plane} is devoted to the study of the image of $\Pi^{0}\setminus \{A,B,C\}$ under the map $F$.
Note that $F(\Pi^{0}\setminus \{A,B,C\}) \subset \mathbb{BP}$, the boundary of $\mathbb{P}$ from Theorem \ref{theo.1}.

We note that the formulation of Problem \ref{prob.1} is preserved if we supplement the set of tetrahedra $\Omega(\triangle ABC)$
with degenerate tetrahedra with vertices $D$ in the plane $ABC$, but under the condition $D\not\in \{A,B,C \}$.
How to deal with cases when $D$ approaches one of the base vertices will be discussed in Section \ref{sec.limit.point}.
In particular, Proposition \ref{pr.closure.space.1} gives useful properties of the
closure of $\Sigma(\triangle ABC)$.

In Section \ref{sec.image.cyl}, some properties of the image of $\operatorname{Cyl}\bigcap \Pi^{+}$ under the map $F$ are studied.
It is proved, that such image is a smooth surface of positive Gaussian curvature, if we exclude three ``edges'' consisting of points of non-smoothness.

In Section \ref{sec.cyl.special}, we discuss the images of some special subsets of the cylinder $\operatorname{Cyl}$ and of $\Pi^{+}$.

A detailed answer to Problem \ref{prob.1} for an arbitrary base is given in the final section (Section \ref{sec.main}) of the paper.
In particular, this section contains
Theorem \ref{final.boundary}, which gives  a complete description of the boundary of the closure of $\Sigma(\triangle ABC)$
for any given base $ABC$.

In order to present the complexity of the original problem, we show in Fig. \ref{firstimage}  the closure of the set $\Sigma(\triangle ABC)$ (or its boundary)
for the case when the base $ABC$ is a regular triangle.

\smallskip

For the convenience of readers,
here we provide {\it the list of some important notations and definitions used in the paper}.

\begin{itemize}

\item
The set $\Omega (\triangle ABC)$ is introduced on Page \pageref{prob.1}.

\item
The set $\Sigma(\triangle ABC)$ is defined by \eqref{eq.sigma.1} on Page \pageref{eq.sigma.1}.

\item
The set $\mathbb{P}$ (the ``pillow'') is introduced on Page~\pageref{theo.1}.

\item
The right cylinder $\operatorname{Cyl}$ is introduced on Page \pageref{theo.1}.

\item
The set $\GT$ is defined
by \eqref{eq.GT} on Page \pageref{eq.GT}.

\item
The map $F:\GT \to \mathbb{R}^3$ is defined by \eqref{eq.fmain} on Page \pageref{eq.fmain}.

\item
The gradient of the function $F=(F_1,F_2,F_3)$ in the Cartesian coordinates
is defined by \eqref{eq.mapF.2} on Page \pageref{eq.mapF.2}.

\item
The boundary $\mathbb{BP}$ of the set $\mathbb{P}$ (the ``pillowcase'')
is defined by \eqref{eq.pillowc} on Page \pageref{eq.pillowc}.

\item
The definition of a {\it toroid} $\mathbb{T}_{KL}^{\,\alpha}$ is given on Page \pageref{eq.toroid.s.1}.

\item
The set $\DF$ is  introduced on Page \pageref{sec.degenerate}.

\item
The subsets
${\mathbb G_0}$, ${\mathbb G_1}$, ${\mathbb G_2}$, ${\mathbb G_3}$, ${\mathbb G_4}$,
${\mathbb G_5}$ and ${\mathbb G_6}$ of the plane are introduced on Page \pageref{break_plane}, see Fig.~\ref{break_plane}.

\item
The decomposition
$\Pi^0  = \mathbb{B}_{-}\bigcup S \bigcup \mathbb{B}_{+}$
is defined by \eqref{eq.decom.plane1} on Page \pageref{eq.decom.plane1}.

\item
The points $\widetilde{A}$, $\widetilde{B}$, and $\widetilde{C}$ of the ``pillowcase'' $\mathbb{BP}$
are defined by \eqref{eq.spec.points.1} on Page \pageref{eq.spec.points.1}.

\item
The definitions of the sets $\operatorname{Lim}(A)$, $\operatorname{Lim}(B)$, and $\operatorname{Lim}(C)$ are given on Page \pageref{le.limitpoint.1}.

\item
The closure $\CFT$ of $F(\GT)$  is introduced on Page \pageref{pr.closure.space.1}.

\item
The definition of the surface $\operatorname{FC}$
is given on Page \pageref{prob.inject.1}.

\item
The set $M_{\,\operatorname{reg}}$ of smooth points of the surface $F(\operatorname{Cyl}\bigcap \Pi^+)$ is introduced on Page \pageref{pro.nondegen.1}.

\item
The definition of a {\it special toroid} is given on Page \pageref{subs.8.1}.

\item
The definition of a {\it special region} of the cylinder $\operatorname{Cyl}$ is given on Page \pageref{def.spec.reg}.

\item
The set $\operatorname{BFT}$  is introduced on Page \pageref{final.boundary}.

\item
The set  $\operatorname{SRT}$ is defined by \eqref{eq.main} on Page \pageref{eq.main}.

\end{itemize}

\smallskip

The authors are grateful to Endre Makai, Jr., Michael Rieck, and Sergey Sadov for useful
discussions and valuable advices.

\section{Using distance coordinates}\label{sec.dist.coord}

Let us recall that the base $\triangle ABC$ is assumed fixed.
In particular, its side lengths $d(A,B)=\overline{x}$, $d(B,C)=\overline{z}$, $d(A,C)=\overline{y}$ are fixed.
Then we can parameterize the set of tetrahedra $ABCD$ with this base by the triples
of positive numbers $(x,y,z)$, where $x=d(D,C)$, $y=d(D,B)$, $z=d(D,A)$, such that the Cayley--Menger determinant~$\mathbb{D}$~\eqref{eq.CM.1} is positive
(note that the $(1,1)$-minor of $\mathbb{D}$ is negative by our assumptions and we can use Proposition \ref{pr.CM.1}).

All such triples constitute some set $\DT$ in $\mathbb{R}^3$ with coordinates $x,y,z$.
These are so-called the distance coordinates.
\smallskip

The condition $\mathbb{D}>(\,\geq)\, 0$ is as follows:
\begin{eqnarray}\label{eq.elpar.1}\notag
\mathbb{D}/2=-\overline{x}^2u^2-\overline{y}^2v^2-\overline{z}^2w^2
+(\overline{x}^2+\overline{y}^2-\overline{z}^2)uv
+(\overline{x}^2-\overline{y}^2+\overline{z}^2)uw\\
+(\overline{y}^2+\overline{z}^2-\overline{x}^2)vw
+(\overline{x}^2\overline{y}^2+\overline{x}^2\overline{z}^2-\overline{x}^4)u
+(\overline{x}^2\overline{y}^2+\overline{y}^2\overline{z}^2-\overline{y}^4)v\\
+(\overline{x}^2\overline{z}^2+\overline{y}^2\overline{z}^2-\overline{z}^4)w
-\overline{x}^2\overline{y}^2\overline{z}^2>(\,\geq)\, 0\,,\notag
\end{eqnarray}
where $u:=x^2$, $v:=y^2$, $w=z^2$.

Let us consider the matrix $\mathbb{\overline{A}}$ associated to the quadric $\mathbb{D}=\mathbb{D}(u,v,w)$:
$$
\mathbb{\overline{A}}=\left(
  \begin{array}{cccc}
    -2\overline{x}^2 & \overline{x}^2+\overline{y}^2-\overline{z}^2 & \overline{x}^2+\overline{z}^2-\overline{y}^2 &  (\overline{y}^2+\overline{z}^2-\overline{x}^2)\overline{x}^2 \\
    \overline{x}^2+\overline{y}^2-\overline{z}^2 & -2\overline{y}^2 & \overline{y}^2+\overline{z}^2-\overline{x}^2 & (\overline{x}^2+\overline{z}^2-\overline{y}^2)\overline{y}^2 \\
    \overline{x}^2+\overline{z}^2-\overline{y}^2 &  \overline{y}^2+\overline{z}^2-\overline{x}^2 &  -2\overline{z}^2 &  (\overline{x}^2+\overline{y}^2-\overline{z}^2)\overline{z}^2 \\
     (\overline{y}^2+\overline{z}^2-\overline{x}^2)\overline{x}^2 &  (\overline{x}^2+\overline{z}^2-\overline{y}^2)\overline{y}^2 &  (\overline{x}^2+\overline{y}^2-\overline{z}^2)\overline{z}^2 & -2\overline{x}^2\overline{y}^2\overline{z}^2 \\
  \end{array}
\right).
$$
By direct computations we get that the determinant of $\mathbb{\overline{A}}$ is negative, indeed
$$
\det(\mathbb{\overline{A}})=-(\overline{x}+\overline{y}+\overline{z})^3(\overline{x}+\overline{y}-\overline{z})^3
(\overline{x}+\overline{z}-\overline{y})^3(\overline{y}+\overline{z}-\overline{x})^3<0.
$$
It is easy to check that $(4,4)$-minor of $\mathbb{\overline{A}}$ (i.e. the determinant of the corresponding
quadratic form) is zero.
Therefore, the equality $\mathbb{D}=0$ defines some elliptic paraboloid, whereas
the inequality $\mathbb{D}\geq 0$ describes the convex hull of this elliptic paraboloid.

If, say, $w=0$, then
\begin{eqnarray*}\label{eq.elpar.2}\notag
\mathbb{D}/2=-\overline{x}^2u^2-\overline{y}^2v^2
+(\overline{x}^2+\overline{y}^2-\overline{z}^2)uv\\
+(\overline{x}^2\overline{y}^2+\overline{x}^2\overline{z}^2-\overline{x}^4)u
+(\overline{x}^2\overline{y}^2+\overline{y}^2\overline{z}^2-\overline{y}^4)v
-\overline{x}^2\overline{y}^2\overline{z}^2>(\,\geq)\, 0\notag
\end{eqnarray*}

It is easy to check that the intersection of  elliptic paraboloid \eqref{eq.elpar.1} with the plane $w=0$ consists of one point, namely,
the point $(u=\overline{y}^2,v=\overline{x}^2,w=0)$.
Indeed, the inequality~\eqref{eq.elpar.1} can be rewritten as follows:

$$
\overline{x}^2\bigl(u-\overline{y}^2\bigr)^2- \bigl(\overline{x}^2+\overline{y}^2-\overline{z}^2\bigr)\bigl(u-\overline{y}^2\bigr)\bigl(v-\overline{x}^2\bigr) +\overline{y}^2\bigl(v-\overline{x}^2\bigr)^2
<(\,\leq)\, 0\,.
$$
On the other hand,
$$
\bigl(\overline{x}^2 + \overline{y}^2 - \overline{z}^2\bigr)^2 - 4\overline{x}^2\overline{y}^2=-\bigl(\overline{x} + \overline{y} + \overline{z}\bigr)
\bigl(\overline{y} - \overline{z} -\overline{x}\bigr) \bigl(\overline{x} + \overline{z}- \overline{y}) \bigl(\overline{x}+\overline{y} -\overline{z}\bigr) <0,
$$
since we have the triangle inequality for the triangle $ABC$. This proves that the intersection of
elliptic paraboloid \eqref{eq.elpar.1} with the plane $w=0$ consists of the point $\overline{C}:=(\overline{y}^2,\overline{x}^2,0)$.

Similar arguments prove that the intersections of elliptic paraboloid \eqref{eq.elpar.1} with the planes $u=0$ and $v=0$
consist of the points $\overline{A}:=(0,\overline{z}^2,\overline{y}^2)$ and $\overline{B}:=(\overline{z}^2,0,\overline{x}^2)$ respectively.
We note that the points $\overline{A}$, $\overline{B}$, and $\overline{C}$ represent respectively the points $A$, $B$, and $C$ in this special parametrization.
\smallskip

Let us define the map
$\overline{F}=(\overline{F}_1,\overline{F}_2,\overline{F}_3):\DT \rightarrow \mathbb{R}^3$
as follows:
\begin{eqnarray}\label{eq.mapF.1}\notag
\overline{F}_1(x,y,z)&=&\frac{x^2+y^2-\overline{z}^2}{2xy}=\cos \angle BDC= \cos \overline{\alpha},\\
\overline{F}_2(x,y,z)&=&\frac{x^2+z^2-\overline{y}^2}{2xz}=\cos \angle CDA = \cos \overline{\beta}, \\ \notag
\overline{F}_3(x,y,z)&=&\frac{y^2+z^2-\overline{x}^2}{2yz}=\cos \angle ADB =\cos \overline{\gamma}.
\end{eqnarray}

We are going to prove the following simple result.

\begin{lemma}\label{le.infin.1}
Let us consider tetrahedra $ABCD$ with a fixed base $ABC$, where
$d(A,B)=\overline{x}$, $d(B,C)=\overline{z}$, $d(A,C)=\overline{y}$, $d(D,A)=z$, $d(D,B)=y$, $d(D,C)=x$.
Then $\overline{F}(x,y,z) \to (1,1,1)$ if and only if $\min \{x,y,z\} \to \infty$ or, equivalently,  $\max \{x,y,z\} \to \infty$.
\end{lemma}

\begin{proof} We note that $M \to \infty$ is equivalent to $m \to \infty$ due to
the inequality $M\geq m \geq M - M_b$ (it easily follows from the triangle inequality),
where $m=\min \{x,y,z\}$, $M_b=\max \{\overline{x}, \overline{y},\overline{z}\}$, and $M=\max \{x,y,z\}$.

If $m=\min \{x,y,z\} \to \infty$,
then $x \,(y,\,z) \geq m \geq M - M_b$.
Let us consider only tetrahedra with $M\geq  M_b$. Then we have
\begin{eqnarray*}
\cos \angle ADB =\frac{y^2+z^2-\overline{x}^2}{2yz}\geq \frac{(M - M_b)^2+(M - M_b)^2-M_b^2}{2yz}\\
\geq \frac{2(M - M_b)^2-M_b^2}{2M^2}=1-\frac{4MM_b-M_b^2}{2M^2} \to 1
\end{eqnarray*}
when $M \to \infty$ (which is equivalent to $m \to \infty$).
Similar arguments prove that $\cos \angle BDC \to 1$ and  $\cos \angle CDA \to 1$ when $m \to \infty$.

Conversely, if we have a sequence of points $\{D_n\}$, such
that  $\cos \angle AD_nB \to 1$, $\cos \angle BD_nC \to 1$, and  $\cos \angle CD_nA \to 1$ when $n \to \infty$,
then
$$
M_n=\max\{d(D_n,A),d(D_n,B),d(D_n,C)\} \to \infty \quad \mbox{  as }n \to \infty.
$$
Otherwise, passing to a subsequence (if necessary) we
may suppose that all $M_n$ are less than some constant and $D_n \to D'\in \mathbb{R}^3$ as $n \to \infty$.
On the other hand, this implies $\angle AD'B=\angle BD'C= \angle CD'A=0$, that is impossible since the triangle $ABC$ is supposed to be non-degenerate.
This proves the lemma.
\end{proof}
\medskip

Although the use of distance coordinates leads to some useful results related to Problem \ref{prob.1},
we will use the more familiar and somewhat more convenient Cartesian coordinates for our study.

\section{Cartesian coordinates and some general observations}\label{sec.cart.coord}

Let us consider another possibility to parameterize the set of tetrahedra $ABCD$ with the base $ABC$ by the triples
$(p,q,r)\in \mathbb{R}^3$, that are suitable Cartesian coordinates of the point $D$.
In this case, the map $F=(F_1,F_2,F_3)=(\cos \overline{\alpha},\cos \overline{\beta},  \cos \overline{\gamma})$ can be consider as map from $\mathbb{R}^3$ to itself.

\smallskip

Let us consider a Cartesian system  in $\mathbb{R}^3$ such that
$A=(a_1,a_2,0)$, $B=(b_1,b_2,0)$, $C=(c_1,c_2,0)$, and $D=(p,q,r)$.
It is clear that the points $D=(p,q,r)$ and $D=(p,q,-r)$ give us two isometric
tetrahedra with a common base $ABC$. Therefore, we may restrict our attention to the case $r\geq 0$.

To relate this parametrization with the previous one, we note that
\begin{eqnarray}\label{eq.twocoord.1}\notag
\overline{x}^2=(a_1-b_1)^2+(a_2-b_2)^2, \qquad\qquad\,\, \overline{y}^2=(a_1-c_1)^2+(a_2-c_2)^2, \\
\overline{z}^2=(c_1-b_1)^2+(c_2-b_2)^2, \quad\, u=x^2=(p-c_1)^2+(q-c_2)^2+r^2, \\ \notag
v=y^2=(p-b_1)^2+(q-b_2)^2+r^2, \quad w=z^2=(p-a_1)^2+(q-a_2)^2+r^2.
\end{eqnarray}
\smallskip

In what follows, we will use the following notations:
\begin{equation}\label{eq.Pi+Pi0}
\Pi^+=\{(p,q,r) \in \mathbb{R}^3\,|\, r>0\}, \quad \Pi^0=\{(p,q,0) \in \mathbb{R}^3\},
\end{equation}
\begin{equation}\label{eq.GT}
\GT=\Pi^+ \bigcup \bigl(\Pi^0 \setminus \{A,B,C\} \bigr).
\end{equation}

The latter set is the set of all tetrahedra with the base $ABC$
(we consider only the case when $D \not\in \{A,B,C\}$). We can also formally define the map $F:\GT \to \mathbb{R}^3$, using \eqref{eq.twocoord.1} and \eqref{eq.mapF.1}.
Note, that
\begin{equation}\label{eq.fmain}
F(p,q,r)=\bigl(\cos \overline{\alpha}=\cos \angle BDC, \cos \overline{\beta}=\cos \angle CDA, \cos \overline{\gamma}= \cos \angle ADB \bigr),
\end{equation}
where $D=(p,q,r)\in  \GT=\Pi^+ \bigcup \bigl(\Pi^0 \setminus \{A,B,C\} \bigr)\subset \mathbb{R}^3\setminus \{A,B,C\}$.

Taking into account Lemma \ref{le.infin.1}, we can add to the set $\GT$ the {\it infinity point} $\infty$.
Then $\Pi^+ \bigcup \Pi^0 \bigcup \{\infty\}$ is homeomorphic to a closed $3$-dimensional ball in the topology of the one-point compactification,
where $\Pi^0 \bigcup \{\infty\}$ is the boundary of this ball (a sphere).

\smallskip

We will use the symbol
$\overrightarrow{\nabla} G$ for the gradient of the function $G$ in the Cartesian coordinates, i.e.,
$\overrightarrow{\nabla} G=\left(\frac{\partial{G}}{\partial p}, \frac{\partial{G}}{\partial q}, \frac{\partial{G}}{\partial r}\right)$.
It is easy to see (taking into account \eqref{eq.mapF.1}  and~\eqref{eq.twocoord.1}) that
\begin{eqnarray}\label{eq.mapF.2}\notag
\overrightarrow{\nabla} F_1 = \overrightarrow{\nabla} \cos \overline{\alpha}=
\frac{x^2-y^2-\overline{z}^2}{2y^3x}\cdot\overrightarrow{DB}
+\frac{y^2-x^2-\overline{z}^2}{2x^3y}\cdot\overrightarrow{DC},\\
\overrightarrow{\nabla} F_2 = \overrightarrow{\nabla} \cos \overline{\beta}=
\frac{x^2-z^2-\overline{y}^2}{2z^3x}\cdot\overrightarrow{DA}
+\frac{z^2-x^2-\overline{y}^2}{2x^3z}\cdot\overrightarrow{DC},\\ \notag
\overrightarrow{\nabla} F_3 = \overrightarrow{\nabla} \cos \overline{\gamma}=
\frac{y^2-z^2-\overline{x}^2}{2z^3y}\cdot\overrightarrow{DA}
+\frac{z^2-y^2-\overline{x}^2}{2y^3z}\cdot\overrightarrow{DB}.
\end{eqnarray}

These formulas will be used in what follows.

\begin{remark}\label{re.nond.1}
The triangle $ABC$ is assumed to be non-degenerate,
therefore the vectors $\overrightarrow{DA}$, $\overrightarrow{DB}$, and $\overrightarrow{DC}$ are linearly independent
for all points $D=(p,q,r)$ with $r\neq 0$. On the other hand, they are linearly dependent in the case when $r=0$, i.e. $D$ is in the plane $ABC$ and the tetrahedron $ABCD$
is degenerate.
\end{remark}

Now we are going to describe some properties of the image $F(\GT)$ under the map $F:\GT \to \mathbb{R}^3$, see \eqref{eq.fmain}.
The main objects of our study are the closure $\CFT$ of $F(\GT)$ and the boundary $\BFT$ of $\CFT$ in the standard topology of $\mathbb{R}^3$.
\smallskip

Along with the set $\mathbb{P}$ from Theorem~\ref{theo.1} (the ``pillow''), it is useful for us to consider its boundary
\begin{equation}\label{eq.pillowc}
\mathbb{BP}=\left\{(x,y,z)\in [-1,1]^3\,|\,1+2xyz-x^2-y^2-z^2 = 0 \right\}
\end{equation}
(the ``pillowcase'').
\smallskip

By \eqref{eq.vol.1}, we get that $F(p,q,r) \in \mathbb{P}$ for every $(p,q,r) \in \GT$.

\begin{remark} It should be noted that for $(x,y,z) \in F(\GT)$, the preimage set $F^{-1}\bigl((x,y,z)\bigr)$ of the point  $(x,y,z)$
has at least $1$ and at most $4$ elements. This result is well known and it was obtained within the framework of solving the so-called
``the perspective $3$-point problem'',
see e.g., \cite{WMMS} and \cite{HLON}. In these publications one can find many examples where points have one, two, three or four preimages.
The proof that more than four
preimages are impossible comes down to the fact that each of the possible preimages
corresponds to some root of a special fourth-degree polynomial.
\end{remark}

If we fix one of the variables in \eqref{eq.pillowc},
then we obtain the equation of some ellipse, which lies in the plane parallel to the coordinate plane of two other variables.

It is possible to find a simple parametrization of  $\mathbb{BP}$.

\begin{lemma}\label{le.param.pilc}
The surface $\mathbb{BP}$ is the image of the map
\begin{equation}\label{eq.pilc.param}
(\varphi,\psi) \mapsto \bigl(\cos(\varphi - \psi), \cos (\psi), \cos (\varphi)\bigr), \quad \varphi,\psi \in \mathbb{R}/2\pi\mathbb{Z}.
\end{equation}
\end{lemma}

\begin{proof}
Direct computations show that the map \eqref{eq.pilc.param} is smooth and it is non-degenerate at all points excepting $\varphi,\psi \in \pi \mathbb{Z}$.
It is easy to verify that if some pairs $(\varphi_1,\psi_1)$ and $(\varphi_2,\psi_2)$ have one and the same image, then either
$\varphi_2-\varphi_1, \psi_2-\psi_1 \in 2\pi \mathbb{Z}$, or $\varphi_2+\varphi_1, \psi_2+\psi_1 \in 2\pi\mathbb{Z}$.
Moreover, it is easy to check by direct calculations that
$\bigl(\cos(\varphi - \psi), \cos (\psi), \cos (\varphi)\bigr) \in \mathbb{BP}$ for any $\varphi,\psi \in \mathbb{R}$.

If we identify every point $(\varphi,\psi)$ with $(-\varphi,-\psi)$, then the torus $(\mathbb{R}/2\pi\mathbb{Z})^2$ becomes a sphere.
Therefore, the image of map \eqref{eq.pilc.param} is an injective image of a sphere and is a subset of $\mathbb{BP}$.
Consequently, this image is exactly $\mathbb{BP}$.
\end{proof}

\begin{prop}\label{pr.prop.pc}
The surface $\mathbb{BP}$ has the following properties:

{\rm 1)}  $\mathbb{BP}$ is invariant under every permutation of the variables $x,y,z$ and under the map $(x,y,z)\mapsto (\pm x, \pm y, \pm z)$,
where there are even number of the sign ``$-$'' in the right side of the formula.

{\rm 2)} The  intersection of $\mathbb{BP}$ with the plane $z=1$ {\rm(}$z=-1${\rm)} is the closed interval with the endpoints $(-1,-1,1)$ and $(1,1,1)$
{\rm(}the closet interval with the endpoints $(-1,1,-1)$ and $(1,-1,-1)$, respectively{\rm)}.

{\rm 3)} The  surface $\mathbb{BP}$ contains all edges of the tetrahedron with the vertices
 $(1,1,1)$, $(1,-1,-1)$, $(-1,1,-1)$, and $(-1,-1,1)$.

{\rm 4)} If $c=\operatorname{const} \in (-1,1)$, then the  intersection of $\mathbb{BP}$ with the plane $z=c$ is an ellipse, which is invariant under the maps
 $(x,y)\mapsto (y,x)$ and $(x,y)\mapsto (-x,-y)$; in particular this ellipse is the unit circle for $c=0$.

{\rm 5)}  The surface $\mathbb{BP}$ is convex, it is smooth {\rm(}with positive Gaussian curvature{\rm)} at all points except the points
$(1,1,1)$, $(1,-1,-1)$, $(-1,1,-1)$, $(-1,-1,1)$ {\rm(}at all points except the edges of the tetrahedron with the vertices
 $(1,1,1)$, $(1,-1,-1)$, $(-1,1,-1)$, and $(-1,-1,1)${\rm)}.

\end{prop}

\begin{proof}
The first property is obvious.
The second property follows from the equalities $(x-y)^2=0$ {\rm(}$(x+y)^2=0${\rm)} for $z=1$ {\rm(}$z=-1$, respectively{\rm)}.
The third property is follows easily from the previous two.
It is clear that the equality $x^2-2cxy+y^2=1-c^2$  determines an ellipse with suitable symmetries, which proves the fourth property.
The last property can be proved easily by the standard analytical methods.
\end{proof}

\medskip

For given points $K,L \in \mathbb{R}^3$ and a given $\alpha \in [0,\pi]$, we consider {\it a toroid}
\begin{equation}\label{eq.toroid.s.1}
\mathbb{T}_{KL}^{\,\alpha}=\left\{M\in {\mathbb R}^3\,|\,\angle KML=\alpha \right\}
\end{equation}
(this concept was introduced in \cite{Rieck}, where it was successfully and effectively used).
It is clear that $\mathbb{T}_{KL}^{\,\alpha}$ is a surface of revolution with axis $KL$.

Note that $K,L \not\in \mathbb{T}_{KL}^{\,\alpha}$.
If $K=(k_1,k_2,k_3)$ and $L=(l_1,l_2,l_3)$ then the point $M=(p,q,r)$ is in $\mathbb{T}_{KL}^{\,\alpha}$ if and only if
and
\begin{equation}\label{eq.toroid.s.2}
\frac{(p-k_1)(p-l_1)+(q-k_2)(q-l_2)+(r-k_3)(r-l_3)}{\sqrt{(p-k_1)^2+(q-k_2)^2+(r-k_3)^2}\cdot \sqrt{(p-l_1)^2+(q-l_2)^2+(r-l_3)^2}}=\cos (\alpha).
\end{equation}
Therefore, $\mathbb{T}_{KL}^{\,\alpha}$ is a subset of a suitable algebraic surface of degree $4$. If, say, $K=(-t,0,0)$ and $L=(t,0,0)$ for some $t>0$,
then this surface is described by the equation
$$
\bigl(p^2+q^2+r^2-t^2 \bigr)^2= \cos^2 (\alpha) \cdot \Bigl(\bigl(p^2+q^2+r^2+t^2 \bigr)^2-4t^2p^2\Bigr).
$$

If $\Sigma$ is any plane in $\mathbb{R}^3$ such that $K,L \in \Sigma$,
then $T_{KL,\,\Sigma}^{\,\alpha}=\mathbb{T}_{KL}^{\,\alpha} \bigcap \Sigma$ consists of two congruent circular arcs.

It is clear that there is a tetrahedron $ABCD$ with a given base $ABC$ and with given angles
$(\angle BDC,\angle CDA,\angle ADB)=(\overline{\alpha},\overline{\beta},\overline{\gamma})$
if and only if
$$
\mathbb{T}_{BC}^{\,\overline{\alpha}} \bigcap \mathbb{T}_{AC}^{\,\overline{\beta}} \bigcap \mathbb{T}_{AB}^{\,\overline{\gamma}} \neq \emptyset.
$$
The study of toroids and their properties is very important for our main problem, see some discussion in \cite{Rieck}.

\section{Degenerate points of $F$}\label{sec.degenerate}

We are going to find the set $\DF$ of points $D \in \GT$, for which
the map $F$ is degenerate, i.e. the linear span of the vectors $\overrightarrow{\nabla} F_1$, $\overrightarrow{\nabla} F_2$, and $\overrightarrow{\nabla} F_3$
(see \eqref{eq.mapF.2}) is not $3$-dimensional.

\begin{prop}\label{pr.nondeg.main}
If a point $D\in \mathbb{R}^3\setminus \{A,B,C\}$ is degenerate for the map $F$, then either $D \in  \Pi^0$ or $D \in  \operatorname{Cyl}$,
where $\Pi^0$ is the plane $ABC$ and $\operatorname{Cyl}$ is the right cylinder
over the circle circumscribed around the triangle $ABC$.
\end{prop}

The remaining part of this section is devoted to the proof of this proposition.
The idea of paying special attention to degenerate points of the mapping is natural
(compare, for example, with ``the danger cylinder'' \cite{HLON} or with Lemma 4.2 in \cite{Rieck}).

Using Remark \ref{re.nond.1}, we get that all $D$ from the plane $ABC$, $F$ is degenerate, hence, $D \in \DF$.
Now we consider the case when $D$ is not in the plane $ABC$. In this case,
the vectors $\overrightarrow{DA}$, $\overrightarrow{DB}$, and $\overrightarrow{DC}$ are linearly independent, hence $D\in \DF$ if and only if
the determinant $|JF|$ of a $(3\times3)$-matrix, which rows are the coordinates
of the vectors $\overrightarrow{\nabla} F_3$, $\overrightarrow{\nabla} F_2$, $\overrightarrow{\nabla} F_1$ in the basis
$\left(\overrightarrow{DC}, \overrightarrow{DB},\overrightarrow{DA}\right)$, is $0$.

Using \eqref{eq.mapF.2}, we get the following computation of $|JF|$:
\begin{eqnarray*}
|JF|&=&\det \left(
            \begin{array}{c}
              \nabla F_3 \\
              \nabla F_2 \\
              \nabla F_1 \\
            \end{array}
          \right)=-\frac{1}{8}\det\left(
                             \begin{array}{ccc}
                               0 & \frac{y^2-z^2+\overline{x}^2}{y^3z} & \frac{z^2-y^2+\overline{x}^2}{yz^3} \\
                               \frac{x^2-z^2+\overline{y}^2}{x^3z} & 0 & \frac{z^2-x^2+\overline{y}^2}{xz^3} \\
                               \frac{x^2-y^2+\overline{z}^2}{x^3y} & \frac{y^2-x^2+\overline{z}^2}{xy^3} & 0 \\
                             \end{array}
                           \right)\\
&&=-\frac{1}{8x^4y^4z^4}\det\left(
                             \begin{array}{ccc}
                               0 & y^2-z^2+\overline{x}^2 & z^2-y^2+\overline{x}^2 \\
                               x^2-z^2+\overline{y}^2 & 0 & z^2-x^2+\overline{y}^2 \\
                               x^2-y^2+\overline{z}^2 & y^2-x^2+\overline{z}^2 & 0 \\
                             \end{array}
                           \right)\\
&&=-\frac{1}{8x^4y^4z^4}((\overline{x}^2+ z^2-y^2)(\overline{y}^2+x^2-z^2)(\overline{z}^2+y^2-x^2)\\
&&+(\overline{x}^2+ y^2-z^2)(\overline{y}^2+z^2-x^2)(\overline{z}^2+x^2-y^2))\\
&&=-\frac{1}{4x^4y^4z^4}\Bigl(\overline{x}^2(x^2-z^2)(y^2-x^2)+ \overline{y}^2(z^2-y^2)(y^2-x^2)\Bigr. \\
&&+\Bigl. \overline{z}^2(z^2-y^2)(x^2-z^2)+\overline{x}^2\overline{y}^2\overline{z}^2\Bigr).
\end{eqnarray*}

It is easy to show what the condition $|JF|=0$ is equivalent to the condition
\begin{equation}\label{eq.tetr.for1}
\overline{x}^2(x^2-z^2)(y^2-x^2)+ \overline{y}^2(z^2-y^2)(y^2-x^2)+
\overline{z}^2(z^2-y^2)(x^2-z^2)+\overline{x}^2\overline{y}^2\overline{z}^2=0,
\end{equation}
which one can rewrite as follows:
\begin{equation}\label{eq.tetr.for2n}
\frac{x^2-z^2}{\overline{y}^2} \cdot \frac{y^2-x^2}{\overline{z}^2}+\frac{z^2-y^2}{\overline{x}^2} \cdot \frac{y^2-x^2}{\overline{z}^2}+
\frac{z^2-y^2}{\overline{x}^2}\cdot  \frac{x^2-z^2}{\overline{y}^2}+1=0.
\end{equation}

\begin{prop}\label{prop.deg.1}
The equality \eqref{eq.tetr.for1} holds for a point $D=(p,q,r)\in \mathbb{R}^3$ if and only if $D$ lies in the right cylinder $\operatorname{Cyl}$
over the circle circumscribed around the triangle $ABC$.
\end{prop}

Since the condition \eqref{eq.tetr.for1} is invariant under  substitutions of the form $(x,y,z) \mapsto \left(\sqrt{x^2+l},\sqrt{y^2+l},\sqrt{z^2+l}\right)$
for any constant $l\geq \max\{-x^2,-y^2,-z^2\}$, it
is equivalent both for a point $D$ and its orthogonal projection $\overline{D}$ onto the plane $ABC$.
Hence, it suffices to prove the following result.

\begin{lemma}\label{le.tetr.1}
If the point $D$ is in the plane $ABC$, then the equality \eqref{eq.tetr.for1} holds if and only if the quadrilateral $ABCD$ is cyclic,
i.e., all its vertices lie on a circle.
\end{lemma}

\begin{proof} Let us consider some Cartesian coordinate system in the plane $ABC$. Suppose that $A=(a_1,a_2)$, $B=(b_1,b_2)$, $C=(c_1,c_2)$, $D=(t,s)$.
Then it is easy to check that the left side of  \eqref{eq.tetr.for1} with the negative sign is as follows:
$$
(2C_1t-C_2)^2+(2C_1s-C_3)^2-C_4,
$$
where
\begin{eqnarray*}
C_1&=&a_1b_2-a_1c_2-a_2b_1+a_2c_1+b_1c_2-b_2c_1,\\
C_2&=&(a_1^2+a_2^2)(b_2-c_2) + (b_1^2+b_2^2)(c_2-a_2) + (c_1^2+c_2^2)(a_2-b_2),\\
C_3&=&(a_1^2+a_2^2)(c_1-b_1) + (b_1^2+b_2^2)(a_1-c_1) + (c_1^2+c_2^2)(b_1-a_1),\\
C_4&=& \bigl((b_1-c_1)^2+(b_2-c_2)^2\bigr)\bigl((a_1-c_1)^2+(a_2-c_2)^2\bigr)\bigl((b_1-a_1)^2+(b_2-a_2)^2\bigr).
\end{eqnarray*}
Therefore, the set given by \eqref{eq.tetr.for1} in the plane of $ABC$ is the circle circumscribed around a triangle $ABC$ (it is easy to check that the points
$A=(a_1,a_2)$, $B=(b_1,b_2)$, $C=(c_1,c_2)$ satisfy~\eqref{eq.tetr.for1}).
Note also that $|C_1|=\left| (b_1 - a_1)(c_2 - a_2) - (b_2 - a_2)(c_1 - a_1)\right|$
is equal to twice the area~$S$ of the triangle $ABC$
and $\dfrac{\sqrt{C_4}}{2|C_1|}= \dfrac{\overline{x}\cdot\overline{y}\cdot\overline{z}}{4S}$ is the radius of the circumscribed circle.
\end{proof}

\begin{remark}\label{re.curcucumsribed}
Note that Lemma \ref{le.tetr.1} gives a criterion for the quadrilateral $ABCD$ to be inscribed in a circle.
It should be noted that several similar results are known. The most famous is Ptolemy's theorem.
A review of the relevant results can be found in \cite{Jos2019.1, Jos2019.2, RaAj2003, Sad2004, MS}.
\end{remark}

\begin{figure}[t]
\center{\includegraphics[width=0.4\textwidth]{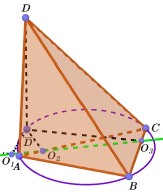}}\\
\caption{A tetrahedron and Menelaus's theorem.}
\label{jfp0}
\end{figure}
\medskip

Note, that Lemma \ref{le.tetr.1} immediately implies Proposition \ref{prop.deg.1}, which in turn implies Proposition \ref{pr.nondeg.main}.
\smallskip

Below we will consider another proof of Proposition \ref{prop.deg.1} based on the use of the classical approach.

\begin{proof}[Second proof of Proposition \ref{prop.deg.1}]
One can also consider a proof Proposition \ref{prop.deg.1} based on the theorems of Menelaus and Simpson.
Indeed, since
$$
\overline{x}^2+ z^2-y^2=2\overline{x}\cdot z \cdot\cos \angle DAB,~~ \overline{y}^2+ x^2-z^2=2\overline{y}\cdot x \cdot\cos \angle DCA,
$$
$$
\overline{z}^2+ y^2-x^2=2\overline{z}\cdot y \cdot\cos \angle DBC,~~ \overline{x}^2+ y^2-z^2=2\overline{x}\cdot y \cdot\cos \angle DBA,
$$
$$
\overline{y}^2+ z^2-x^2=2\overline{y}\cdot z \cdot\cos \angle DAC,~~
\overline{z}^2+ x^2-y^2=2\overline{z}\cdot x \cdot\cos \angle DCB,
$$
we get that
\begin{eqnarray*}
\overline{x}^2(x^2-z^2)(y^2-x^2)+ \overline{y}^2(z^2-y^2)(y^2-x^2)+
\overline{z}^2(z^2-y^2)(x^2-z^2)+\overline{x}^2\overline{y}^2\overline{z}^2
\\
=-4 xyz \overline{x}~\overline{y}~\overline{z} \\
\times \Bigl(\cos \angle DAB \cdot\cos \angle DCA\cdot\cos \angle DBC+
\cos \angle DBA\cdot\cos \angle DAC\cdot\cos \angle DCB\Bigr)=0.
\end{eqnarray*}

Let us recall some notation. Let  $\overrightarrow{KL}$ and $\overrightarrow{MN}$ be collinear vectors. Denote by $\frac{\overline{KL}}{\overline{MN}}$
the quantity $\pm\frac{d(K,L)}{d(M,N)}$, where the plus
sign is taken if the vectors  $\overrightarrow{KL}$ and $\overrightarrow{MN}$ are directed in the same way, and the minus sign if the vectors are
directed opposite to each other.

The orthogonal projection of the point $D$ onto the plane $ABC$ will be denoted by $\overline{D}$. Let us  consider
points $O_1, O_2, O_3$ on the straight lines $AB, AC, BC$ respectively, such that $\overline{D}O_1\bot AB$,
$\overline{D}O_2\bot AC$ and $\overline{D}O_3\bot BC$ (see Fig.~\ref{jfp0}).
Since $DO_1\bot AB$, $DO_2\bot AC$ and $DO_3\bot BC$ we have (in the above notation) the following equalities:
$$
\frac{\cos \angle DAB}{\cos \angle DBA}=-\frac{y}{z} \cdot \frac{\overline{AO_1}}{\overline{BO_1}}, \quad
\frac{\cos \angle DCA}{\cos \angle DAC}= -\frac{z}{x} \cdot \frac{\overline{CO_2}}{\overline{AO_2}}, \quad
\frac{\cos \angle DBC}{\cos \angle DCB}= -\frac{x}{y} \cdot \frac{\overline{BO_3}}{\overline{CO_3}}.
$$
Hence, $|JF|=0$ if and only if
$$
\frac{\overline{AO_1}}{\overline{BO_1}} \cdot \frac{\overline{CO_2}}{\overline{AO_2}}\cdot \frac{\overline{BO_3}}{\overline{CO_3}} =1,
$$
which implies that the points $O_1, O_2, O_3$ are collinear by Menelaus's theorem (see, e.g., Theorem 218 in \cite{Johnson} or Problem 5.69 in \cite{Pra2006}).
Therefore, the point $\overline{D}$ lies on the circle circumscribed around the triangle $ABC$ by Simpson's theorem
(see, e.g., Theorem 192 in \cite{Johnson} or Problem 5.105 in \cite{Pra2006}).
\end{proof}
\medskip

 \begin{figure}[t]
\center{\includegraphics[width=0.4\textwidth]{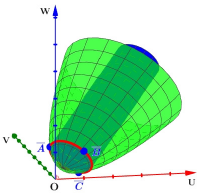}}\\
\caption{The description of suitable tetrahedra using distance coordinates.}
\label{jfp}
\end{figure}

\begin{remark}\label{re.dist.plane}
Note, that \eqref{eq.tetr.for1} obviously implies the following equality:
\begin{eqnarray}\label{eq.cyl.1}\notag
-\overline{x}^2u^2-\overline{y}^2v^2-\overline{z}^2w^2+(\overline{x}^2+\overline{y}^2-\overline{z}^2)uv+
(\overline{x}^2-\overline{y}^2+\overline{z}^2)uw&&\\
+(\overline{y}^2+\overline{z}^2-\overline{x}^2)vw+\overline{x}^2\overline{y}^2\overline{z}^2&=&0\,,
\end{eqnarray}
where $u:=x^2$, $v:=y^2$, $w=z^2$.

It is easy to see that the intersection of cylinder \eqref{eq.cyl.1} with the elliptic paraboloid \eqref{eq.elpar.1} lies in the plane
\begin{equation}\label{eq.plane.1}
\bigl(\overline{y}^2+\overline{z}^2-\overline{x}^2\bigr)\overline{x}^2\,u + \bigl(\overline{x}^2+\overline{z}^2-\overline{y}^2\bigr)\overline{y}^2\,v
+ \bigl(\overline{x}^2 + \overline{y}^2 - \overline{z}^2\bigr)\overline{z}^2\,w - 2\overline{x}^2\overline{y}^2\overline{z}^2=0.
\end{equation}
In particular, this intersection is an ellipse in the corresponding distance coordinates, see Fig.~\ref{jfp}.
\end{remark}
\smallskip

The inverse function theorem and Proposition \ref{pr.nondeg.main} imply the following important result.

\begin{theorem}\label{theo.nondeg}
The map $F:\GT \to \mathbb{R}^3$ is non-degenerate at every point $(p,q,r) \in \Pi^+ \setminus \operatorname{Cyl}$.
In particular, the map $F$ is locally bijective on the set $\Pi^+ \setminus \operatorname{Cyl}$
and, consequently,  $F\bigl(\Pi^+ \setminus \operatorname{Cyl} \bigr)$ is an open subset of full measure in the closure $\CFT$ of $F(\GT)$.
\end{theorem}

\section{Properties of the map $F$ on the plane $ABC$}\label{sec.image.plane}

In this section we study the images of  points $D \in \Pi^0\setminus \{A,B,C\}$ under the map $F$. At first we consider one simple result.
Recall that
$F(D)\in \mathbb{BP}$ means that
$$
1+2\cos \overline{\alpha}\cdot \cos \overline{\beta} \cdot \cos \overline{\gamma} - \cos^2 \overline{\alpha} - \cos^2 \overline{\beta} - \cos^2 \overline{\gamma}=0,
$$
where $\overline{\alpha}=\angle BDC$, $\overline{\beta}= \angle ADC$, $\overline{\gamma}= \angle  ADB$.

\begin{lemma}\label{le.inpl.1}
If $D \in \Pi^0\setminus \{A,B,C\}$, then $F(D)\in \mathbb{BP}$. As a consequence, we have $F \Bigl( \Pi^0\setminus \{A,B,C\}\Bigr )\subset \mathbb{BP}$.
\end{lemma}

\begin{proof}
Since  $D \in \Pi^0\setminus \{A,B,C\}$, then the volume of the tetrahedron $ABCD$ is zero. Hence,
$1+2\cos \overline{\alpha}\cdot \cos \overline{\beta}\cdot \cos \overline{\gamma} - \cos^2 \overline{\alpha} - \cos^2 \overline{\beta} - \cos^2 \overline{\gamma}=0$ by \eqref{eq.vol.1}.
\end{proof}

\begin{remark}
There are more simple arguments to prove Lemma \ref{le.inpl.1}.
We have one of the following possibilities for points $D\in \Pi^0\setminus \{A,B,C\}$:
\begin{eqnarray*}
\angle ADB=\angle BDC +\angle CDA,&&
\angle BDC =\angle CDA+\angle ADB,\\
\angle CDA=\angle BDC +\angle ADB,&&
\angle ADB+\angle BDC +\angle CDA =2\pi.
\end{eqnarray*}
It is easy to see that $F(D)=( \cos\angle BDC, \cos\angle CDA,\cos\angle ADB)\in \mathbb{BP}$ in all these cases.
Indeed, if, say $\angle BDC=\angle CDA +\angle ADB$, then
$x=yz-\sqrt{(1-y^2)(1-z^2)}$, where
$x=\cos\angle BDC$, $y=\cos\angle CDA$, and $z= \cos\angle ADB$.
Then we get $(x-yz)^2=(1-y^2)(1-z^2)$, which gives the equality for
$\mathbb{BP}$.
\end{remark}

Let us consider the decomposition
\begin{equation}\label{eq.decom.plane1}
\Pi^0  = \mathbb{B}_{-}\bigcup S \bigcup \mathbb{B}_{+},
\end{equation}
where $S$ is the circumscribed circle for the triangle $ABC$, $\mathbb{B}_{-}$ is the open disc bounded by $S$, and $\mathbb{B}_{+}$ is the set of points outside $S$.
\medskip

\begin{figure}[t]
\center{\includegraphics[width=0.6\textwidth]{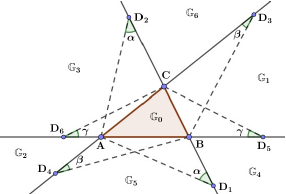}}\\
\caption{Important subsets of the plane $ABC$.}
\label{break_plane}
\end{figure}

Recall that the map $F$ is not defined at the points $A,B,C \in S$.
Let us consider the following important points in the ``pillowcase'' $\mathbb{BP}$ (see \eqref{eq.pillowc}):
\begin{eqnarray}\label{eq.spec.points.1} \notag
\widetilde{A}&:=&\Bigl(-\cos \angle BAC, \cos \angle ABC, \cos \angle ACB\Bigl),\\
\widetilde{B}&:=&\Bigl(\cos \angle BAC, -\cos \angle ABC, \cos \angle ACB\Bigl),\\\notag
\widetilde{C}&:=&\Bigl(\cos \angle BAC, \cos \angle ABC, -\cos \angle ACB\Bigl).
\end{eqnarray}

\begin{lemma}\label{le.map.plane.1}
The map $F$ is locally constant and takes only three values on the set \linebreak $S \setminus  \{A,B,C\}$.
Moreover,  if $D$ is from the arc $BC$,  $AC$, or $AB$, then $F(D)$ is equal to $\widetilde{A}$, $\widetilde{B}$, or $\widetilde{C}$, respectively.
\end{lemma}

\begin{proof}
It is easy to see that the map $F$ is locally constant of the set
$S \setminus  \{A,B,C\}$. Indeed, $F(D)=\bigl(\cos\angle BDC,\cos\angle ADC,\cos \angle ADB\bigl)$ by the definition of $F$.
If $D$ is on the arc $AB$ of $S$, then $F(D)=(\cos \angle BAC, \cos \angle ABC, -\cos \angle ACB)=\widetilde{C}$
due to $\angle BDC=\angle BAC$, $\angle ADC= \angle ABC$, and $\angle ADB=\pi- \angle ACB$.
If $D$ is on the arc $BC$ of $S$, then $F(D)=(-\cos \angle BAC, \cos \angle ABC, \cos \angle ACB)=\widetilde{A}$ (due to $\angle BDC=\pi- \angle BAC$, $\angle ADC= \angle ABC$, and $\angle ADB=\angle ACB$).
Finally, if $D$ is on the arc $AC$ of $S$, then $F(D)=(\cos \angle BAC, -\cos \angle ABC, \cos \angle ACB)=\widetilde{B}$.
\end{proof}
\smallskip

Moreover, we have the following result.

\begin{prop}\label{pr.map.plane.1}
The map $F$ is locally constant and takes only three values on the set $S \setminus  \{A,B,C\}$.
Moreover, the map $F$ is injective on the set $\mathbb{B}_{-}$ and is not injective on the set $\mathbb{B}_{+}$.
\end{prop}

\begin{proof}
The first assertion follows from Lemma \ref{le.map.plane.1}.

Let us suppose that there are $D_1, D_2 \in \mathbb{B}_{-}$ such that $D_1 \neq D_2$ and $F(D_1)=F(D_2)$.
The segment $[D_1,D_2]$ cannot intersect all sides of the triangle $ABC$. Suppose that it does not intersect the side $BC$.
Then $[D_1,D_2]$ should intersect two other sides of the triangle $ABC$. Indeed, if it does not intersect, say the side $AB$,
then both points $D_1, D_2$ are on a circle $\mathcal{C}_1$ with the chord $AB$ as well as are on a circle $\mathcal{C}_2$ with the chord $BC$.
Hence, $D_1,D_2, B \in \mathcal{C}_1 \bigcap \mathcal{C}_2$ that impossible (an intersection of two circles has at most two points).
Now, we may suppose without loss of generality that $D_1$ and $C$ are in different half-planes with respect the straight line $AB$ and $[D_1,D_2]$ intersect the sides $AB$ and $AC$. Hence, we have a convex pentagon $AD_1BCD_2$, which implies $\angle AD_1C<\angle AD_1B$ and $\angle AD_2C > \angle AD_2B$.
But the condition $F(D_1)=F(D_2)$ implies $\angle AD_1C= \angle AD_2C$ and $\angle AD_1B=\angle AD_2B$, which is incompatible with the previous inequalities.
Hence, $F$ is injective on the set $\mathbb{B}_{-}$.

The fact that the map $F$ is not injective on the set $\mathbb{B}_{+}$, follows from the examples below, see e.g., \eqref{eq.noninj.plane.1} and Fig. \ref{break_plane}.
\end{proof}

\smallskip

Recall that the map $F$ is not defined at the points $A$, $B$, and $C$, moreover, there are no (direction-independent) limits for $F(D)$, when $D$ tends to one of these points,
see details in Section \ref{sec.limit.point}.

\medskip

It is easy to see that the straight lines $AB, AC$ and $BC$ divide the plane of $ABC$
into $7$ (closed) subsets ${\mathbb G_0}$, ${\mathbb G_1}$, ${\mathbb G_2}$, ${\mathbb G_3}$, ${\mathbb G_4}$,
${\mathbb G_5}$ and ${\mathbb G_6}$, see Fig. \ref{break_plane}.

\begin{figure}[t]
\begin{minipage}[h]{0.395\textwidth}
\center{\includegraphics[width=0.95\textwidth]{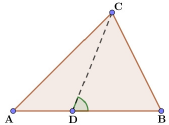}} \\ a)                     \\
\end{minipage}
\quad\quad\quad\quad
\begin{minipage}[h]{0.395\textwidth}
\center{\includegraphics[width=0.95\textwidth]{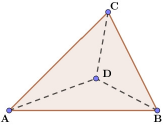}}   \\ b) \\
\end{minipage}
\caption{a) Point $D\in[A,B]$; b) Point $D$ lies inside $\triangle ABC$.}
\label{Base}
\end{figure}

\bigskip

At first, we consider the case where the point $D$ lies inside $[A,B]$, see
the left panel of Fig.~\ref{Base}. Since $\angle ADB = \angle ADC + \angle CDB=\pi$,
therefore, $\cos\angle ADB=-1$, $\cos \angle BDC =- \cos \angle ADC$.
Note also that $\angle CDB > \angle CAB$ and  $\angle CDA > \angle CBA$ that implies
$$
-\cos\angle CBA< -\cos\angle ADC=\cos\angle  BDC < \cos\angle CAB.
$$
Therefore, the image of the interval $(A,B)$ under the map $F$ is the interval between the points
$(\cos \angle CAB, -\cos \angle CAB,-1)$ and
$(-\cos \angle ABC, \cos \angle ABC,-1)$, which is
a subset of the interval $I_1$ with the endpoints $(-1,1,-1)$ и $(1,-1,-1)$.

The same idea implies that the image of the interval $(B,C)$ under the map $F$ \linebreak
is
the interval between the endpoints
$(-1,\cos \angle ACB, -\cos \angle ACB)$ and \linebreak
$(-1,-\cos \angle ABC, \cos \angle ABC)$, which is
a subset of the interval $I_2$ with the endpoints $(-1,-1,1)$ и $(-1,1,-1)$.

Finally, similar arguments show that the image of the interval $(A,C)$ under the map $F$ is the interval between the point
$(-\cos \angle ABC,-1, \cos \angle ABC)$ and the point \linebreak
$(\cos \angle CAB, -1, -\cos \angle CAB)$, which is
a subset of the interval $I_3$ with the endpoints $(1,-1,-1)$ и $(-1,-1,1)$.

\begin{remark}\label{re.edge.pc}
The intervals  $I_1$, $I_2$, and $I_3$ can be considered as some ``edges'' of the ``pillowcase'' $\mathbb{BP}$.
\end{remark}
\smallskip

Further, let us consider the case where the point $D$ lies in the interior the triangle $ABC$,
see the right panel of Fig.~\ref{Base}. It implies
$\angle ADB + \angle BDC + \angle CDA = 2\pi$ and
$$
\angle BDC > \angle BAC,\quad \angle CDA > \angle CBA,\quad \angle ADB > \angle ACB.
$$
Thus only quite special triples $(\cos \overline{\alpha}, \cos \overline{\beta},\cos \overline{\gamma})$ can be realized as the images of such points~$D$.
Indeed, $(\cos \overline{\alpha}, \cos \overline{\beta},\cos \overline{\gamma})$ lies in the intersection of the following three open half-spaces:
\begin{eqnarray*}
\cos \overline{\alpha} &=&\cos \angle BDC< \cos \angle BAC,\\
\cos \overline{\beta}&=&\cos \angle CDA < \cos \angle CBA, \\
\cos \overline{\gamma}&=&\cos \angle ADB< \cos \angle ACB.
\end{eqnarray*}

\begin{remark}
It is easy to see that for a given base $ABC$, not every point of $\mathbb{BP}$ is the image $F(D)$ for some $D \in \Pi^0$.
Indeed, if we take numbers $\overline{\alpha}, \overline{\beta}, \overline{\gamma} \in [0,\pi]$ with $\overline{\alpha}+\overline{\beta}+\overline{\gamma}=2\pi$,
then, as it is easy to check, $(\cos \overline{\alpha}, \cos \overline{\beta},\cos \overline{\gamma}) \in \mathbb{BP}$.
On the other hand, if we have $D \in \Pi^0$ such that  $(\angle BDC, \angle ADC,\angle ADB)=(\overline{\alpha}, \overline{\beta}, \overline{\gamma})$,
then $D$ should be situated in the triangle $ABC$. Therefore, $\overline{\alpha}=\angle BDC \geq \angle BAC$, $\overline{\beta}=\angle ADC \geq \angle ABC$, and $\overline{\gamma}=\angle ADB \geq \angle ACB$ by simple arguments. Consequently, we have the following necessary condition: If
$(\cos \overline{\alpha}, \cos \overline{\beta},\cos \overline{\gamma})\in F(\Pi^0\setminus\{A,B,C\})$ and $\overline{\alpha}+\overline{\beta}+\overline{\gamma}=2\pi$, then
$\overline{\alpha} \geq \angle BAC$, $\overline{\beta} \geq \angle ABC$, and $\overline{\gamma} \geq \angle ACB$.
A bit modified arguments show that not every point in the interior of $\mathbb{P}$ is the image of some point in $\Pi^+$ under the map $F$
(for a given base).
\end{remark}

\medskip
Recall that {\it an angle on the plane is an intersection of two closed half-spaces with respect two non-parallel straight lines}.
This is a convex set bounded by two rays issued from one point.
\smallskip

Let $\mathbb{G}_1$ be the set that is the intersection of the angular domain $CAB$ with
a half-plane relative to $[C,B]$, not containing point $A$,
and the set $\mathbb{G}_2$ is opposite to the angular domain $CAB$, see Fig. \ref{break_pl}.
\smallskip

If $D_1, D_2\in BC\backslash [B,C]$ such that the interval $[B,C]$ lies between the points $D_1$ and $D_2$ and $\angle AD_1D_2=\angle AD_2D_1=:\alpha$, then
$\angle BD_1A=\angle CD_1A=\angle CD_2A=\angle BD_2A=\alpha$.
It implies
\begin{eqnarray*}
F(D_1)=\left(\cos \angle BD_1C,\cos \angle CD_1A,\cos\angle AD_1B,\right)=
\left(1,\cos\alpha,\cos\alpha\right)\\
=\left(\cos \angle BD_2C,\cos \angle CD_2A,\cos\angle AD_2B\right)=F(D_2).
\end{eqnarray*}
It is clear that we can apply this construction for any
$\alpha\in\left(0,\min\{\angle ACB,\angle ABC\}\right)$.
For example, if the triangle $ABC$ is regular, then
$\alpha\in(0,\pi/3)$ and $\cos \alpha\in(1/2,1)$.

Reasoning similarly, we obtain also that $F(D_3)=F(D_4)$ and $F(D_5)=F(D_6)$ for
$\beta\in\left(0,\min\{\angle BAC,\angle BCA\}\right)$ and
$\gamma \in\left(0,\min\{\angle CAB,\angle CBA\}\right)$, see Fig.~\ref{break_plane}.
\smallskip

\begin{figure}[t]
\center{\includegraphics[width=0.6\textwidth]{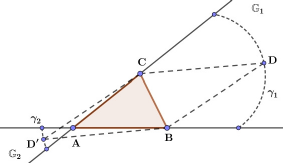}}\\
\caption{Intersection of the toroid $\mathbb{T}_{CB}^{\alpha}$ with the plane $ABC$}
\label{break_pl}
\end{figure}

For a fixed $\alpha\in(0,\pi)$, let us consider $T_{CB}^{\alpha}=\{D\in \Pi^0\,|\, \angle CDB=\alpha\}$,
the intersection of the toroid $\mathbb{T}_{CB}^{\alpha}$ with the plane $ABC$.
Then we define the following curves:
$$
\gamma_1:=T_{CB}^{\alpha}\cap\mathbb{G}_1,\quad   \gamma_2:=T_{CB}^{\alpha}\cap\mathbb{G}_2.
$$
It is easy to check the following observation: if $\alpha < \min\{\angle ACB,\angle ABC\}$,
then the curve $\gamma_1$ has its endpoints on the ray $AC$ and on the ray $AB$;
$\angle CDA$ can take any value in $[0,\alpha]$ for points $D\in\gamma_1$; and $\angle BDA = \alpha-\angle CDA$,  see Fig.~\ref{break_pl}.
Therefore, any triple
$(\alpha,\alpha_1, \alpha_2)$ such that $\alpha_1\geq 0$, $\alpha_2\geq 0$, and $\alpha=\alpha_1 + \alpha_2$,
can be realized as
$\left(\angle BDC,\angle CDA,\angle BDA\right)$ for some $D\in\gamma_1$.

It is easy to see that for each $\alpha\geq\angle BAC$ the curve $\gamma_2$ is an empty set.
However, if $\alpha<\angle BAC$ then $\gamma_2$ is a non-empty curve and for any point $D'\in \gamma_2$ we have
$F(D')\in F(\gamma_1)$, since the triple $(\alpha=\angle CD'B,\angle CD'A,\angle BD'A)$ can be realized as above.
It should be noted that this inclusion is not necessarily true if $\min\{\angle ACB,\angle ABC\}\leq \alpha<\angle BAC$ (in this case $\gamma_1$ has a non-trivial intersection with $\triangle ABC$).

Now, let us consider the following six sets:
\begin{eqnarray*}\label{eq.plane.ss1}
{\mathbb G_1^0\,(\mathbb G_2^0)}&:=&\{D\in \mathbb G_1\,(\mathbb G_2)\,|\, \angle BDC \in \left(0,\min\{\angle ACB,\angle ABC\}\right)\},\\
{\mathbb G_3^0\,(\mathbb G_4^0)}&:=&\{D\in \mathbb G_3\,(\mathbb G_4)\,|\, \angle ADC \in \left(0,\min\{\angle BAC,\angle BCA\}\right)\},\\
{\mathbb G_5^0\,(\mathbb G_6^0)}&:=&\{D\in \mathbb G_5\,(\mathbb G_6)\,|\, \angle ADB \in \left(0,\min\{\angle CAB,\angle CBA\}\right)\}.
\end{eqnarray*}

The above arguments, applied to all three sides of the base $ABC$, imply the following inclusions:
\begin{equation}\label{eq.noninj.plane.1}
F({\mathbb G_2^0})\subset F({\mathbb G_1^0}), \quad  F({\mathbb G_4^0})\subset F({\mathbb G_3^0}), \quad F({\mathbb G_6^0})\subset F({\mathbb G_5^0}).
\end{equation}

\begin{remark}\label{re.base.reg.1}
This observation implies in particular, that
it is enough to study the images $F({\mathbb G_1})$, $F({\mathbb G_3})$, and $F({\mathbb G_5})$ in the case of regular (equilateral) base $ABC$.
\end{remark}

If we consider the plane $ABC$ together with the infinite point $\infty$, then we get $F(\infty)=(1,1,1)$ by Lemma \ref{le.infin.1}.

The rays along the sidelines issued from the infinite point $\infty$ and passing through the points  $D_i$, $1\leq i \leq 6$, are glued as follows (see Fig.~\ref{break_plane}):
$$
[\infty,D_1)\leftrightarrow[\infty,D_2),\quad [\infty,D_3)\leftrightarrow[\infty,D_4), \quad [\infty,D_5)\leftrightarrow[\infty,D_6).
$$

Nevertheless, the following result holds.

\begin{prop}\label{prop7}
The map $F$ is injective in the interior of any of the angular domains $ABC$, $BCA$, $CAB$.
\end{prop}

\begin{proof}
It suffices to consider only the angular domain $ABC$ ($\mathbb{G}_0 \bigcup \mathbb{G}_3$). Let us suppose that there are
$D_1 \neq D_2$ in this  angular domain such that $F(D_1)=F(D_2)$. Then $D_1$ and $D_2$ lie on a circle $C_1$ with chord $[AB]$ ($\angle BD_1A =\angle BD_2A$).
By similar reasons $D_1$ and $D_2$ lie on a circle $C_2$ with chord $[BC]$. Therefore, $D_1,D_2 \in C_1 \bigcap C_2$. On the other hand, $B \in  C_1 \bigcap C_2$.
Since the intersection of two distinct circles has at most two points, we get the proposition.
\end{proof}

\smallskip

The following result shows that the case of obtuse-angled triangle $ABC$ is very special.

\begin{prop}\label{pr.tops.inter}
The condition $F(\mathbb{B}_{-}) \bigcap F(\mathbb{B}_{+})\neq \emptyset$ holds if and only if the triangle $ABC$ is obtuse-angled.
\end{prop}

 \begin{figure}[t]
\center{\includegraphics[width=0.4\textwidth]{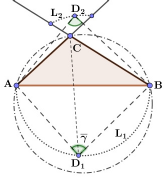}}\\
\caption{The intersection of $F(\mathbb{B}_{-})$ with $F(\mathbb{B}_{+})$}
\label{plane_tor1}
\end{figure}

\begin{proof} Let us suppose that
$$
F(D_1)=(\angle BD_1C, \angle AD_1C, \angle AD_1B)=(\angle BD_2C, \angle AD_2C, \angle AD_2B)=F(D_2)
$$
for some $D_1\in \mathbb{B}_{-}$ and $D_2\in \mathbb{B}_{+}$.
We can assume without loss of generality that
$\angle AD_2B=\angle AD_2C+ \angle CD_2B$, which is equivalent to $\angle AD_1B=\angle AD_1C+ \angle CD_1B$, see Fig. \ref{plane_tor1}.
Therefore, $D_1, D_2 \in {\mathbb G_5}\bigcup {\mathbb G_6}$, moreover,  $D_1 \in {\mathbb G_5} $ and by Proposition \ref{prop7},
$D_2 \in {\mathbb G_6}$, see Fig. \ref{break_plane}.
Consequently, $D_2$ is in the angular domain that is opposite to the angular domain $ACB$.
This implies $\angle AD_2B< \angle ACB$, hence,  $\angle ACB>\angle AD_2B=\angle AD_1B > \pi- \angle ACB$, that implies $\angle ACB> \pi/2$.

Now, we suppose that $\angle ACB> \pi/2$. We fix $\overline{\gamma}\in (\pi-\angle ACB=\angle CAB +\angle CBA, \angle ACB)$ and
$\overline{\alpha}\in (\angle CAB, \overline{\gamma}-\angle CBA)$,
then  $\overline{\beta}:=\overline{\gamma}-\overline{\alpha}\in (\angle CBA, \overline{\gamma}-\angle CAB)$.

Let us consider $T_{AB}^{\overline{\gamma}}=\{D\in \Pi^0\,|\, \angle ADB=\overline{\gamma}\}$,
the intersection of the toroid $\mathbb{T}_{AB}^{\overline{\gamma}}$ with the plane $\Pi^0$.
$T_{AB}^{\overline{\gamma}}$ contains two circular arcs $L_1$ and $L_2$,
where $L_1\setminus \{A,B\} \subset \mathbb{B}_{-}$ and  $L_2$ is a part of $T_{AB}^{\overline{\gamma}}$ that is contained in the
angular domain $ACB$ on the plane $\Pi^0$.
For points $D_1 \in L_1$ the pair $(\angle AD_1C, \angle CD_1B)$, where  $\angle AD_1C+\angle CD_1B=\overline{\gamma}$, takes all values between $(\overline{\gamma}-\angle CAB,\angle CAB)$ and $(\angle CBA,\overline{\gamma}-\angle CBA)$. Therefore, there is $D_1 \in L_1$ such that
$(\angle AD_1C, \angle CD_1B)=(\overline{\beta},\overline{\alpha})$.
Similarly, for points $D_2 \in L_2$ the pair $(\angle AD_2C, \angle CD_2B)$, where  $\angle AD_2C+\angle CD_2B=\overline{\gamma}$, takes all values
of the type $(\overline{\gamma}-\theta,\theta)$, where  $\theta\in [0,\overline{\gamma}]$. Therefore, there is $D_2 \in L_2$ such that
$(\angle AD_2C, \angle CD_2B)=(\overline{\beta},\overline{\alpha})$.
Therefore, we have found  $D_1\in \mathbb{B}_{-}$ and $D_2\in \mathbb{B}_{+}$, such that $F(D_1)=F(D_2)$. The proposition is proved.
\end{proof}

\begin{remark}
We see from the above proof that for any point $(\cos \overline{\alpha}, \cos \overline{\beta}, \cos \overline{\gamma})\in \mathbb{BP}$,
where $\overline{\gamma}\in  (\pi-\angle ACB, \angle ACB)$, $\overline{\alpha}\in (\angle CAB, \overline{\gamma}-\angle CBA)$, and $\overline{\beta}=\overline{\gamma}-\overline{\alpha}\in (\angle CBA, \overline{\gamma}-\angle CAB)$ (for the case when $\angle ACB> \pi/2$),
has one pre-image in  $\mathbb{B}_{-}$ and at least one pre-image in $\mathbb{B}_{+}$.
\end{remark}

The final result in this section is the following helpful proposition.

\begin{prop}\label{pr.top.cap}
Let us consider $\mathfrak{m}:=\max \{\cos \angle BAC, \cos \angle ABC, \cos \angle BCA \}$ for a given base $ABC$.
If $(1,1,1)\neq (\tilde{x},\tilde{y},\tilde{z})\in \mathbb{BP}$ and the inequality $\min \{\tilde{x},\tilde{y},\tilde{z}\} > \mathfrak{m}$ holds, then
there exists $D\in \Pi^0 \setminus \{A,B,C\}$ such that $F(D)=(\tilde{x},\tilde{y},\tilde{z})$.
\end{prop}

 \begin{figure}[t]
\center{\includegraphics[width=0.25\textwidth]{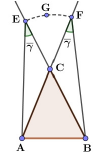}}\\
\caption{Intersection of the toroid $\mathbb{T}_{AB}^{\,\overline{\gamma}}$ with a special angular domain on the plane $ABC$}
\label{plane_tor}
\end{figure}

\begin{proof}
It is clear that $\mathfrak{m}\geq 1/2=\cos (\pi/3)$. Without loss of generality we assume that $\tilde{x} \geq \tilde{y} \geq \tilde{z} >\mathfrak{m} \geq 1/2$.  Let us consider the numbers
$\overline{\alpha}:=\arccos (\tilde{x})$, $\overline{\beta}:=\arccos (\tilde{y})$, and $\overline{\gamma}:=\arccos (\tilde{z})$.
Obviously, $0\leq \overline{\alpha} \leq \overline{\beta} \leq \overline{\gamma} < \pi/3$.
Using Lemma \ref{le.param.pilc}, we can suppose that
$$
\tilde{x}=\cos(\phi - \psi)\geq \tilde{y}=\cos (\psi) \geq \tilde{z}=\cos (\phi)>1/2
$$
for some $\psi,\phi \in \mathbb{R}$. Without loss of generality, we assume that $\phi \in (0,\pi/3)$. Then $|\phi-\psi|\leq |\psi|\leq \phi <\pi/3$.
This implies that $\psi >0$. Indeed, if $\psi \leq 0$ then $|\phi-\psi|=\phi-\psi \geq \phi$, therefore, $|\phi-\psi|=|\psi|=\phi$ and $\psi=0$ that imply $\phi=0$. Now, $\psi >0$ implies $\psi \leq \phi$ and $\phi-\psi\geq 0$. Therefore,
$$
\overline{\alpha}+\overline{\beta}=\arccos(\tilde{x})+\arccos( \tilde{y})=(\phi-\psi)+\psi=\phi=\arccos (\tilde{z})=\overline{\gamma}.
$$

Now, let us find a point $D\in \Pi^0 \setminus \{A,B,C\}$ such that
$$
F(D)=(\cos \angle BDC, \cos \angle ADC, \cos \angle BDA)=(\tilde{x},\tilde{y},\tilde{z}).
$$

On the plane $\Pi^0$, we consider the angular domain, defined by the rays $AC$ and $BC$ and its intersection with the toroid
$\mathbb{T}_{AB}^{\,\overline{\gamma}}$. It is clear that this intersection is an arc $L$ of a suitable circle (with the chord $AB$) between
some point $E$ on the ray $BC$ and some point $F$ on the ray $AC$ (recall that $\angle BCA>\arccos(\mathfrak{m}) >\overline{\gamma}$).
By definition of the above toroid, $\angle AGB =\overline{\gamma}$ for any
point $G$ of $L$. It is clear also that $\angle AGC+ \angle CGB=\angle AGC=\overline{\gamma}$.
Let us consider the values $(\angle AGC, \angle CGB)$ for various points on the arc $L$.
We have  $(\angle AEC, \angle CEB)=(\overline{\gamma},0)$ and
$(\angle AFC, \angle CFB)=(0,\overline{\gamma})$. It is clear that for some point $G$ on the arc $L$ between $E$ and $F$ we get
$(\angle AGC, \angle CGB)=(\overline{\beta},\overline{\alpha})$, see Fig.~\ref{plane_tor}. Hence, we can take this point $G$ as $D$. The proposition is proved.
\end{proof}

\section{On the limit points of the image $F(\GT)$} \label{sec.limit.point}

Now, we discuss the limit points of the images $F(D)$, when $D$ tends to one of the base vertices $A$, $B$, or $C$.
We will use some Cartesian coordinate system in $\mathbb{R}^3$ such that  $A, B, C\in \mathbb{R}^2=\{(p,q,0)\in \mathbb{R}^3\}$.
We consider the vectors $\tau_i$ and real numbers $\varphi_i$, $i=1,2,3$, where
$\tau_1=\frac{\overrightarrow{AB}}{\|\overrightarrow{AB}\|}=(\cos \varphi_1, \sin \varphi_1)$,
$\tau_2=\frac{\overrightarrow{BC}}{\|\overrightarrow{BC}\|}=(\cos \varphi_2, \sin \varphi_2)$,
$\tau_3=\frac{\overrightarrow{CA}}{\|\overrightarrow{CA}\|}=(\cos \varphi_3, \sin \varphi_3)$.
It is clear that $\cos \angle BAC=-(\tau_1,\tau_3)=-\cos(\varphi_1-\varphi_3)$, $\cos \angle ABC=-(\tau_1,\tau_2)=-\cos(\varphi_2-\varphi_1)$, and
$\cos \angle ACB=-(\tau_2,\tau_3)=-\cos(\varphi_3-\varphi_2)$.

The third coordinate of the point $D=(p,q,r)$ is assumed to be non-negative.
Consider the case where $D$ tends to $B$  (the cases with $A$ and $C$ are treated similarly).

For a given nontrivial vector $\eta \in \mathbb{R}^3$, we consider the points $D=B+t\cdot \eta$,  $t\in \mathbb{R}$.
If the third coordinate of $\eta$ is not $0$, then we consider $t$ only with the same sign as this coordinate has.
It is easy to check that
\begin{eqnarray*}
\cos \overline{\alpha}&=&\frac{\left(\overrightarrow{DB},\overrightarrow{DC}\right)}{\|\overrightarrow{DB}\|\cdot\|\overrightarrow{DC}\|}=
\frac{\left(-t\eta,\overrightarrow{BC}-t\eta\right)}{\|t\eta\|\cdot\|\overrightarrow{BC}-t\eta\|}\rightarrow
-\left(\tau_2, \frac{\eta}{\|\eta\|}\right),\\
\cos \overline{\beta}&=&\frac{\left(\overrightarrow{DA},\overrightarrow{DC}\right)}{\|\overrightarrow{DA}\|\cdot\|\overrightarrow{DC}\|}\rightarrow
\frac{(\overrightarrow{BA},\overrightarrow{BC})}{\|\overrightarrow{BA}\|\cdot\|\overrightarrow{BC}\|}=-(\tau_1,\tau_2)=\cos \angle ABC,\\
\cos
\overline{\gamma}&=&\frac{\left(\overrightarrow{DA},\overrightarrow{DB}\right)}{\|\overrightarrow{DA}\|\cdot \|\overrightarrow{DB}\|}=
\frac{\left(\overrightarrow{BA}-t\eta,-t\eta\right)}{\|\overrightarrow{BA}-t\eta\|\cdot\|t\eta\|}\rightarrow
\left(\tau_1, \frac{\eta}{\|\eta\|}\right),
\end{eqnarray*}
when $t \to 0$.

Now, we consider a sequence of points $D_n$ converging to the point $B$, when  $n \to \infty$. Then (passing to a subsequence, if necessary) we can assume that
$$
\frac{\overrightarrow{BD_n}}{\|\overrightarrow{BD_n}\|}\rightarrow\eta \in S^2=\{(p,q,r)\in\mathbb{R}^3\,|\,p^2+q^2+r^2=1\}.
$$
It is easy to check that

\begin{eqnarray*}
\cos \overline{\alpha}&=&
\left(\frac{\overrightarrow{D_nB}}{\|\overrightarrow{D_nB}\|},\frac{\overrightarrow{D_nC}}{\|\overrightarrow{D_nC}\|}\right)=
\left(-\eta, \frac{\overrightarrow{BC}}{\|\overrightarrow{BC}\|}\right)\rightarrow
-\left(\tau_2, \eta\right),\\
\cos \overline{\beta}&=&
\left(\frac{\overrightarrow{D_nA}}{\|\overrightarrow{D_nA}\|},\frac{\overrightarrow{D_nC}}{\|\overrightarrow{D_nC}\|}\right)\rightarrow
\left(\frac{\overrightarrow{BA}}{\|\overrightarrow{BA}\|},\frac{\overrightarrow{BC}}{\|\overrightarrow{BC}\|}\right)=
-\left(\tau_1, \tau_2\right)=\cos \angle ABC,\\
\cos \overline{\gamma}&=&\left(\frac{\overrightarrow{D_nA}}{\|\overrightarrow{D_nA}\|},\frac{\overrightarrow{D_nB}}{\|\overrightarrow{D_nB}\|}\right)=
\left(\frac{\overrightarrow{BA}}{\|\overrightarrow{BA}\|},-\eta\right)\rightarrow
\left(\tau_1, \eta\right),
\end{eqnarray*}
when $n \to \infty$.
If $\eta=\eta_h+\eta_v$, where $\eta_h$ ($\eta_v$) is parallel (respectively, is orthogonal) to the plane $ABC$, then
(since $(\eta_v, \tau_i)=0$ for $i=1,2,3$) we have

\begin{equation}\label{eq.lim.point.im}
(\cos \overline{\alpha}, \cos \overline{\beta}, \cos \overline{\gamma})
\rightarrow \left( -\frac{(\tau_2,\eta_h)}{\|\eta\|}, \cos \angle ABC,\frac{(\tau_1,\eta_h)}{\|\eta\|}\right).
\end{equation}

In particular, if $\eta_h=0$, then
$(\cos \overline{\alpha}, \cos \overline{\beta}, \cos \overline{\gamma}) \rightarrow \left(0, \cos \angle ABC,0\right)$.

Let us fix  vectors $a,b \in \mathbb{R}^3$ such that $a$ is parallel and $b$ is orthogonal to the plane $ABC$, $\|a\|=\|b\|=1$.
Now, let us consider a one-parameter family of vectors $\eta=\eta(t)=ta+\sqrt{1-t^2}b$, $t\in [0,1]$.
It is clear that $\|\eta(t)\|=1$, $\eta_h=ta$, $\eta_v=\sqrt{1-t^2}b$,  $(\eta, \tau_i)=(\eta_h, \tau_i)=t(a,\tau_i)$ for $i=1,2,3$.
In particular,  $\|\eta_h\| / \|\eta\|=t$ can be any number in $[0,1]$. For a fixed $t\in [0,1]$, we get
\begin{equation*}
(\cos \overline{\alpha}, \cos \overline{\beta}, \cos \overline{\gamma})
\rightarrow \left( -(\tau_2,a)\cdot t, \cos \angle ABC,(\tau_1,a)\cdot t \right).
\end{equation*}
This observation implies the following result.

\begin{lemma}\label{le.limitpoint.1}
The limit points of the points $F(D)$, when $D \to B$, is the following set:
$$
\operatorname{Lim}(B)=\left\{\bigl(x, y_0, z\bigr)\in \mathbb{R}^3 \,\left|\,
x^2+z^2-2y_0xz \leq 1-y_0^2,\, y_0=\cos \angle ABC\right.\right\}.
$$
\end{lemma}

\begin{proof} The above discussion and the equality  $\cos \angle ABC=-\cos(\varphi_2-\varphi_1)$
imply that it suffices to prove the following result: the real numbers $x$ and $z$
satisfy the equation $x^2+z^2+2xz\cdot \cos(\varphi_2-\varphi_1)=\sin^2 (\varphi_2-\varphi_1)$ if and only if there is a vector $\theta \in S^1=\{(x,z)\in \mathbb{R}^2\,|\,x^2+z^2=1\}$ such that
$x=-(\tau_2, \theta)$ and $z=(\tau_1, \theta)$.

If $\theta=(\cos u, \sin u)$, then
\begin{eqnarray*}
x&=&-(\tau_2, \theta)=-\cos(\varphi_2-u)=-\cos \varphi_2\cos u-\sin \varphi_2\sin u,\\
z&=&(\tau_1, \theta)=\cos(\varphi_1-u)=\cos \varphi_1\cos u+\sin \varphi_1\sin u.
\end{eqnarray*}
Therefore, we get (note that $\sin(\varphi_1-\varphi_2)\neq 0)$ that
$$
\cos u=-\frac{z \sin \varphi_2+x \sin \varphi_1}{\sin(\varphi_1-\varphi_2)}, \quad
\sin u=\frac{z \cos \varphi_2+x \cos \varphi_1}{\sin(\varphi_1-\varphi_2)},
$$
hence,
$$
1=\cos^2 u+\sin^2 u=\frac{x^2+z^2+2xz \cos(\varphi_2-\varphi_1)}{\sin^2 (\varphi_2-\varphi_1)}.
$$
On the other hand, if $x^2+z^2+2xz\cdot \cos(\varphi_2-\varphi_1)=\sin^2 (\varphi_2-\varphi_1)$, then we can find some $u\in \mathbb{R}$, such that
$\theta=(\cos u, \sin u)$ satisfies the previous two equalities. Therefore, $x=-(\tau_2, \theta)$ and $z=(\tau_1, \theta)$. The lemma is proved.
\end{proof}
\smallskip

\begin{figure}[t]
\begin{minipage}[h]{0.3\textwidth}
\center{\includegraphics[width=0.97\textwidth]{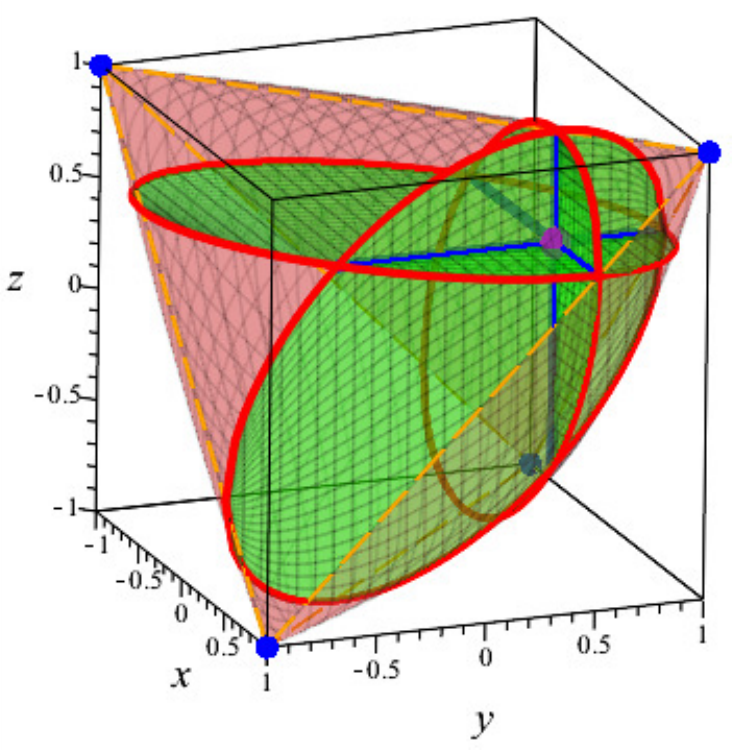}} \\ a)                     \\
\end{minipage}
\begin{minipage}[h]{0.3\textwidth}
\center{\includegraphics[width=0.92\textwidth]{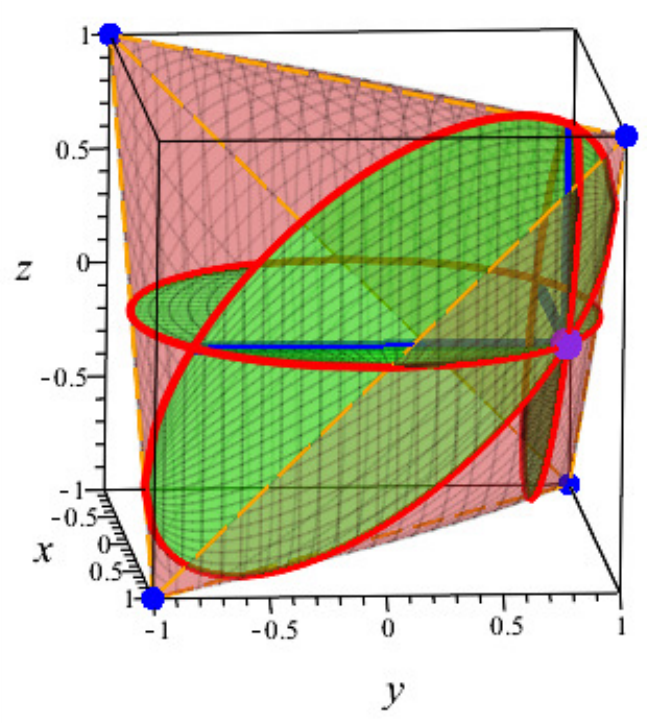}}   \\ b) \\
\end{minipage}
\begin{minipage}[h]{0.3\textwidth}
\center{\includegraphics[width=0.95\textwidth]{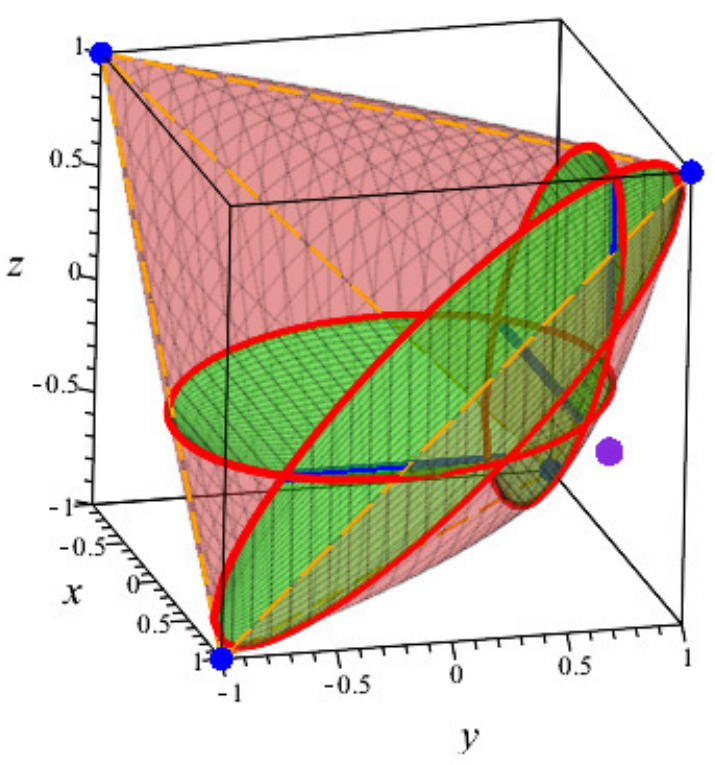}}   \\ c) \\
\end{minipage}
\caption{The limit solid ellipses $\operatorname{Lim}(A)$, $\operatorname{Lim}(B)$ and $\operatorname{Lim}(C)$ for:
a) an acute-angled $\triangle ABC$; b) a right-angled $\triangle ABC$; c) an obtuse-angled $\triangle ABC$.
Here the ``pillowcase'' $\mathbb{BP}$ is drawn in orange, the blue dots on $\mathbb{BP}$ are its four  non-smoothness points
(which are the vertices of a regular tetrahedron), the purple point is the intersection point of three planes containing three limit solid ellipses.}
\label{Ellipset}
\end{figure}

Using similar arguments for the vertices $A$ and $C$, we obtain the following result.

\begin{prop}\label{pr.limitpoint.1}
The limit points of the points $F(D)$, when $D \to A$, $D \to B$, or $D \to C$, are as follows:
$$
\operatorname{Lim}(A)=\left\{\bigl(x_0, y, z\bigr)\in \mathbb{R}^3 \,\left|\,
y^2+z^2-2 x_0yz \leq 1-x_0^2, \, x_0=\cos \angle BAC \right.\right\},
$$
$$
\operatorname{Lim}(B)=\left\{\bigl(x, y_0, z\bigr)\in \mathbb{R}^3 \,\left|\,
x^2+z^2-2y_0xz \leq 1-y_0^2,\, y_0=\cos \angle ABC\right.\right\},
$$
$$
\operatorname{Lim}(C)=\left\{\bigl(x, y,z_0\bigr)\in \mathbb{R}^3 \,\left|\,
x^2+y^2-2z_0xy \leq 1-z_0^2,\, z_0=\cos \angle ACB  \right.\right\}.
$$
\end{prop}

It is clear that each of the sets $\operatorname{Lim}(A)$, $\operatorname{Lim}(B)$, and $\operatorname{Lim}(C)$
is the convex hull of some ellipse in a suitable plane (so-called ``solid ellipses''). We will call them {\it limit solid ellipses}.

\begin{prop}\label{pr.limitpoint.2}
The point $\Bigl(\cos \angle BAC, \cos \angle ABC, \cos \angle ACB\Bigr)$ is a common point of three planes containing
the limit solid ellipses $\operatorname{Lim}(A)$, $\operatorname{Lim}(B)$, and $\operatorname{Lim}(C)$.
Moreover, this point lies in the interior of $\mathbb{P}$, on the surface $\mathbb{BP}$, and outside of $\mathbb{P}$ in according to the property of
$\triangle ABC$ to be  acute-angled, right-angled and obtuse-angled respectively.
\end{prop}

\begin{proof}
We know that the points $\widetilde{A}=\Bigl(-\cos \angle BAC, \cos \angle ABC, \cos \angle ACB\Bigr)$
lie on the surface $\mathbb{BP}$ (see \eqref{eq.spec.points.1} or Proposition \ref{pr.limitpoint.1}).
This means that
$$
1-2 \cos \angle BAC \cdot \cos \angle ABC \cdot \cos \angle ACB-\cos^2 \angle BAC -\cos^2 \angle  ABC-\cos^2 \angle ACB=0.
$$
On the other hand, the point $\Bigl(\cos \angle BAC, \cos \angle ABC, \cos \angle ACB\Bigr)$ lies in the interior of $\mathbb{P}$, on the surface $\mathbb{BP}$,
and outside of $\mathbb{P}$ just according to the property of the sign of
\begin{eqnarray*}
1+2 \cos \angle BAC \cdot \cos \angle ABC \cdot \cos \angle ACB-\cos^2 \angle BAC -\cos^2 \angle  ABC-\cos^2 \angle ACB\\
=4\cos \angle BAC \cdot \cos \angle ABC \cdot \cos \angle ACB
\end{eqnarray*}
to be positive, zero, or negative. This proves the proposition.
\end{proof}

The relative position of the limit solid ellipses $\operatorname{Lim}(A)$, $\operatorname{Lim}(B)$, and $\operatorname{Lim}(C)$,
and the point
$\Bigl(\cos \angle BAC, \cos \angle ABC, \cos \angle ACB\Bigr)$
for the acute-angled, right-angled and obtuse-angled cases is shown in
Fig.~\ref{Ellipset}

\medskip

\begin{prop}\label{pr.closure.space.1}
The closure $\CFT$ of $F(\GT)$, where $\GT=\Pi^+ \bigcup \bigl(\Pi^0 \setminus \{A,B,C\} \bigr)$ {\rm(}see \eqref{eq.fmain}{\rm)},
is a subset of the ``pillow''
$$
\mathbb{P}=\left\{(x,y,z)\in [-1,1]^3\,|\,1+2xyz-x^2-y^2-z^2 \geq  0 \right\}
$$ and has the following form:
$$
\CFT=F\Bigl(\Pi^+ \bigcup \Pi^0\setminus \{A,B,C\}\Bigr) \bigcup \{(1,1,1)\}\bigcup \operatorname{Lim}(A) \bigcup \operatorname{Lim}(B) \bigcup \operatorname{Lim}(C).
$$
\end{prop}

\begin{proof}
We know that $F(\GT) \subset \CFT \subset \mathbb{P}$ by Theorem \ref{theo.1}.
Let us consider any $w_0 \in \CFT$. Then there is a sequence of points $w_n \in F(\GT)$, $n\in \mathbb{N}$, such that $w_n \to w_0$ as $n\to \infty$.
Now, we consider  a sequence  $\{u_n\}_{n\in \mathbb{N}}$ of points $u_n \in \GT$ such that $w_n=F(u_n)$, $n\in \mathbb{N}$.

If the sequence $\{u_n\}_{n\in \mathbb{N}}$ is unbounded, then (passing to a subsequence if necessary) we may assume that $\|u_n\| \to \infty$ when $n\to \infty$.
In this case $w_n \rightarrow (1,1,1)$, when $n\to \infty$ by Lemma \ref{le.infin.1}.
If the sequence $\{u_n\}_{n\in \mathbb{N}}$ is bounded, then (passing to a subsequence if need) we may assume that $u_n \to u_0 \in \Pi^+ \bigcup \Pi^0$
when $n \to \infty$. If $u_0 \not\in\{A,B,C\}$, then $w_0 =F(u_0) \in F\Bigl(\Pi^+ \bigcup \Pi^0\setminus \{A,B,C\}\Bigr)$.
On the other hand, if, say,  $u_0=A$ (the cases $u_0=B$ and $u_0=C$ can be settled similarly), then $w_0 \in \operatorname{Lim}(A)$ by the construction of
$\operatorname{Lim}(A)$. The proposition is proved.
\end{proof}
\smallskip

Using \eqref{eq.lim.point.im} and Propositions \ref{pr.limitpoint.1} and \ref{pr.closure.space.1}, we obtain the following result:

\begin{prop}\label{pr.closure.plane.1}
The closure of $F\Bigl(\Pi^0\setminus \{A,B,C\}\Bigr)$ is a subset of the
``pillowcase'' $\mathbb{BP}=\left\{(x,y,z)\in [-1,1]^3\,|\,1+2xyz-x^2-y^2-z^2 = 0 \right\}$ and has the following form:
$$
F\Bigl(\Pi^0\setminus \{A,B,C\}\Bigr) \bigcup \{(1,1,1)\} \bigcup E_A \bigcup E_B \bigcup E_C,
$$
where $E_A, E_B, E_C$ are the following ellipses:
$$
E_A=
\left\{\bigl(x_0, y, z\bigr)\in \mathbb{R}^3 \,\left|\,
y^2+z^2-2x_0yz  = 1-x_0^2, \, x_0=\cos \angle BAC \right.\right\},
$$
$$
E_B=
\left\{\bigl(x, y_0, z\bigr)\in \mathbb{R}^3 \,\left|\,
x^2+z^2-2y_0 xz = 1-y_0^2,\, y_0=\cos \angle ABC\right.\right\},
$$
$$
E_C=
\left\{\bigl(x, y,z_0\bigr)\in \mathbb{R}^3 \,\left|\,
x^2+y^2-2 z_0 xy = 1-z_0^2,\, z_0=\cos \angle ACB  \right.\right\}.
$$
In this case, $E_A, E_B, E_C$ are the intersections of the surface $\mathbb{BP}$
with the planes $x=\cos \angle BAC$, $y=\cos \angle ABC$, $z=\cos \angle ACB$, respectively. Moreover, we get
\begin{eqnarray*}
\widetilde{A}&=&\Bigl(-\cos \angle BAC, \cos \angle ABC, \cos \angle ACB\Bigl) \in E_B\bigcap E_C,\\
\widetilde{B}&=&\Bigl(\cos \angle BAC, -\cos \angle ABC, \cos \angle ACB\Bigl) \in E_A\bigcap E_C,\\
\widetilde{C}&=&\Bigl(\cos \angle BAC, \cos \angle ABC, -\cos \angle ACB\Bigl) \in  E_A\bigcap E_B.
\end{eqnarray*}
\end{prop}

\begin{proof}
We know that $F(D)\in \mathbb{BP}$ for any $D \in \Pi^0\setminus \{A,B,C\}$ by Lemma
\ref{le.inpl.1}. Since $\mathbb{BP}$ is closed in $\mathbb{R}^3$, then we get that the closure of  $F\Bigl(\Pi^0\setminus \{A,B,C\}\Bigr)$ is a subset
of $\mathbb{BP}\bigcap \CFT$, see Proposition \ref{pr.closure.space.1}.

Now, if we consider the limiting sets when  $D \in \Pi^0\setminus \{A,B,C\}$ tends to one of the vertices $A,B,C$,
then we get the sets $E_A$, $E_B$, $E_C$   according to Proposition \ref{pr.limitpoint.1},
formula~\eqref{eq.lim.point.im}, and the fact that the vertical part of the vector $\eta$ is trivial in our case.

The last assertion is easily checked by direct computations (see also Lemma \ref{le.map.plane.1}).
\end{proof}

\begin{example}\label{ex.ellril.1}
Recall, that $\widetilde{A}$, $\widetilde{B}$, $\widetilde{C}$ (see \eqref{eq.spec.points.1}) are on the ``pillowcase'' $\mathbb{BP}$.
If, say, $\angle BAC= \pi/2$ and $\angle ABC=\varphi$, then $\operatorname{Lim}(A)$ is a circle of unit radius in the plane of the second and the third coordinates
by Proposition \ref{pr.limitpoint.1} and \eqref{eq.pillowc}.
By Proposition~\ref{pr.closure.plane.1}, we get that
$\widetilde{A}=\Bigl(0, \cos \varphi, \sin \varphi \Bigl) \in  E_B\bigcap E_C$,
$\widetilde{B}=\Bigl(0, -\cos \varphi, \sin \varphi \Bigl) \in E_A\bigcap E_C$,
$\widetilde{C}=\Bigl(0, \cos \varphi, -\sin \varphi\Bigl) \in E_A\bigcap E_B$.
In particular, the point $\widetilde{C}$ is in all three ellipses $E_A$, $E_B$, and $E_C$.
\end{example}
\smallskip

In the last part of this section, we consider some special immersions of the $2$-dimensional sphere $S^2$ into $3$-dimensional Euclidean space
$\mathbb{E}^3$, related to our problem.

The map $F$ is defined and is continuous on $\Pi^0 \setminus \{A,B,C\}$. If we add to the plane $\Pi^0$ an infinite point $\infty$
(that is, we pass to its one-point compactification), then we can assume by continuity that $F(\infty)=(1,1,1)$, see Lemma \ref{le.infin.1}.
Moreover, we can naturally  ``glue'' three copies of the unit disk instead of the points $A,B,C$ with a natural continuous extension $F$ to these disks.

As these three copies of the unit disk we will use the limit solid ellipses $\operatorname{Lim}(A)$, $\operatorname{Lim}(B)$, $\operatorname{Lim}(C)$.
If we extend the map $F$ to each of these solid ellipses as the identity map, then we preserve the continuity of $F$
(see the arguments in the proof of Lemma \ref{le.limitpoint.1}).
If fact, we can determine the required gluings by the fact that the extended mapping $F$ remains continuous.

Hence, we obtain a surface
$F \Bigl(\bigl(\Pi^0 \setminus \{A,B,C\}\bigr)\bigcup \{\infty\}\Bigr)$ with the above-described gluing of $\Bigl(\operatorname{Lim}(A) \coprod \operatorname{Lim}(B) \coprod \operatorname{Lim}(C)\Bigr)$ to it (here we used disjoint unions),
that is homeomorphic to $S^2$, and the (extended) map $F$, that is an immersion of $S^2$
into $\mathbb{R}^3$.
The image of this immersion is
$$
F\Bigl(\Pi^0 \setminus \{A,B,C\}\Bigr)\bigcup \{(1,1,1)\} \bigcup \operatorname{Lim}(A) \bigcup\operatorname{Lim}(B)\bigcup\operatorname{Lim}(C).
$$

It is known that any two $C^2$ immersions of $S^2$ into $\mathbb{E}^3$ are regularly homotopic  (i.e., they can be connected with
immersions) \cite{Smale}.

For any immersion $g$  of $S^2$ into $\mathbb{E}^3$ in general position (i.e. when there are finitely many triple
points of $g$), the number of triple points is even (see e.g. \cite{Ban} or \cite{Nowik}).

In our case, the immersion $F$ is not in general position. Moreover, in any neighborhood of the point $\infty$, $F$ is not injective, see Proposition
\ref{pr.map.plane.1} and \eqref{eq.noninj.plane.1}.

\section{On the images of the points of $\Cyl$}\label{sec.image.cyl}

Consider the circumscribed circle $S$ around the triangle $ABC$ with the center $O$.
Then $\operatorname{Cyl}=S\times (-\infty,\infty)$, i.e. $D=(p,q,r) \in \operatorname{Cyl} \bigcap \Pi^{+}$ if and only if  $(p,q) \in S$ and $r>0$.
{\it Without loss of generality, we may and will assume that the radius of $S$ is $1$}.
\smallskip

First, let us consider the case  $r=0$.
Let us suppose that
the point $D=(p,q,0)\in \Pi^0$ lies on the circle $S$ and $D\not\in\{A,B,C\}$,
see the left panel of Fig.~\ref{Circ}.
By the definition of $F$, we have
$F(D)=\bigl(\cos \angle BDC,\cos\angle ADC,\cos\angle ADB\bigr)$.

We know that
$F(D)=\widetilde{A}=\Bigl(-\cos \angle BAC, \cos \angle ABC, \cos \angle ACB\Bigr)$, if the point $D$ lies on the arc $BC$;
$F(D)=\widetilde{B}=\Bigl(\cos \angle BAC, -\cos \angle ABC, \cos \angle ACB\Bigr)$, if the point $D$ lies on the arc $AC$ and
$F(D)=\widetilde{C}=\Bigl(\cos \angle BAC, \cos \angle ABC, -\cos \angle ACB\Bigr)$,
if the point $D$ lies on the arc $AB$
(see \eqref{eq.spec.points.1} and Lemma \ref{le.map.plane.1}).
Therefore, as we already know, the function $F$ on the set $S\backslash\{A,B,C\}$ can take only three possible values.
\bigskip

Now, we are going to find the limit points of the points $F(D)$, where $D\in \operatorname{Cyl}$ and when $D \to A$, $D \to B$, or $D \to C$.
For this goal we will use \eqref{eq.lim.point.im}.

We consider the vectors $\tau_i$ and real numbers $\varphi_i$, $i=1,2,3$, that are determined by the formulas
{\small$$
\tau_1=\frac{\overrightarrow{AB}}{\|\overrightarrow{AB}\|}=(\cos \varphi_1, \sin \varphi_1),\,
\tau_2=\frac{\overrightarrow{BC}}{\|\overrightarrow{BC}\|}=(\cos \varphi_2, \sin \varphi_2),\,
\tau_3=\frac{\overrightarrow{CA}}{\|\overrightarrow{CA}\|}=(\cos \varphi_3, \sin \varphi_3).
$$}
We consider also some vectors  $e_A$, $e_B$, $e_C$, such that they are of unit length, are parallel to the plane $ABC$ and are orthogonal to the vectors
$\overrightarrow{OA}$, $\overrightarrow{OB}$, $\overrightarrow{OC}$ respectively (recall that $O$ is the center of $S$). Let us also consider the vector $e_3=(0,0,1)$.

Let us consider the case where $D$ tends to $B$ on $\operatorname{Cyl}$  (the cases with $A$ and $C$ can be studied by the same methods).
Up to a positive multiple, only vectors of the type $\eta= e_3+ \lambda \cdot e_B$, where $\lambda \in\mathbb{R}$,
are tangent to  $\operatorname{Cyl}$ at the point $B$ with positive inner product $(\eta, e_3)$.

According to \eqref{eq.lim.point.im}, we have
$$
F(D)=(\cos \overline{\alpha}, \cos \overline{\beta}, \cos \overline{\gamma})\rightarrow
\left( -\frac{(\tau_2,\eta_h)}{\|\eta\|}, \cos \angle ABC,\frac{(\tau_1,\eta_h)}{\|\eta\|}\right),
$$
where $\eta_h=  \lambda \cdot e_B$, $\eta_v=e_3$ (note that $(\tau_1,\tau_2)=-\cos \angle ABC$).
Therefore, the limit points of the points $F(D)$, where $D\in \operatorname{Cyl}$ and  $D \to B$,
is a segment with the endpoints  $\bigl(-(\tau_2,e_B), \cos \angle ABC, (\tau_1,e_B)\bigr)$
and  $\bigl((\tau_2,e_B), \cos \angle ABC, -(\tau_1,e_B)\bigr)$.
On the other hand, if $D$ tends to $B$ along the arc $AB$ ($BC$), then $F(D)=\widetilde{C}$ (respectively, $F(D)=\widetilde{A}$).
Hence, $(\tau_1,e_B)=\pm \cos \angle ACB$ and $(\tau_2,e_B)=\pm \cos \angle BAC$.
\smallskip

Applying similar arguments to the points $A$ and $C$, we get the following result:

\begin{prop}\label{pr.limitpoint.c1}
The limit points of the points $F(D)$,  where $D\in \operatorname{Cyl}$, and when $D \to A$, $D \to B$, or $D \to C$, are the segments
\begin{equation*}
[\widetilde{B},\widetilde{C}]\subset \operatorname{Lim}(A),\quad
[\widetilde{A},\widetilde{C}]\subset \operatorname{Lim}(B),\quad
[\widetilde{A},\widetilde{B}]\subset \operatorname{Lim}(C),
\end{equation*}
respectively.
\end{prop}

\begin{remark}
The case $\lambda=0$ in the above arguments gives the limit of $F(D)$ for $D$ moving on the straight line
orthogonal to the plane $\Pi^0$ through the point $B$. The limit point of $F(D)$ for $D\to B$ moving
as described above is the point $(0,\cos \angle ABC,0)$.
Similar arguments give the limits of $F(D)$ for $D$ moving on the straight lines
orthogonal to the plane $\Pi^0$ through the point $A$ and $C$: the points $(\cos \angle BAC, 0,0)$ and $(0,0, \cos \angle ACB)$ respectively.
\end{remark}

\begin{remark}
If we consider a curve $\Gamma$ on $\operatorname{Cyl}\setminus \{A,B,C\}$, which goes around the point $A$,
with  some points $P$ and $Q$ that are  on the arcs $AC$ and $AB$ respectively, then $F(P)=\widetilde{B}$ and  $F(Q)=\widetilde{C}$.
Therefore, the image of $F(\Gamma)$ is a closed curve passing through the points
$\widetilde{B}$ and  $\widetilde{C}$ (even if $\Gamma$ is in a very small neighborhood of $A$).

If we have a closed curve $\Gamma$, which goes through points in all arcs  $BC$, $AC$, and $AB$, then the image $F(\Gamma)$
is closed and goes through the points $\widetilde{A}$, $\widetilde{B}$, and  $\widetilde{C}$.
\end{remark}

\begin{figure}[t]
\begin{minipage}[h]{0.36\textwidth}
\center{\includegraphics[width=0.89\textwidth]{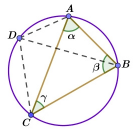}} \\ a)                     \\
\end{minipage}
\quad\quad\quad\quad
\begin{minipage}[h]{0.36\textwidth}
\center{\includegraphics[width=0.89\textwidth]{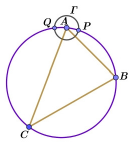}}   \\ b) \\
\end{minipage}
\caption{a) The circumscribed circle $S$; b) A graph of $ \Gamma$.}
\label{Circ}
\end{figure}

\medskip

Let us consider the asymptotics of $F(D)$  when $D=(p,q,r)\in \operatorname{Cyl}$ and $r \to \infty$.
We know that  $F(p,q,r) \to (1,1,1)\in \mathbb{BP}$ for $r \to \infty$ by Lemma \ref{le.infin.1}. But we can say more.

\begin{prop}\label{pr.asym.cyl}
Let us suppose that $A=(1,0,0)$, $B=(\cos \mu, \sin \mu,0)$, \linebreak
$C=(\cos \nu, \sin \nu,0)$, $\mu \neq 0$, $\nu\neq 0$, and $\mu \neq \nu$.
Then we have the following asymptotics for the points $D=(p,q,r) \in \operatorname{Cyl}$:
$$
r^2\Bigl( F(p,q,r) -(1,1,1) \Bigr) \to \bigl(\cos(\mu - \nu)-1, \cos (\nu) -1, \cos (\mu)-1\bigr)
$$
as $r \to \infty$, $p^2+q^2=1$. In this case, the right side of the given asymptotic relation can be considered as a vector
with an initial and final point on the surface $\mathbb{BP}$.
\end{prop}

\begin{proof} Using \eqref{eq.fmain}, it is easy to check that
\begin{equation}\label{eq.asym.cyl}
r^2\Bigl( F(p,q,r) -(1,1,1) \Bigr) \to \bigl(\cos(\mu - \nu)-1, \cos (\nu) -1, \cos (\mu)-1\bigr)
\end{equation}
as $r \to \infty$, $p^2+q^2=1$. Note that the points $(1,1,1)$ and $\bigl(\cos(\mu - \nu), \cos (\nu), \cos (\mu)\bigr)$
are contained in the surface $\mathbb{BP}$. This proves the proposition.
\end{proof}
\smallskip

From this proposition and Proposition \ref{pr.top.cap} we immediately get the following result.

\begin{corollary}\label{cor.cyl.noi}
There is a real number $M$ such that for any $r>M$ the image of any point $D=(p,q,r)\in \operatorname{Cyl}=S\times (-\infty,\infty)$,
$p^2+q^2=1$, is not on the boundary of the closure of $F(\GT)=F\bigl(\Pi^{+}\bigcup \Pi^0 \setminus \{A,B,C\}\bigr)$.
\end{corollary}

\medskip

It is interesting to study some important geometric property of the surface $\operatorname{FC}:=F(\operatorname{Cyl}\bigcap \Pi^{+})$.
The following problem is currently open.

\begin{problem}\label{prob.inject.1}
Is it true that the surface $\operatorname{FC}=F(\operatorname{Cyl}\bigcap \Pi^{+})$ has no self-intersection, i.e., the map $F$ is injective on
the set $\operatorname{Cyl}\bigcap \Pi^{+}$?
\end{problem}

It is clear that $\operatorname{FC}$ is not a closed subset in $\mathbb{R}^3$, see Proposition \ref{pr.limitpoint.c1}.
There is one important, but  not so obvious fact: some points of $\operatorname{FC}$ are not smooth.
We discuss these properties below in this section.

We consider the Cartesian system such that
$$
O=(0,0,0), \quad A=(1,0,0), \quad B=\left(\frac{u^2 - 1}{u^2 + 1},\frac{2u}{u^2 + 1},0\right),
$$
$$
C=\left(\frac{v^2 - 1}{v^2 + 1},\frac{2v}{v^2 + 1},0\right), \quad D=\left(\frac{1-t^2}{t^2 + 1},\frac{2t}{t^2 + 1},r\right),
$$
where $u,v \in \mathbb{R}$ are some constants,  $t\in \mathbb{R}$ and $r>0$ are parameters on $\operatorname{Cyl}\bigcap \Pi^+$.
It is clear that $B=C$ is equivalent to $u=v$, and $u+v=0$ means that $B$ and $C$ are symmetric each to other with respect to the line $OA$.
{\it In what follows, we assume that $u\neq v$.}

Let us consider the following values:
\begin{eqnarray*}\label{eq.simpl.FG1}
P_1&:=&(t^2 + 1)r^2 + 4t^2, \\
P_2&:=&(u^2 + 1)(t^2 + 1)r^2 + 4(tu - 1)^2, \\
P_3&:=&(v^2 + 1)(t^2 + 1)r^2 + 4(tv - 1)^2.
\end{eqnarray*}

Since $\operatorname{FC}=\operatorname{FC}(t,r)=(F_1,F_2,F_3)$, we get that $\operatorname{FC}$ has the following $x$, $y$ and $z$ coordinates:
{\small \begin{eqnarray*}
&&\frac{(v^2 + 1)(u^2 + 1)(t^2 + 1)r^2 + 4(uv + 1)(tv - 1)(tu - 1)}{\sqrt{u^2 + 1}\sqrt{v^2 + 1}\sqrt{P_2}\sqrt{P_3}},\\
&&\frac{(v^2 + 1)(t^2 + 1)r^2 + (2tv - 1)^2 - 1}{\sqrt{v^2 + 1}\sqrt{P_1}\sqrt{P_3}}, \quad
\frac{(u^2 + 1)(t^2 + 1)r^2 + (2tu - 1)^2 - 1}{\sqrt{u^2 + 1}\sqrt{P_1}\sqrt{P_2}}.
\end{eqnarray*}}

We see that $\operatorname{FC}_t^{\prime}:=\partial \operatorname{FC}/{\partial t}$ has the following $x$, $y$ and $z$ coordinates:
{\small \begin{eqnarray*}
&&\frac{16(t + u)(t + v)(u - v)^2r^2((v + u)t^2 + (2uv - 2)t - v - u)}
{\sqrt{u^2 + 1}\sqrt{v^2 + 1}\cdot  P_2^{3/2}
P_3^{3/2}},\\
&&\frac{16r^2(t + v)(t^2 + 2tv - 1)}{\sqrt{v^2 + 1}\cdot  P_1^{3/2} \cdot P_3^{3/2}},\quad
\frac{16r^2(t + u)(t^2 + 2tu - 1)}{\sqrt{u^2 + 1}\cdot  P_1^{3/2} \cdot P_2^{3/2}}.
\end{eqnarray*}}

Note also that $\operatorname{FC}_r^{\prime}:=\partial \operatorname{FC}/{\partial r}$ has the following $x$, $y$ and $z$ coordinates:
{\small\begin{eqnarray*}
&&\frac{4r\Bigl((v^2 + 1)(u^2 + 1)(t^2 + 1)^2r^2 + 4(tv - 1)(tu - 1)\bigl((t^2 - 1)(uv - 1) - 2(v + u)t\bigr)\Bigr)}
{(u - v)^{-2}(t^2 + 1)^{-1}\sqrt{u^2 + 1}\sqrt{v^2 + 1}\cdot P_2^{3/2}\cdot P_3^{3/2}},\\
&&\frac{4r\Bigl((v^2 + 1)(t^2 + 1)^2r^2 + 4t(1-3t^2)v + 4t^2(t^2-1)v^2 + 8t^2\Bigr)}
{(t^2 + 1)^{-1}\sqrt{v^2 + 1}\cdot P_1^{3/2}\cdot P_3^{3/2}},\\
&&\frac{4r\Bigl( (u^2 + 1)(t^2 + 1)^2r^2 + 4t(1-3t^2)u + 4t^2(t^2-1)u^2 + 8t^2 \Bigr)}
{(t^2 + 1)^{-1}\sqrt{u^2 + 1}\cdot P_1^{3/2}\cdot P_2^{3/2}}.
\end{eqnarray*}}

By direct computations, we obtain the following result.

\begin{lemma}\label{le.norm.s}
The cross product $N=N(t,r)$ of the vectors $\operatorname{FC}_t^{\prime}$ and  $\operatorname{FC}_r^{\prime}$ has the following form:
$$
N=\left(\frac{N_1}{\sqrt{P_1}}, \frac{N_1 (t + u)(v - u)}{\sqrt{P_2} \sqrt{u^2 + 1}},
\frac{N_1 (t + v)(u - v)}{\sqrt{P_3} \sqrt{v^2 + 1}}\right),
$$

where
$$
N_1=\frac{64r^3(u - v)(t^2 + 1)^2 \bigl((v + u)t^3 + 3(uv - 1)t^2 -3(u+v)t - uv + 1\bigr)}{\sqrt{(u^2 + 1)(v^2 + 1)\cdot P_1^3 \cdot P_2^3 \cdot P_3^3}}.
$$
\end{lemma}

\begin{prop}\label{pro.nondegen.1}
There are three different angles $\theta_1,\theta_2,\theta_3 \in [0,2\pi)$, such that
for all points of the set
$$
M_{\,\operatorname{reg}}:=\Bigl\{(\cos \theta, \sin \theta, r)\,|\, \theta \in [0,2\pi)\setminus \{\theta_1,\theta_2,\theta_3\}, \, r \in (0,\infty)\Bigr\},
$$
the surface $\operatorname{FC}=F(\operatorname{Cyl}\bigcap \Pi^+)$ is smooth.
\end{prop}

\begin{proof}
The surface $F(\operatorname{Cyl}\bigcap \Pi^+)$ is smooth at the point $\operatorname{FC}(t,r)$, if the vectors
$\operatorname{FC}_t^{\prime}(t,r)$ and  $\operatorname{FC}_r^{\prime}(t,r)$ are linearly independent.
The vectors $\operatorname{FC}_t^{\prime}$ and  $\operatorname{FC}_r^{\prime}$ are linearly dependent at a given point if and only if
its cross product $N$ is a zero vector at this point.
Using Lemma \ref{le.norm.s}, we see that $N$ is a zero vector
(excluding the case $u=v$) exactly if $t$ satisfies the equation
\begin{equation}\label{eq.nonsm.1}
(u+v)t^3 + 3(uv-1)t^2 - 3(u+v)t - uv + 1=0.
\end{equation}
It is easy to check that this equation has three real solutions with respect to $t$ if $u+v\neq 0$.
Indeed, in this case the discriminant of this equation with respect to $t$ is positive.
If $u+v=0$, then we have two real solutions and one ``infinite'' solution. Indeed, if we substitute $t=\tan (\theta/2)$ in \eqref{eq.nonsm.1},
we get the solution $\theta= \pi$. In any case, we have three solutions with non-smooth points of the type $(\cos \theta_i, \sin \theta_i, r)$, $i=1,2,3$, $r>0$.
\end{proof}
\smallskip

\begin{remark}
As noted by M.Q. Rieck, if we substitute $u=\tan(\mu/2)$, $v=\tan(\nu/2)$ and $t=\tan (\theta/2)$, then equation \eqref{eq.nonsm.1} factors as follows:
$$
\cos\left(\frac{3\theta+\mu+\nu}{2}\right)\sec^3\left(\frac{\theta}{2}\right)\sec\left(\frac{\mu}{2}\right)\sec\left(\frac{\nu}{2}\right)=0.
$$
This equation has the roots $\theta=\frac{1}{3}\bigl(\pi(2k+1)-\mu -\nu \bigr)$, where $k$ is integer and $k \in \left.\left[\frac{\mu+\nu}{2\pi}-1/2, \frac{\mu+\nu}{2\pi}+5/2\right. \right)$.
\end{remark}

\begin{remark}
It is easy to check that the image under the map $F$ of each ray \linebreak
$\{(\cos \theta_i, \sin \theta_i,r) \mid r \in (0,\infty) \}\subset \operatorname{Cyl}\bigcap \Pi^{+}$, where
$i=1,2,3$, is an ``edge'' of the continuous surface $\operatorname{FC}=F(\operatorname{Cyl}\bigcap \Pi^{+})$, i.e.,
is a continuous curve without smooth points, see
Fig. \ref{Image_Cyl}.
\end{remark}

The surface $\operatorname{FC}=F(\operatorname{Cyl}\bigcap \Pi^{+})$ is $C^{\infty}$-smooth
for all points $(\cos \theta, \sin \theta, r) \in M_{\,\operatorname{reg}}$.
The following result is somewhat unexpected, but very interesting.

\begin{prop}\label{prop.poscurv.1}
For any non-degenerate triangle $ABC$, the surface $\operatorname{FC}=F(\operatorname{Cyl}\bigcap \Pi^{+})$
has positive Gaussian curvature on the set  $M_{\,\operatorname{reg}}$.
\end{prop}

\begin{proof} We are going to find the expression for the  Gaussian curvature $K$ for
$\operatorname{FC}=F(\operatorname{Cyl}\bigcap \Pi^{+})$ for the coordinates $t,r$ as above.
It is well known that $K=D_2/D_1$, where $D_1$ and $D_2$ are respectively the determinants of the first and the second quadratic forms of the
surface $\operatorname{FC}$.
We know that $D_1=(N,N)=\|N\|^2$, where $N$ is the cross product of the vectors $\operatorname{FC}_t^{\prime}$ and  $\operatorname{FC}_r^{\prime}$.
It should be noted that $\|N\|>0$ for all points in $M_{\,\operatorname{reg}}$ by Lemma \ref{le.norm.s}.
For $D_2$ we have the formula (see e.g. \cite[2.4]{Top})
$$
D_2=\frac{(\operatorname{FC}_{t,t}^{\prime\prime},N)}{\|N\|}\frac{(\operatorname{FC}_{r,r}^{\prime\prime},N)}{\|N\|}
-\frac{(\operatorname{FC}_{t,r}^{\prime\prime},N)^2}{\|N\|^2}=
\frac{(\operatorname{FC}_{t,t}^{\prime\prime},N)(\operatorname{FC}_{r,r}^{\prime\prime},N)-(\operatorname{FC}_{t,r}^{\prime\prime},N)^2}{\|N\|^2},
$$
where $\operatorname{FC}_{t,t}^{\prime\prime}$ etc.\ denote the respective second partial derivatives of
$\operatorname{FC}$.
By Lemma \ref{le.norm.s} and the proof of Proposition \ref{pro.nondegen.1}, we get that $\|N(t,r)\|>0$ for all points
$(\cos \theta, \sin \theta, r) \in M_{\,\operatorname{reg}}$, where
$t=\tan (\theta/2)$. On the other hand, the direct computations shows that
\begin{eqnarray*}
(\operatorname{FC}_{t,t}^{\prime\prime},N)(\operatorname{FC}_{r,r}^{\prime\prime},N)-(\operatorname{FC}_{t,r}^{\prime\prime},N)^2=\\
\frac{1048576 r^{10}\bigl((u+v)t^3 + 3(uv-1)t^2 - 3(u+v)t - uv + 1\bigr)^4(u - v)^6(r^2 + 4)(t^2 + 1)^8}
{(u^2 + 1)^2(v^2 + 1)^2\cdot P_1^6 \cdot P_2^6 \cdot P_3^6},
\end{eqnarray*}
where $P_1,P_2, P_3$ are as above. We see that $D_2\geq 0$ and
$D_2=0$ implies $N=N_1=0$,
see Lemma \ref{le.norm.s} and  \eqref{eq.nonsm.1}.
Therefore, $D_2>0$ and $K>0$ for all regular points of $\operatorname{FC}$. This proves the proposition.
\end{proof}

\begin{example}\label{ex.imcylcurv}
By direct computations, we obtain that the Gaussian curvature for the cylinder image in the case, where triangle $ABC$ is regular
$\Bigl(u=\frac{1}{\sqrt{3}}$, $v = -\frac{1}{\sqrt{3}}\Bigr)$, has the following form:
$$
K(t,r)=\frac{(r^2+4)(r^2t^2+r^2+4t^2)^2\bigl((r^2t^2+r^2+t^2+3)^2-12t^2\bigr)^2}
{27r^2\bigl((r^4+5r^2)(t^2+1)^3+4t^6+30t^4+6\bigr)^2}.
$$
It should be noted that $r>0$ and
$(t^2+3)^2-12t^2=(t^2-3)^2\geq 0$.
In this case $M_{\,\operatorname{reg}}=
\bigl\{(\cos \theta, \sin \theta, r)\,|\, \theta \in [0,2\pi)\setminus \{\pi/3,\pi,5\pi/3\}, \, r \in (0,\infty)\bigr\}$.
We note that the images of the points $(1/2, \sqrt{3}/2,r)$, $(-1,0,r)$, $(-1/2, \sqrt{3}/2,r)$ under the map $F$
are non-regular points of the surface $\operatorname{FC}=F(\operatorname{Cyl}\bigcap \Pi^{+})$  for any positive $r$.
\end{example}

Recall that $\operatorname{Cyl}$ is a right cylinder over
the circle $S$ circumscribed around triangle $ABC$.
We know that all points $D\in \operatorname{Cyl}$ are degenerate for the map $F$.
For any point $D\in \operatorname{Cyl}\bigcap \Pi^{+}$, we can find explicitly the direction, where the differential of $F$ vanishes.

Let us consider suitable Cartesian coordinates $(p,q,r)$.
As above, we assume that the origin is the center of circumscribed circle $S$ around $\triangle ABC$ and the radius of $S$ is $1$.
We take the coordinates such that
$A=(1,0,0)$,  $B=(\cos(\alpha),\sin(\alpha),0)$, $C=(\cos(\beta),\sin(\beta),0)$,
where $\alpha\in(0,\pi]$ and $\beta\in(-\pi,0)$.

Points of $\operatorname{Cyl}\bigcap \Pi^{+}$ have the form
$D=(\cos(\varphi),\sin(\varphi),r)$ for some $\varphi \in [0,2\pi)=\mathbb{R}/2\pi \mathbb{Z}$ and $r>0$.
Recall that such points are degenerate for the map $F=(F_1,F_2,F_3)$.

Let us fix $D=(\cos \varphi,\sin \varphi ,z)$ and consider a straight line through this point of the form
\begin{equation}\label{str.line}
p = \cos \varphi +\eta_1 \cdot t, \quad    q = \sin \varphi +\eta_2\cdot  t, \quad    r = z +\eta_3\cdot  t
\end{equation}
for some
vector $\eta=(\eta_1,\eta_2,\eta_3)$ and $t\in \mathbb{R}$.

The following result clarifies the fact that $F$ is degenerate on the cylinder $\operatorname{Cyl}$.

\begin{prop}\label{pr.degen.cyl2}
In the above notation, we have the following asymptotic equality for some $\xi_i \in \mathbb{R}$, $i=1,2,3$:
\begin{equation}\label{eq.degen.1}
F_i\bigl(\cos \varphi+\eta_1\cdot  t, \sin \varphi +\eta_2 \cdot t,z+\eta_3\cdot  t\Bigr)= F_i\bigl(\cos \varphi, \sin \varphi,z\bigr)+ \xi_i\cdot  t^2 +o(t^2),
\end{equation}
when $t \to 0$, where $i=1,2,3$,
{\small
\begin{eqnarray*}
\eta_1=\Bigl((\sin \beta -\sin \alpha)(1+\cos \varphi)+(\cos \alpha -\cos \beta )\sin \varphi \Bigr)z^2\\
-
2\Bigl((\sin \beta-\sin \alpha)(\cos \beta-\cos \varphi)+(\cos \alpha-\cos \beta)(\sin \beta-\sin \varphi)\Bigr)\sin^2 \varphi,\\
\eta_2=\bigl(\cos \alpha -\cos \beta\bigr) \Bigl(\cos \varphi \bigl(2\cos ^2 \varphi +z^2-2\bigr)+z^2\Bigr)\\
+
\sin \varphi \bigl(\sin \alpha-\sin \beta\bigr) \bigl(2\cos^2 \varphi+z^2\bigr)+\sin (\beta-\alpha) \sin (2\varphi),\\
\eta_3=z\bigl(\sin \beta -\sin \alpha \bigr)\bigl(2-\cos \varphi-3\cos^2 \varphi \bigr)\\
-
z\bigl(\cos \alpha -\cos \beta\bigr) \sin \varphi \bigl(3\cos \varphi+1\bigr)+z\sin (\beta-\alpha)(1+\cos \varphi).
\end{eqnarray*} }
Moreover, $\xi_1^2$, $\xi_2^2$, and $\xi_3^2$ are rational functions with respect to $z$, such that the degree of the numerator coincides with the degree of the denominator and is equal to $20$, $16$, $16$, respectively.
\end{prop}

\begin{proof}
Indeed, it suffices to choose  the vector $\eta$ orthogonal to each gradient
$\nabla F_i$, where $i=1,2,3$, in the coordinates $(p,q,r)$. Direct computations show that this condition is satisfied for the vector $\eta$ as stated
and suitable $\xi_i$, $i=1,2,3$.
\end{proof}
\smallskip

Since in formula \eqref{eq.degen.1} on the right side the coefficient of $t$ is zero and $\xi_i \neq 0$, $i=1,2,3$, (at least) everywhere
on $S\times  (0,\infty)$,
therefore in a neighbourhood of the point $t=0$ the image of the straight line \eqref{str.line} is ``almost glued'' $(t,-t\mapsto t^2)$.
Locally $F$ acts as the map $(x,y,z)\mapsto (x^2,y,z)$ for small $|x|$ in suitable coordinates. Therefore, $F$
is not locally injective and is not locally surjective at such points $D\in \operatorname{Cyl}$.
In particular, we get

\begin{corollary}\label{cor.cyl.deg}
The image of a point $D\in \operatorname{Cyl}\bigcap \Pi^{+}$ under the map $F$ is a boundary point of $F(\operatorname{GT})$ in the case when
$F^{-1}(F(D))=\{D\}$.
\end{corollary}

\begin{proof} It follows from the above discussion that $F$ is not
locally surjective at $D\in \operatorname{Cyl}\bigcap \Pi^{+}$.
If $Q=F(D)$ and $F^{-1}(F(D))=\{D\}$, then $Q$ cannon be an interior point of $F(\operatorname{GT})$.
\end{proof}

Now we are going to get an interesting addition to Proposition \ref{pr.limitpoint.2}.

\begin{prop}\label{pr.limitpoint.3}
Suppose that the triangle $ABC$ is acute-angled. Then there is a point $D$ inside the cylinder  $\operatorname{Cyl}$
such that $F(D)=\Bigl(\cos \angle BAC, \cos \angle ABC, \cos \angle ACB\Bigr)$.
\end{prop}

\begin{proof}
Without loss of generality, we can suppose that
$A=(1,0,0)$, \linebreak
$B=\left(\frac{1-u^2}{1+u^2},\frac{2u}{1+u^2},0\right)$, and $C=\left(\frac{1-v^2}{1+v^2},\frac{2v}{1+v^2},0\right)$
for some non-zero real numbers $u\neq v$. It is easy to check that
\begin{eqnarray*}
\Bigl(\cos \angle BAC, \cos \angle ABC, \cos \angle ACB\Bigr)\\=
\left(\frac{(1+uv)\sgn(uv)}{\sqrt{(1+u^2)(1+v^2)}},  \frac{\sgn((u - v)u)}{\sqrt{1+v^2}}, -\frac{\sgn((u - v)v)}{\sqrt{1+u^2}}\right).
\end{eqnarray*}
In particular, we get
\begin{equation}\label{rq.cos.prod}
\cos \angle BAC \cdot \cos \angle ABC \cdot \cos \angle ACB=-\frac{1+uv}{(1+u^2)(1+v^2)}
\end{equation}
Therefore, the triangle $ABC$ is acute-angled if and only if $uv+1 <0$.
Let us suppose that $uv+1 <0$ and let us consider the point $D$ with the following coordinates:
$$
p_0=\frac{u^2v^2-u^2-v^2-3}{(1+u^2)(1+v^2)},\quad
q_0=\frac{-2(u+v)(1+uv)}{(1+u^2)(1+v^2)},\quad
r_0=\sqrt{\frac{-16(1+uv)}{(1+u^2)(1+v^2)}}.
$$
Ii is easy to check that
\begin{eqnarray*}
F(D)=\Bigl(\cos \angle BDC, \cos \angle ADC, \cos \angle ADB\Bigr)\\=
\left(\frac{(1+uv)\sgn(uv)}{\sqrt{(1+u^2)(1+v^2)}},  \frac{\sgn((u - v)u)}{\sqrt{1+v^2}}, -\frac{\sgn((u - v)v)}{\sqrt{1+u^2}}\right)\\
=\Bigl(\cos \angle BAC, \cos \angle ABC, \cos \angle ACB\Bigr).
\end{eqnarray*}
Moreover, $p_0^2+q_0^2-1=\frac{8(1+uv)}{(1+u^2)(1+v^2)}<0$, hence, $D$ is inside the cylinder $\operatorname{Cyl}$.
This proves the proposition.
\end{proof}

\begin{remark}
It should be noted that the point $D=(p_0,q_0,r_0)$ in the proof of Proposition \ref{pr.limitpoint.3} is above the {\it de Longchamps point} for $\triangle ABC$,
which is the reflection of the orthocenter of $\triangle ABC$ about the circumcenter.
The see this, note that as vectors, $A+D$, $B+D$, $C+D$ are orthogonal to $B-C$,
$C-A$, $A-B$, respectively.
\end{remark}

\begin{corollary}\label{cor.com.point}
If the triangle $ABC$ is acute-angled, then the point
$$
\Bigl(\cos \angle BAC, \cos \angle ABC, \cos \angle ACB\Bigr)
$$
is contained in the interior of the set $\CFT$ {\rm(}the closure of $F(\GT)${\rm)}.
\end{corollary}

\begin{proof}
We know that $\Bigl(\cos \angle BAC, \cos \angle ABC, \cos \angle ACB\Bigr)=F(D)$ for some point $D$  inside the cylinder  $\operatorname{Cyl}$,
see Proposition \ref{pr.limitpoint.3}. On the other hand, $D$ is not critical point of the map $F$, see Theorem \ref{theo.nondeg}.
Therefore, $F$ maps some neighborhood of $D$ to some neighborhood of the point $F(D)$, that proves the corollary.
\end{proof}

\bigskip

\section{Mappings of special parts of the cylinder  $\Cyl$}\label{sec.cyl.special}

In this section, we consider some properties of the images of some special parts of the cylinder $\operatorname{Cyl}$ and some related parts of $\Pi^{+}$.
This results help us to describe the boundary of the closure of $F(\GT)=\Sigma(\triangle ABC)$ in the final section.

\subsection{On special regions on the cylinder  $\Cyl$}\label{subs.8.1}
We will say that a toroid is {\it special} if it is built on some base side ($KL$ in \eqref{eq.toroid.s.1}) with
an angle equal to the base angle at the opposite vertex.

If this opposite angle is straight or obtuse, then there is no intersection of the toroid with  $\operatorname{Cyl}\bigcap \Pi^+$.
If the angle is acute, then the intersection of the toroid and  $\operatorname{Cyl}\bigcap \Pi^+$ is a curve, whose orthogonal projection to $\Pi^0$
is an arc of the circle $S$ of measure $\pi$.

For a fixed $\alpha_1=\angle BAC$,  $\alpha_2=\angle ABC$ and $\alpha_3=\angle ACB$
we define curves $\gamma_{BC}$, $\gamma_{AC}$ and $\gamma_{AB}$ as follows (if some $\alpha_i \geq \pi/2$, then the corresponding curve is empty):
\begin{equation}\label{eq.spes.curves}
\gamma_{BC}=\mathbb{T}_{BC}^{\,\alpha_1}\cap\operatorname{Cyl}, \quad
\gamma_{AC}=\mathbb{T}_{AC}^{\,\alpha_2}\cap\operatorname{Cyl}, \quad
\gamma_{AB}=\mathbb{T}_{AB}^{\,\alpha_3}\cap\operatorname{Cyl}.
\end{equation}

\begin{figure}[t]
\begin{minipage}[h]{0.36\textwidth}
\center{\includegraphics[width=0.89\textwidth]{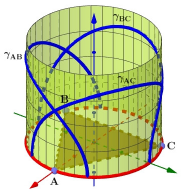}} \\ a)                     \\
\end{minipage}
\quad\quad
\begin{minipage}[h]{0.36\textwidth}
\center{\includegraphics[width=0.89\textwidth]{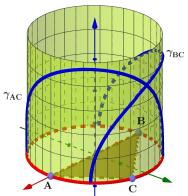}}   \\ b) \\
\end{minipage}
\quad\quad
\begin{minipage}[h]{0.36\textwidth}
\center{\includegraphics[width=0.89\textwidth]{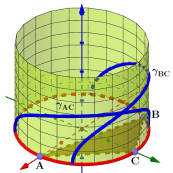}}   \\ c) \\
\end{minipage}
\caption{Special regions on Cyl for: a) an acute-angled $\triangle ABC$; b) a right-angled $\triangle ABC$; c) an obtuse-angled $\triangle ABC$.}
\label{Cyl}
\end{figure}

\begin{definition}\label{def.spec.reg}
We will call a region of $\operatorname{Cyl}$ bounded by any two of the curves $\gamma_{BC}$, $\gamma_{AC}$, $\gamma_{AB}$ and by the plane $\Pi^0$,
a {\it special region} of the cylinder $\operatorname{Cyl}$.
\end{definition}

We are going to derive an explicit equation of the curve $\gamma_{BC}$ for an acute angle $\alpha_1$
(similar formulas exist for the curves $\gamma_{AC}$ and $\gamma_{AB}$).
Since $\angle BAC=\alpha_1$, therefore $\angle BOC=2\alpha_1$, where $O$ is the center of the circle $S$ circumscribed around the triangle $ABC$.
Without loss of generality, we assume that the radius of $S$ is $1$.

\begin{prop}\label{pr.crit.tor.cyl}
Suppose that $\alpha_1 <\pi/2$ and the Cartesian coordinates are chosen such that $O=(0,0,0)$,  $B=\bigl(\cos(\alpha_1),\sin(\alpha_1),0\bigr)$,
and $C=\bigl(\cos(\alpha_1),-\sin(\alpha_1),0\bigr)$.
Then a point $D=(\cos(\varphi),\sin(\varphi),z) \in \operatorname{Cyl}$, where $z>0$, is contained in the toroid $\mathbb{T}_{BC}^{\,\alpha_1}$
if and only if $\varphi\in(-\pi/2,\pi/2)$ and $z= 2\sqrt{\cos(\alpha_1) \cos(\varphi)}$. Moreover,
the curve
$$
\gamma_{BC}(\varphi)=\left(\cos(\varphi),\sin(\varphi),2\sqrt{\cos(\alpha_1) \cos(\varphi)}\right)
$$
can be continuously extended to the closed interval
$[-\pi/2,\pi/2]$, where $\gamma_{BC}(-\pi/2)=(0,-1,0)$ and $\gamma_{BC}(\pi/2)=(0,1,0)$.
\end{prop}

\begin{proof} Note that $A$ lies on the longer arc $BC$ of the circle $S$ under the hypotheses of the proposition.
Recall that $D\in \mathbb{T}_{BC}^{\,\alpha_1}$ if and only if $\left(\overrightarrow{DB},\overrightarrow{DC}\right)=\|DB\|\cdot \|DC\|\cdot \cos (\alpha_1)$.
By direct computations, we have
\begin{equation*}
\left(\overrightarrow{DB},\overrightarrow{DC}\right)^2 - \cos^2(\alpha_1)\cdot \|DB||^2 \cdot \|DC\|^2=
\sin^2(\alpha_1)\cdot z^2 \cdot \bigl(z^2 - 4\cos(\alpha_1)\cdot \cos(\varphi)\bigr).
\end{equation*}
Therefore, we have only one possibility:
$z= 2\sqrt{\cos(\alpha_1) \cos(\varphi)}$ and $\varphi\in(-\pi/2,\pi/2)$ due to $\cos(\alpha_1)>0$. Direct computations shows that
any point $D=(\cos(\varphi),\sin(\varphi),z)$ with $\varphi\in(-\pi/2,\pi/2)$ and $z= 2\sqrt{\cos(\alpha_1) \cos(\varphi)}$
actually belongs to the toroid $\mathbb{T}_{BC}^{\,\alpha_1}$. The last assertion is obvious. The proposition is proved.
\end{proof}

\begin{figure}[t]
\begin{minipage}[h]{0.3\textwidth}
\center{\includegraphics[width=0.89\textwidth]{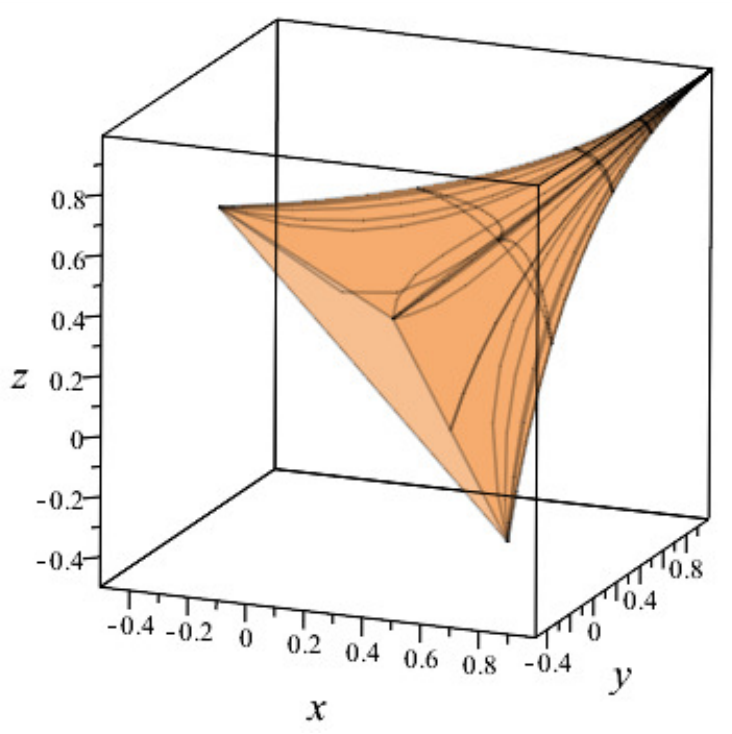}} \\ a)                     \\
\end{minipage}
\begin{minipage}[h]{0.3\textwidth}
\center{\includegraphics[width=0.89\textwidth]{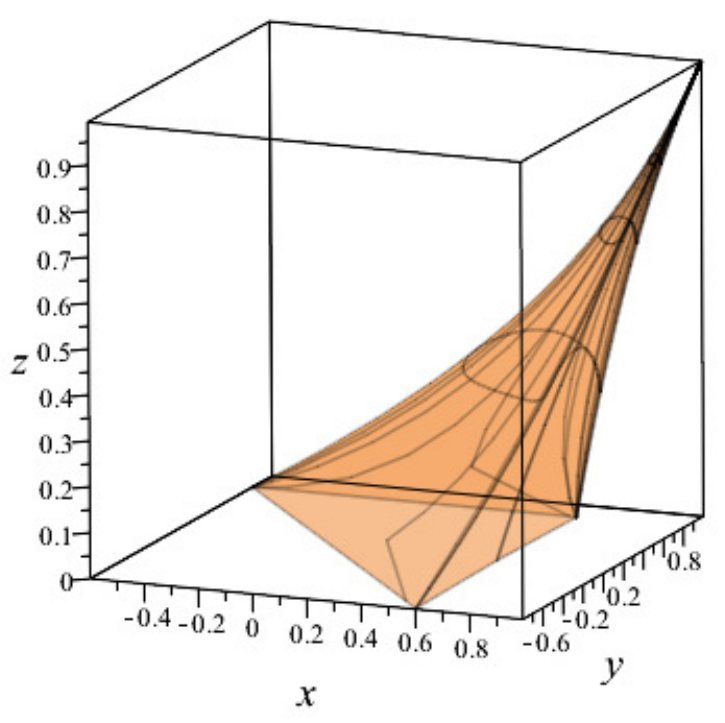}}   \\ b) \\
\end{minipage}
\begin{minipage}[h]{0.3\textwidth}
\center{\includegraphics[width=0.89\textwidth]{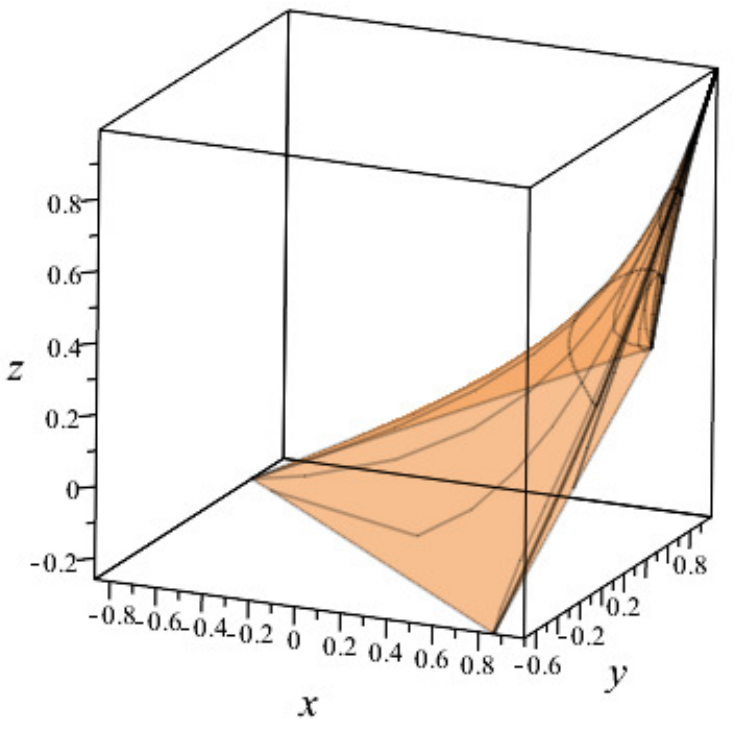}}   \\ c) \\
\end{minipage}
\caption{The image of  $\operatorname{Cyl}\bigcap \Pi^{+}$ under the map $F$ for:
a) an acute-angled $\triangle ABC$; b) a right-angled $\triangle ABC$; c) an obtuse-angled $\triangle ABC$.}
\label{Image_Cyl}
\end{figure}

\begin{figure}[t]
\begin{minipage}[h]{0.3\textwidth}
\center{\includegraphics[width=0.89\textwidth]{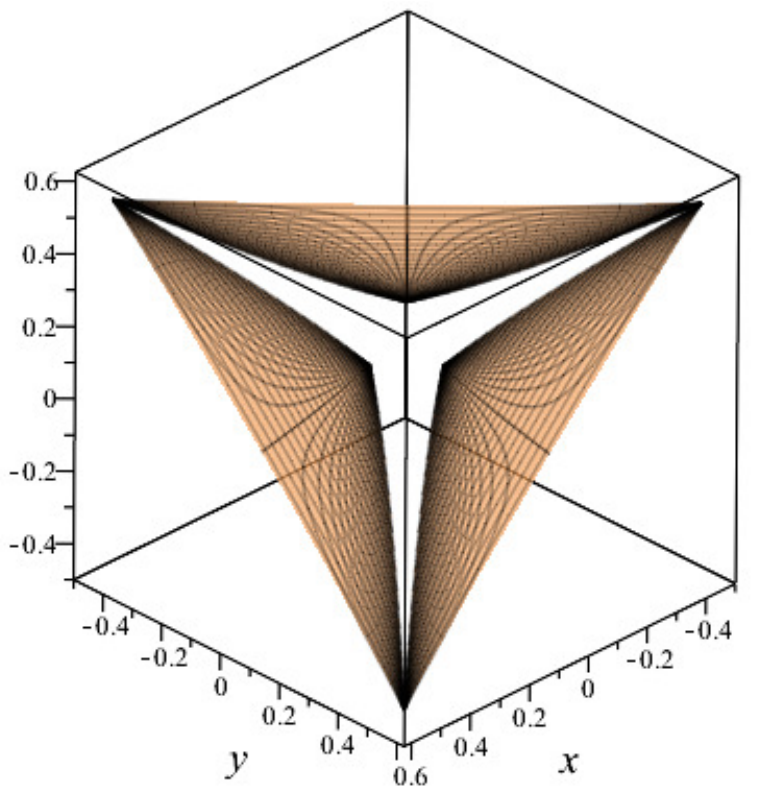}} \\ a)                     \\
\end{minipage}
\begin{minipage}[h]{0.3\textwidth}
\center{\includegraphics[width=0.89\textwidth]{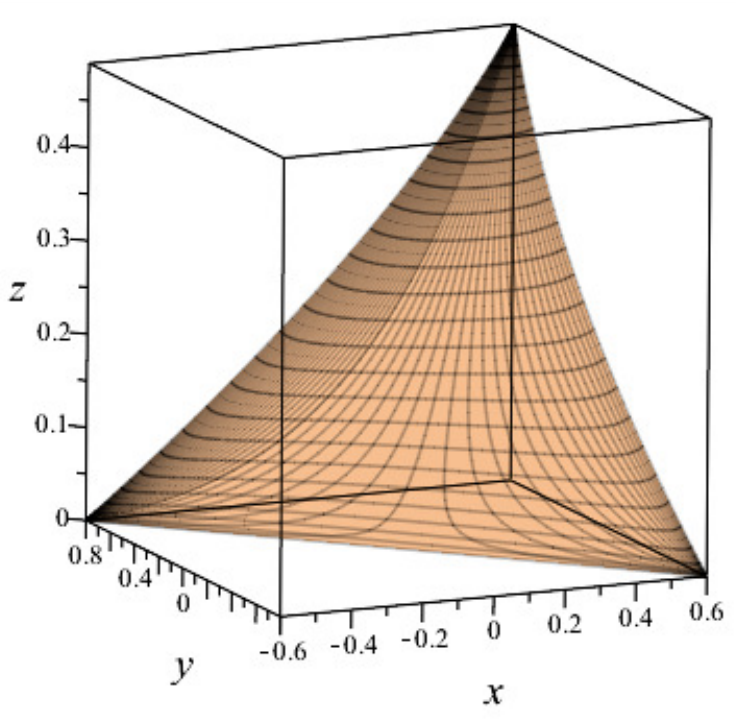}}   \\ b) \\
\end{minipage}
\begin{minipage}[h]{0.3\textwidth}
\center{\includegraphics[width=0.89\textwidth]{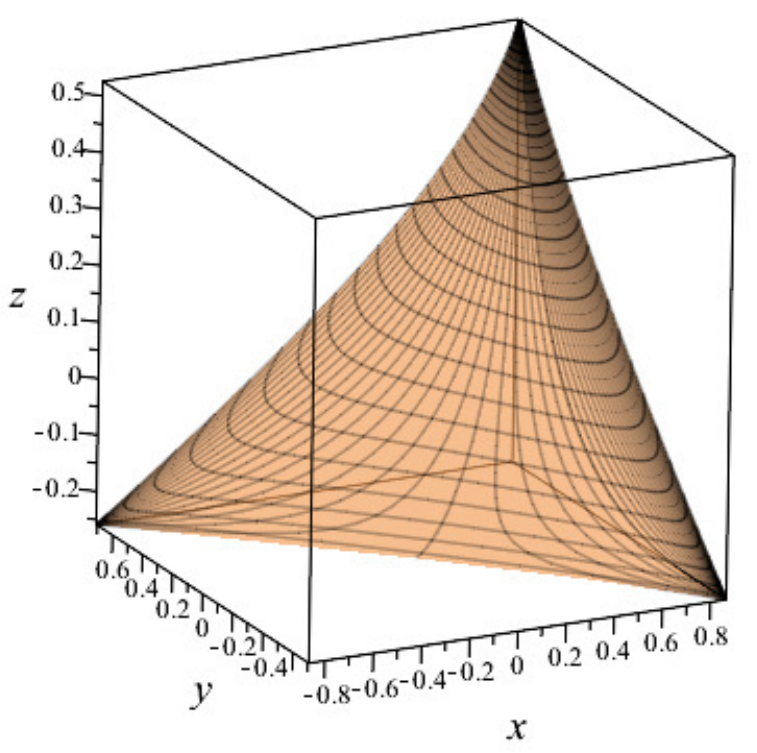}}   \\ c) \\
\end{minipage}
\caption{The part of the image of $\operatorname{Cyl}\bigcap \Pi^{+}$ (under the map $F$) lying on  $\BFT$ for:
a) an acute-angled $\triangle ABC$; b) a right-angled $\triangle ABC$; c) an obtuse-angled $\triangle ABC$.}
\label{Part_image_Cyl}
\end{figure}

\begin{remark}\label{re.crit.tor.cyl}
If the Cartesian coordinates are chosen such that $O=(0,0,0)$,  $B=\bigl(\cos(\theta+2\alpha_1),\sin(\theta+2\alpha_1),0\bigr)$,
and $C=\bigl(\cos(\theta),\sin(\theta),0\bigr)$ for some constant $\theta \in \mathbb{R}$, then the curve  $\gamma_{BC}$ has parametrization as follows:
$z= 2\sqrt{\cos(\alpha_1) \cos(\varphi-\theta-\alpha_1)}$, $\varphi\in(-\pi/2+\theta+\alpha_1,\pi/2+\theta+\alpha_1)$.
To prove this result, it suffices to rotate the plane $ABC$ to the angle
$\theta+\alpha_1$ around $O$ and apply Proposition \ref{pr.crit.tor.cyl}.
\end{remark}

Now we consider an explicit description of a special region of the cylinder $\operatorname{Cyl}$, which is bounded by
the curves $\gamma_{AC}$, $\gamma_{AB}$ and by the plane $\Pi^0$ (here we assume that $\alpha_2, \alpha_3 \in (0,\pi/2)$).
The case when a special region is bounded by another pair of curves is considered similarly.

\begin{prop}\label{pr.crit.tor.cyl.2}
Suppose that the Cartesian coordinates are chosen such that $O=(0,0,0)$, $A=(1,0,0)$, $B=\bigl(\cos(2\alpha_3),\sin(2\alpha_3),0\bigr)$,
and $C=\bigl(\cos(2\alpha_2),-\sin(2\alpha_2),0\bigr)$. Then the corresponding special region
is such that $\varphi \in [\alpha_3-\pi/2, -\alpha_2+\pi/2]\subset(-\pi/2, \pi/2)$ and
\begin{eqnarray*}
0\leq  z \leq 2\sqrt{\cos(\alpha_3) \cos(\varphi-\alpha_3)}&& \mbox{for\,\,\,}   \varphi \in [\alpha_3-\pi/2, \varphi_0], \\
0\leq  z \leq 2\sqrt{\cos(\alpha_2) \cos(\varphi+\alpha_2)}&& \mbox{for\,\,\,}   \varphi \in [\varphi_0, -\alpha_2+\pi/2],
\end{eqnarray*}
where
$\varphi_0=\arctan \left(\frac{\cos^2(\alpha_2) - \cos^2(\alpha_3)}{\sin(\alpha_3)\cos(\alpha_3) + \sin(\alpha_2)\cos(\alpha_2)} \right)$.
\end{prop}

\begin{proof} Using Proposition \ref{pr.crit.tor.cyl} and Remark \ref{re.crit.tor.cyl}, we obtain the following equations for the curves
$\gamma_{AC}$ and $\gamma_{AB}$:
\begin{eqnarray*}
\gamma_{AB}(\varphi)=\left(\cos(\varphi),\sin(\varphi),2\sqrt{\cos(\alpha_3) \cos(\varphi-\alpha_3)}\right),&&  \varphi \in [\alpha_3-\pi/2, \alpha_3+\pi/2];\\
\gamma_{AC}(\varphi)=\left(\cos(\varphi),\sin(\varphi),2\sqrt{\cos(\alpha_2) \cos(\varphi+\alpha_2)}\right),&&  \varphi \in [-\alpha_2-\pi/2, -\alpha_2+\pi/2].
\end{eqnarray*}
Therefore, we have $\varphi \in [\alpha_3-\pi/2, -\alpha_2+\pi/2]$ for our special region.

It is easy to see that there is a unique point $\varphi_0 \in [\alpha_3-\pi/2, -\alpha_2+\pi/2]\subset(-\pi/2, \pi/2)$ such that
$2\sqrt{\cos(\alpha_3) \cos(\varphi_0-\alpha_3)}=2\sqrt{\cos(\alpha_2) \cos(\varphi_0+\alpha_2)}$. Indeed, such a point satisfies the following equation:
$$
\bigl(\cos^2(\alpha_3) - \cos^2(\alpha_2)\bigr) \cdot \cos(\varphi_0) + \bigl(\sin(\alpha_3)\cos(\alpha_3) + \sin(\alpha_2)\cos(\alpha_2)\bigr)\cdot \sin(\varphi_0)=0.
$$
Hence, $\varphi_0$ is determined uniquely by the condition
$\tan(\varphi_0)=\frac{\cos^2(\alpha_2) - \cos^2(\alpha_3)}{\sin(\alpha_3)\cos(\alpha_3) + \sin(\alpha_2)\cos(\alpha_2)}$.

Finally, it suffices to note that
$2\sqrt{\cos(\alpha_3) \cos(\varphi-\alpha_3)} < 2\sqrt{\cos(\alpha_2) \cos(\varphi+\alpha_2)}$ for $\varphi \in [\alpha_3-\pi/2,\varphi_0)$ and
$2\sqrt{\cos(\alpha_3) \cos(\varphi-\alpha_3)} > 2\sqrt{\cos(\alpha_2) \cos(\varphi+\alpha_2)}$ for $\varphi \in (\varphi_0,-\alpha_2+\pi/2]$.
The proposition is proved.
\end{proof}

\begin{example}\label{ex.imcyl.spec}
Fig. \ref{Cyl} shows examples of curves \eqref{eq.spes.curves}, when the base $ABC$ is an acute-angled, a right-angled or an obtuse-angled triangle.

For the regular $ABC$ we get the following parametrization:
\begin{equation*}
\gamma_{AC}(\varphi)=\left(\, \cos \varphi, \, \sin \varphi, \, \sqrt{\cos \varphi +\sqrt{3}\sin \varphi}\, \right),
\end{equation*}
where $\varphi\in[-\pi/6,5\pi/6]$.
The curves $\gamma_{BC}$ and $\gamma_{AB}$ are obtained by rotating $\gamma_{AC}$ around the third coordinate axis through the angles $\pm 2\pi/3$.

For an obtuse triangle $ABC$ with the $\angle BAC =\pi/6$, $\angle ABC =\pi/4$ and $\angle ACB =7\pi/12$ ($u=2-\sqrt{3}$ and $v=1$) we have
\begin{equation*}
 \gamma_{AC}(\varphi)=\left(\, \cos \varphi, \, \sin \varphi, \, \sqrt{2 \cos \varphi + 2 \sin \varphi} \, \right),\mbox{\,\,\,where\,\,\,}
\varphi\in\left[-\pi/4,3\pi/4\right];
\end{equation*}
\begin{equation*}
\gamma_{BC}(\varphi)=\left(\, \cos \varphi, \, \sin \varphi, \, \sqrt{-\sqrt{3} \cos \varphi + 3 \sin \varphi} \, \right),\mbox{\,\,\,where\,\,\,}
\varphi\in\left[{\pi}/{6},{7\pi}/{6}\right].
\end{equation*}

For an Egyptian triangle $ABC$ (it is a right-angled triangle whose lengths of sides form ratios $3 : 4 : 5$) with $u=0$ and $v=\frac{4}{3}$ we have
\begin{equation*}
\gamma_{AC}(\varphi)=\left(\, \cos \varphi, \, \sin \varphi, \, \frac{4}{5}\sqrt{4 \cos \varphi + 3\sin \varphi} \, \right),
\end{equation*}
where $\varphi\in[-\arctan(4/3),\pi - \arctan(4/3)]$;
\begin{equation*}
\gamma_{BC}(\varphi)=\left(\, \cos \varphi, \, \sin \varphi, \, \frac{2}{5}\sqrt{-9 \cos \varphi + 12 \sin \varphi} \, \right),
\end{equation*}
where $\varphi\in\left[\arctan(3/4),\arctan(3/4) + \pi\right]$.
\end{example}

In Fig.~\ref{Image_Cyl}, we show the cylinder images for the cases when the base $ABC$ is an acute-angled,
right-angled and obtuse-angled triangle, respectively.
In Fig.~\ref{Part_image_Cyl}, the parts of the cylinder image that are parts of the boundary $\BFT$ of $\CFT$ are depicted (see the last section for details).

Fig. \ref{Image0} shows the image of the set $\Pi^0 \setminus S$ under the map $F$, the lower part of the surface is the image of the set $\mathbb{B}_{-}$,
and the upper part is the image of the set $\mathbb{B}_{+}$.

\begin{figure}[t]
\begin{minipage}[h]{0.31\textwidth}
\center{\includegraphics[width=0.89\textwidth]{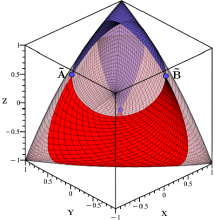}} \\ a)  \\
\end{minipage}
\quad
\begin{minipage}[h]{0.31\textwidth}
\center{\includegraphics[width=0.89\textwidth]{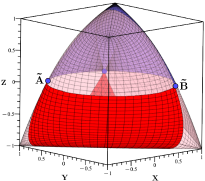}}   \\ b) \\
\end{minipage}
\quad
\begin{minipage}[h]{0.31\textwidth}
\center{\includegraphics[width=0.89\textwidth]{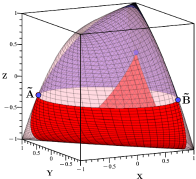}}   \\ c) \\
\end{minipage}
\begin{minipage}[h]{0.31\textwidth}
\center{\includegraphics[width=0.89\textwidth]{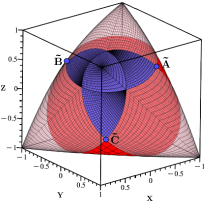}} \\ d)  \\
\end{minipage}
\quad
\begin{minipage}[h]{0.31\textwidth}
\center{\includegraphics[width=0.89\textwidth]{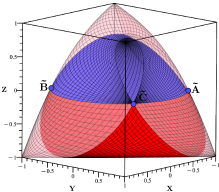}}   \\ e) \\
\end{minipage}
\quad
\begin{minipage}[h]{0.31\textwidth}
\center{\includegraphics[width=0.89\textwidth]{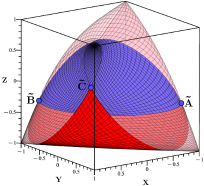}}   \\ f) \\
\end{minipage}
\caption{The image of $\Pi^0 \setminus S$ for:
a) and d) an acute-angled $\triangle ABC$; b) and e) a right-angled $\triangle ABC$; c) and f) an obtuse-angled $\triangle ABC$.
Here $\mathbb{BP}$ is drawn in orange, the set $F(\mathbb{B}_{-})$ is shown in red, and the set $F(\mathbb{B}_{+})$ is shown in purple.
The blue dots are the points $\widetilde{A}$, $\widetilde{B}$, and $\widetilde{C}$.}
\label{Image0}
\end{figure}

\medskip

\subsection{On regions bounded by special toroids}\label{subs.8.2}
Without loss of generality, we may suppose that $\angle   BAC =\alpha_1\leq \angle ABC =\alpha_2\leq \angle ACB=\alpha_3$.
It is clear that $\alpha_2 <\pi/2$.
\smallskip

Recall that $S$ is the  circumscribed circle for the triangle $ABC$, $\mathbb{B}_{-}$ is the open disc on the plane $\Pi^0$, bounded by $S$,
and $\mathbb{B}_{+}$ is the set of points outside $S$ (see \eqref{eq.decom.plane1}).
Denote by $W_1$, $W_2$,  and $W_3$ the bounded components of  $\Pi^+ \setminus \mathbb{T}_{BC}^{\,\alpha_1}$,
$\Pi^+ \setminus \mathbb{T}_{AC}^{\,\alpha_2}$, and $\Pi^+ \setminus \mathbb{T}_{AB}^{\,\alpha_3}$, respectively,
where $\mathbb{T}_{BC}^{\,\alpha_1}$,
$\mathbb{T}_{AC}^{\,\alpha_2}$, and $\mathbb{T}_{AB}^{\,\alpha_3}$ are special toroids.
\smallskip

If $\alpha_3 \geq \pi/2$, then $W_3 \subset W_1$ and $W_3 \subset W_2$. Therefore, $W_3 \bigcap W_i =W_3$, $i=1,2$, moreover, $W_3$
is bounded by $\mathbb{T}_{AB}^{\alpha_3}$ and
$S\bigcup \mathbb{B}_{-}$.
Note that $W_3$ does not contain points of the cylinder $\operatorname{Cyl}$ in this case.
Hence, the boundary of the image
$F(W_3)=F(W_3\cap W_1)=F(W_3\cap W_2)$ is contained in $F(\mathbb{B}_{-}) \bigcup \operatorname{Lim}(A)\bigcup \operatorname{Lim}(B)\bigcup \operatorname{Lim}(C)$.
Indeed,  $\operatorname{Cyl} \bigcap W_3=\emptyset$, therefore,
every point of $W_3$ is a non-degenerate point of the function $F$, hence it could not be a boundary point of $F(W_3)$ by Theorem \ref{theo.nondeg}.

Moreover, if $\alpha_3 > \pi/2$, then  $\mathbb{B}_{-}\bigcap \mathbb{T}_{AB}^{\,\alpha_3}$ is non-empty and is homeomorphic to the interval $(0,1)$,
hence, the set  $\operatorname{Lim}(C)\bigcap F(\mathbb{B}_{-})$ has the same property, see Fig.~\ref{Image0}~f) and Fig.~\ref{Image1}~f).
\smallskip

Now, we consider $i$ and $j$, such that $i\neq j$, $\alpha_i <\pi/2$, and $\alpha_j <\pi/2$.
In this case $W_i \bigcap W_j$ are bounded by two corresponding toroids and by $\Pi^0$.
The most important feature is that  $W_i \bigcap W_j$ contains some part of the cylinder $\operatorname{Cyl}$.
Indeed, the intersection $W_i \bigcap W_j \bigcap \operatorname{Cyl}$
is a special region on the cylinder $\operatorname{Cyl}$, see Definition \ref{def.spec.reg}.
Therefore, in this case  $F\bigl(W_i \bigcap W_j \bigcap \operatorname{Cyl}\bigr)$ can have a non-empty intersection
with the  boundary of  $F(W_i \bigcap W_j)$. Hence, the boundary of  $F(W_i \bigcap W_j)$ is contained in the union of
$F\bigl(\mathbb{B}_{-}\bigr)$, three solid ellipses
$\operatorname{Lim}(A)$, $\operatorname{Lim}(B)$, and $\operatorname{Lim}(C)$ (see Proposition \ref{pr.limitpoint.1}) and $F(\operatorname{Cyl})$.
\medskip

\subsection{On the region inside $\mathbb{BP}$}\label{subs.8.3}

As usual, we fix a base $ABC$ and suppose that $\angle BAC \leq \angle ABC \leq \angle ACB$.

\begin{lemma}\label{le.cont.b}
Suppose that $(x,y,z)\in  F(\Pi^0\setminus \{A,B,C\})\subset \mathbb{BP}$,
then for any $\varepsilon >0$ there exist points in $F(\Pi^{+})$ {\rm(}in particular, inside $\mathbb{BP}${\rm)} at distances to $(x,y,z)$ smaller than~$\varepsilon$.
\end{lemma}

\begin{proof} Suppose that $D=(p,q,0)\in \Pi^0\setminus \{A,B,C\}$ with $F(D)=(x,y,z)$.
It is clear that there is  a sequence of the points $(p_n,q_n,r_n)\in \Pi^{+}$ such that $(p_n,q_n,r_n) \rightarrow (p,q,0)$ as $n \to \infty$.
Using the fact that $F$ is continuous at every point of $\Pi^0\setminus \{A,B,C\}$,
we get that $F(p_n,q_n,r_n)\rightarrow F(p,q,0)=(x,y,z)$ as  $n \to \infty$. By definition, the points $F(p_n,q_n,r_n)$ are inside $\mathbb{BP}$. The lemma is proved.
\end{proof}
\smallskip

Let us consider the point
\begin{equation}\label{eq.importa.p}
\mathbb{R}^3\ni (x_0,y_0,z_0):=(\cos \angle BAC, \cos \angle ABC, \cos \angle ACB\Bigr).
\end{equation}

\smallskip

By Proposition \ref{pr.limitpoint.2},
the point $(x_0,y_0,z_0)$ is a common point of three planes containing
the limit solid ellipses $\operatorname{Lim}(A)$, $\operatorname{Lim}(B)$, and $\operatorname{Lim}(C)$
(the equations of these planes are  $x=x_0$, $y=y_0$, $z=z_0$, respectively).
Moreover, $(x_0,y_0,z_0)$ lies in the interior of $\mathbb{P}$, on the surface $\mathbb{BP}$, and outside of $\mathbb{P}$, according to as $\triangle ABC$ is acute-angled, right-angled or obtuse-angled, respectively (which is equivalent to $z_0$ >0, $z_0=0$, $z_0<0$, respectively).

Now, we consider eight open connected subsets of the following set:
$$
\left\{ (x,y,z)\in \mathbb{R}^3\,|\, x\neq x_0, y\neq y_0, z\neq z_0 \right\}.
$$
We will call it $\Xi(i,j,k)$, where $i=\sgn(x-x_0)$, $j=(y-y_0)$, $k=\sgn(z-z_0)$, and $(x,y,z)$ is any point in a given subset.
We are going to mark every  $\Xi(i,j,k)$ by the symbol ``$+$'' (respectively, ``$-$'') if $\Xi(i,j,k)$ contains
a point in the interior of the set $\CFT$ (respectively, a point in the interior of the set $\mathbb{R}^3 \setminus \CFT$).
If  $\Xi(i,j,k)$ contains some points of the interior of $\CFT$ as well as some points of $\mathbb{R}^3 \setminus \CFT$, then we mark it by symbol ``$\pm$''.
We will do it separately for the cases of acute-angled, right-angled and obtuse-angled base.

By Corollary \ref{cor.com.point}, the point $(x_0,y_0,z_0)$ is situated in the interior of the set $\CFT$ {\rm(}the closure of $F(\GT)${\rm)}
if the triangle $ABC$ is acute-angled (i.e., $z_0$ >0). Therefore, each region  $\Xi(i,j,k)$ contains
a point in the interior of the set $\CFT$. Moreover, in the acute-angled case we have
$F(\mathbb{B}_{-}) = \Xi(-,-,-) \bigcap \mathbb{BP}$, therefore, $\Xi(-,-,-)$ should be marked by ``$+$''.
It is easy to see also that
$$
\Bigl(\Xi(-,+,+)\bigcup \Xi(+,-,+)\bigcup \Xi(+,+,-) \bigcup \Xi(+,+,+) \Bigr) \bigcap \bigl(\mathbb{BP} \setminus \{(1,1,1)\}\bigr)\subset F(\mathbb{B}_{+}),
$$
therefore, $\Xi(-,+,+)$,  $\Xi(+,-,+)$, $\Xi(+,+,-)$, and $\Xi(+,+,+)$ should also be marked by ``$+$''.
On the other hand, $\Xi(-,-,+)$,  $\Xi(-,+,-)$, and $\Xi(+,-,-)$ should  be marked by ``$\pm$''.
This observation corresponds to the fact that the images of special regions on the cylinder $\operatorname{Cyl}$
(see Definition~\ref{def.spec.reg}) are in the subsets  $\Xi(-,-,+)$,  $\Xi(-,+,-)$, and $\Xi(+,-,-)$, see Subsections \ref{subs.8.1} and \ref{subs.8.2}.

{\scriptsize\begin{table}[t]
\caption{Markers of the regions $\Xi(i,j,k)$ for the acute-angled, right-angled, and obtuse-angled cases}
\centering
\begin{tabular}{|c|c| c | c | c | c | c | c | c |}
\hline
 &\! $(-,-,-)$ \!&\! $(-,-,+)$ \! &\! $(-,+,-)$\! & \! $(-,+,+)$ \! &\!$(+,-,-)$\! &\! $(+,-,+)$\!  &\!$(+,+,-)$\! &\! $(+,+,+)$\!
\\
\hline
$z_0>0$ &$+$&$\pm$&$\pm$&$+$&$\pm$&$+$&$+$&$+$\\
\hline
$z_0=0$ &$+$&$\pm$&$-$&$+$&$-$&$+$&$-$&$+$ \\
\hline
$z_0<0$ &$+$&$\pm$&$-$&$+$&$-$&$+$&$-$&$+$ \\
\hline
\end{tabular}\label{Omega.reg}
\end{table}}

The cases of right-angled and obtuse-angled bases can be considered similarly.  The principal feature of these cases is that  only
$\Omega(-,-,+)$ should  be marked by ``$\pm$''. It is related to the fact that there is only one
special region on the cylinder $\operatorname{Cyl}$ for these cases, see Subsection \ref{subs.8.1}. All corresponding computations are included in Table \ref{Omega.reg}.

\section{The main construction}\label{sec.main}

Now, we are going to construct the boundary of $\operatorname{CFT}$ for any given base $ABC$.
Recall that $\GT=\Pi^+ \bigcup \bigl(\Pi^0 \setminus \{A,B,C\} \bigr)$, where
$\Pi^+=\{(p,q,r) \in \mathbb{R}^3\,|\, r>0\}$ and $\Pi^0=\{(p,q,0) \in \mathbb{R}^3\}$.
Our main goal is to describe the image $F(\GT)$ under the map $F:\GT \to \mathbb{R}^3$, see \eqref{eq.fmain}.

To do it we need to study the boundary $\BFT$ of the closure $\CFT$ of $F(\GT)$ in the standard topology of $\mathbb{R}^3$.
We know that (see Theorem \ref{theo.1})
$$
F(\GT)\subset \mathbb{P}=\left\{(x,y,z)\in [-1,1]^3\,|\,1+2xyz-x^2-y^2-z^2 \geq 0 \right\},
$$
hence, $\BFT \subset \CFT \subset  \mathbb{P}$. Moreover, $F\bigl(\Pi^0 \setminus \{A,B,C\}\bigr) \subset \mathbb{BP}$, where
$\mathbb{BP}$ is the boundary of $\mathbb{P}$, see \eqref{eq.pillowc}.

We know also that
the closure of $F\Bigl(\Pi^0\setminus \{A,B,C\}\Bigr)$ is a subset of
$$
\mathbb{BP}=\left\{(x,y,z)\in [-1,1]^3\,|\,1+2xyz-x^2-y^2-z^2 = 0 \right\}
$$
and it has the form
$F\Bigl(\Pi^0\setminus \{A,B,C\}\Bigr)\bigcup E_A \bigcup E_B \bigcup E_C$,
where $E_A, E_B, E_C$ are ellipses, that
are the intersections of the surface $\mathbb{BP}$
with the planes $x=\cos \angle BAC$, $y=\cos \angle ABC$, $z=\cos \angle ACB$ respectively (see Proposition \ref{pr.closure.plane.1}).
Recall that the limit solid ellipses $\operatorname{Lim}(A)$, $\operatorname{Lim}(B)$, $\operatorname{Lim}(C)$
are the convex hulls of $E_A, E_B, E_C$ respectively, see Proposition \ref{pr.limitpoint.1}.

\begin{figure}[t]
\begin{minipage}[h]{0.31\textwidth}
\center{\includegraphics[width=0.89\textwidth]{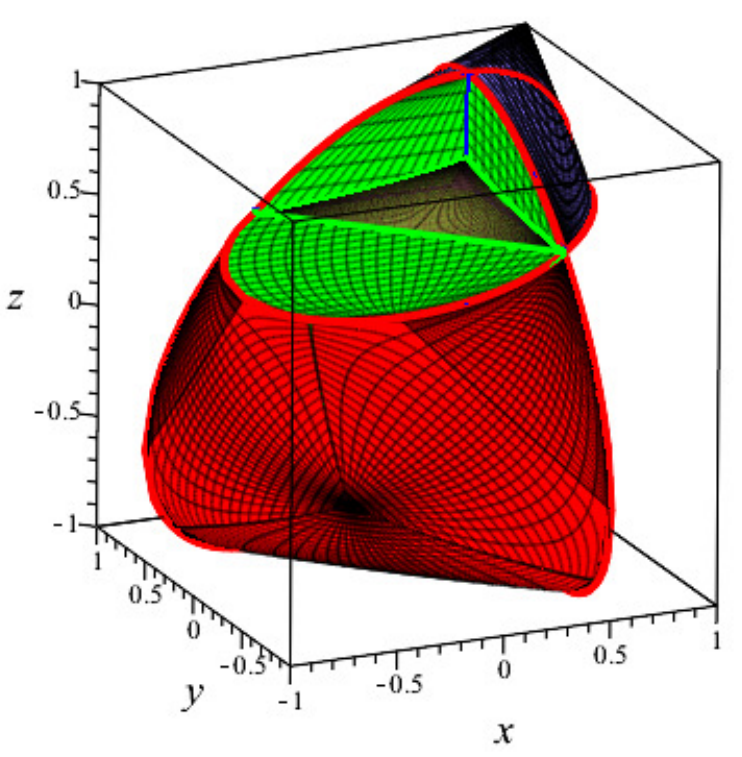}} \\ a)  \\
\end{minipage}
\quad
\begin{minipage}[h]{0.31\textwidth}
\center{\includegraphics[width=0.89\textwidth]{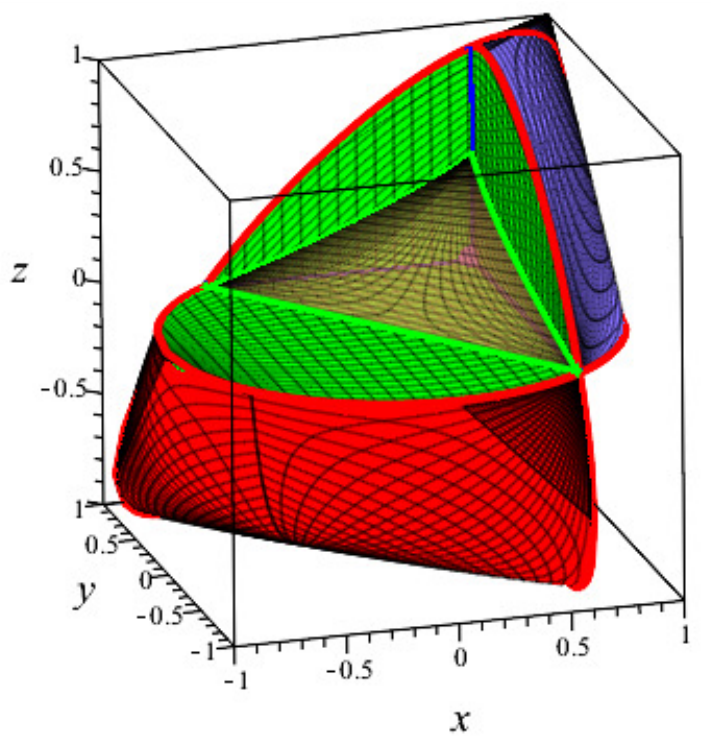}}   \\ b) \\
\end{minipage}
\quad
\begin{minipage}[h]{0.31\textwidth}
\center{\includegraphics[width=0.89\textwidth]{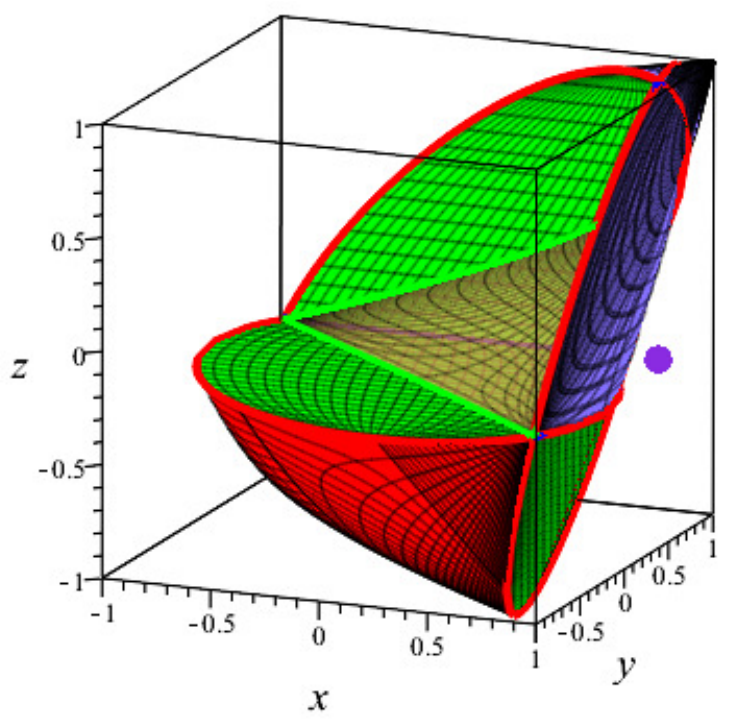}}   \\ c) \\
\end{minipage}
\begin{minipage}[h]{0.31\textwidth}
\center{\includegraphics[width=0.89\textwidth]{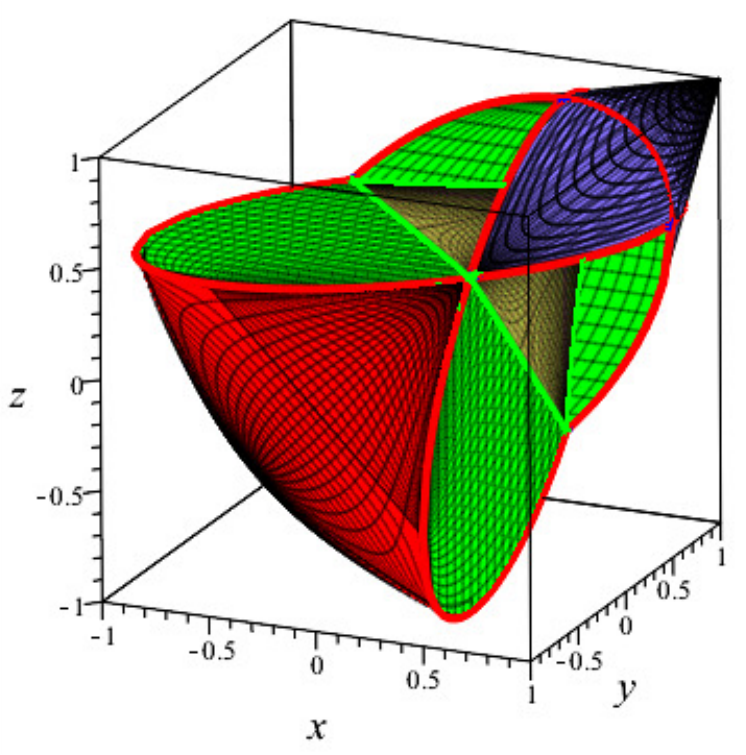}} \\ d)  \\
\end{minipage}
\quad
\begin{minipage}[h]{0.31\textwidth}
\center{\includegraphics[width=0.89\textwidth]{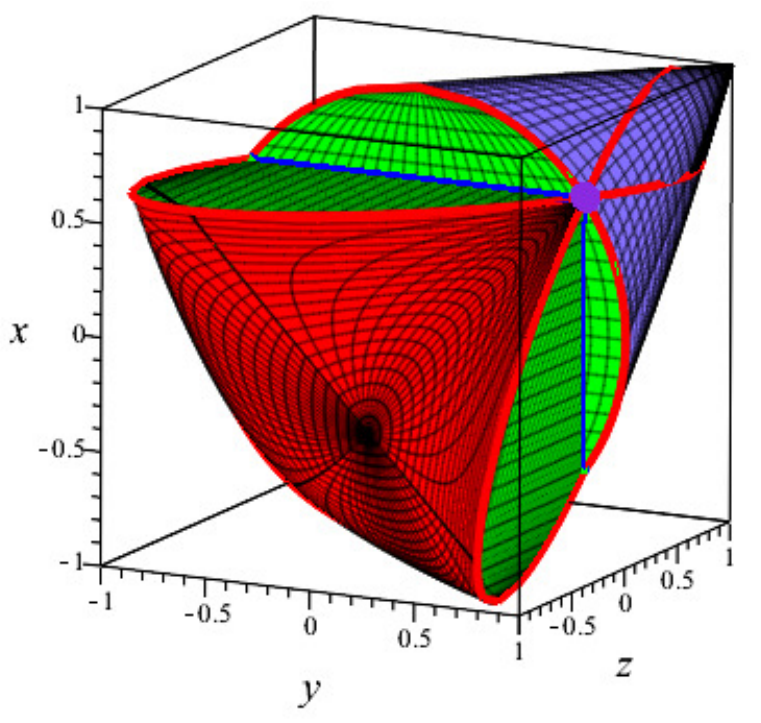}}   \\ e) \\
\end{minipage}
\quad
\begin{minipage}[h]{0.31\textwidth}
\center{\includegraphics[width=0.89\textwidth]{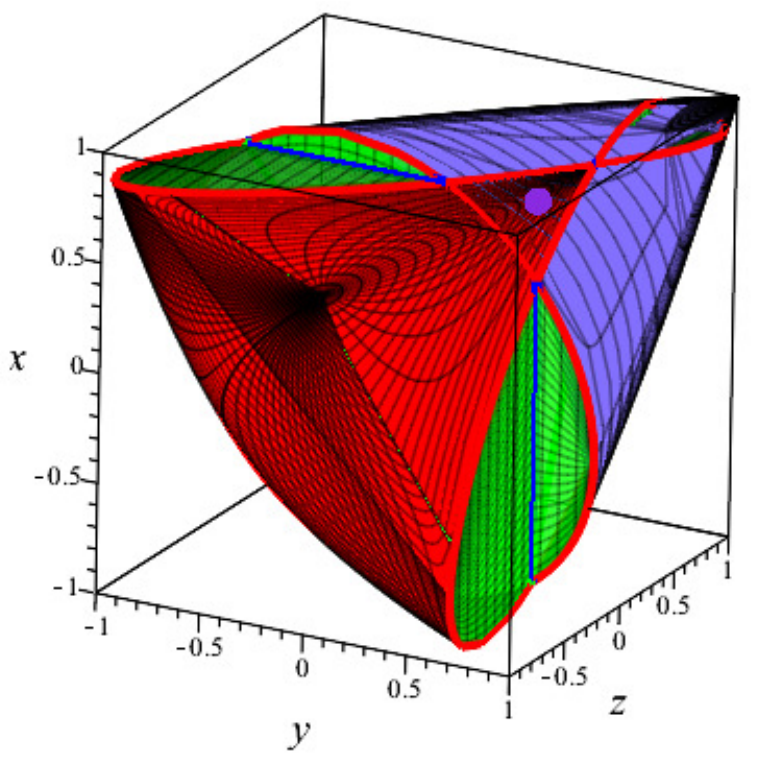}}   \\ f) \\
\end{minipage}
\caption{Boundary of $\CFT$ for:
a) and d) an acute-angled $\triangle ABC$; b) and e) a right-angled $\triangle ABC$; c) and f) an obtuse-angled $\triangle ABC$.
Here the set $F(\mathbb{B}_{-})$ is shown in red, and the set $F(\mathbb{B}_{+})$ is shown in purple. Parts of the limit solid ellipses are shown in green,
the images of special regions of the cylinder $\operatorname{Cyl}$ under the map $F$
are drawn in brown. The purple point is the intersection point of three planes containing three limit solid ellipses.}
\label{Image1}
\end{figure}

The point $I_{\triangle ABC}:=\Bigl( \cos \angle BAC, \cos \angle ABC, \cos \angle ACB\Bigr)$
is the intersection of three planes, spanned by the ellipses $E_A, E_B, E_C$, respectively.
This point is inside of $\mathbb{BP}$, on $\mathbb{BP}$, or outside of $\mathbb{BP}$ exactly in the case, when
$\triangle ABC$ is acute-angled, right-angled, or obtuse-angled, respectively.
Moreover,
\begin{eqnarray*}
\widetilde{A}&:=&\Bigl(-\cos \angle BAC, \cos \angle ABC, \cos \angle ACB\Bigl) \in E_B\bigcap E_C,\\
\widetilde{B}&:=&\Bigl(\cos \angle BAC, -\cos \angle ABC, \cos \angle ACB\Bigl) \in E_A\bigcap E_C,\\
\widetilde{C}&:=&\Bigl(\cos \angle BAC, \cos \angle ABC, -\cos \angle ACB\Bigl) \in  E_A\bigcap E_B.
\end{eqnarray*}

All the above observations lead to the following important result.

\begin{theorem}\label{final.boundary}
For any given base $ABC$, the closure  $\operatorname{CFT}$ of  $F(\GT)$ is such that $\operatorname{CFT}\subset \mathbb{P}$
and is homeomorphic to  to a $3$-dimensional closed ball, while
the boundary $\operatorname{BFT}$ of  $\operatorname{CFT}$ is homeomorphic to a $2$-dimensional sphere $S^2$.
Moreover, $\operatorname{BFT}\bigcap \mathbb{BP}$ is the union
of $F(\mathbb{B}_{-})$,  $F(\mathbb{B}_{+})$, $\{(1,1,1)\}$, $E_A$, $E_B$, and $E_C$,
while $\operatorname{BFT}\setminus \mathbb{BP}$ is the union of several connected subsets of
the limit solid ellipses $\operatorname{Lim}(A)$, $\operatorname{Lim}(B)$, and $\operatorname{Lim}(C)$, and the images of  special regions on the cylinder
$\operatorname{Cyl}$.
\end{theorem}

\begin{proof} Let us recall that $F\Bigl(\Pi^0\setminus \{A,B,C\}\Bigr)\subset \mathbb{P} \bigcap F(\GT)$ and (see Proposition \ref{pr.closure.space.1})
\begin{eqnarray*}
&&\operatorname{BFP}:=F\Bigl(\Pi^0\setminus \{A,B,C\}\Bigr)\bigcup \{(1,1,1)\}\bigcup \operatorname{Lim}(A) \bigcup \operatorname{Lim}(B) \bigcup \operatorname{Lim}(C)\\
&&\subset \CFT=F\Bigl(\Pi^+ \bigcup \Pi^0\setminus \{A,B,C\}\Bigr) \bigcup \{(1,1,1)\}\bigcup \operatorname{Lim}(A) \bigcup \operatorname{Lim}(B) \bigcup \operatorname{Lim}(C).
\end{eqnarray*}

Note also that $\operatorname{BFP}$
is a closed surface in $\mathbb{R}^3$ with self-intersections (it is an immersion of $S^2$, as it was discussed in the final part of Section \ref{sec.limit.point}).
In particular, it bounds one of several open connected subsets of $\mathbb{R}^3$.

On the other hand, $\BFT$ cannot consist only of points belonging to $\operatorname{BFP}$,
since $F(\operatorname{Cyl}\setminus \Pi^0)$ is not located inside the sets bounded by the surface $\operatorname{BFP}$
(see Table \ref{Omega.reg} and related discussions).

According to Theorem \ref{theo.nondeg} and Corollary \ref{cor.cyl.noi},
only some bounded part of the cylinder $\operatorname{Cyl}$ is mapped into the set $\BFT$ under the mapping $F$.
Indeed, there is a real number $M$ such that for any $r>M$ the image of any point
$D=(p,q,r)\in \operatorname{Cyl}\bigcap \Pi^+ =S\times (0,\infty)$, $p^2+q^2=1$, is not on the boundary $\BFT$ of $\CFT$.

Since $\operatorname{Cyl}\bigcap \mathbb{B}_{-} =\emptyset$ and $\operatorname{Cyl}\bigcap \mathbb{B}_{+} =\emptyset$,
therefore some part of $F(\operatorname{Cyl})$ can be a part of $\operatorname{BFT}$, only if its boundary is the union of some curves in
the limit solid ellipses $\operatorname{Lim}(A)$, $\operatorname{Lim}(B)$, and $\operatorname{Lim}(C)$ (see Proposition \ref{pr.limitpoint.1}).
It should be noted that some part of any of the limit segments $[\widetilde{B},\widetilde{C}]$, $[\widetilde{A},\widetilde{C}]$, $[\widetilde{A},\widetilde{B}]$ may be part of such a boundary
(see Proposition \ref{pr.limitpoint.c1}).
Thus, the most important role in this construction is played by curves that are the intersections of the image of the cylinder
$\operatorname{Cyl}\bigcap \Pi^+$ with the limit solid ellipses $\operatorname{Lim}(A)$, $\operatorname{Lim}(B)$, and $\operatorname{Lim}(C)$.

Based on all of the above, we come to the conclusion that such parts of the image of a cylinder $\operatorname{Cyl}$
can only be the images of {\it special regions} on this cylinder
(see Definition~\ref{def.spec.reg}, Proposition \ref{pr.crit.tor.cyl.2}, and Fig.~\ref{Part_image_Cyl}).
All other points of $\BFT$ are  some points of $\operatorname{BFP}$.
Using the reasoning in the previous section (see, e.g., Table \ref{Omega.reg}), we get the required result immediately.
\end{proof}

\begin{remark} The above results imply also that $\mathbb{P}\setminus \operatorname{CFT}$ is a disjoint union of three convex sets for any base $ABC$, see
Fig.~\ref{firstimage} for details.
\end{remark}

\smallskip

Next we will discuss the features of constructing the body $\operatorname{CFT}$ and the surface $\operatorname{BFT}$
in the cases of acute-angled, right-angled and obtuse-angled base.

\smallskip

{\bf Let us suppose that the triangle $ABC$ is acute-angled.}
In this case we can apply
Corollary~\ref{cor.com.point}, which implies that
then the point
$$
I_{\triangle ABC}=\Bigl(\cos \angle BAC, \cos \angle ABC, \cos \angle ACB\Bigr)
$$
is contained in the interior of the set $\CFT$ {\rm(}the closure of $F(\GT)${\rm)}.
Moreover, $I_{\triangle ABC}$ is the intersection of three solid ellipses
$\operatorname{Lim}(A)$, $\operatorname{Lim}(B)$, and $\operatorname{Lim}(C)$.
It is clear that $\BFT$ (the boundary of $\CFT$) should contain the images of three special regions on $\operatorname{Cyl}$,
see Subsections \ref{subs.8.1} and \ref{subs.8.3}.

\smallskip

{\bf Let us suppose that the triangle $ABC$ is obtuse-angled.}
It is clear that $\BFT$ (the boundary of $\CFT$) should contain the image of a unique special region on $\operatorname{Cyl}$,
see Subsections \ref{subs.8.1} and \ref{subs.8.3}.

\smallskip

{\bf Finally, let us suppose that the triangle $ABC$ is right-angled.}
As in the previous case $\BFT$ (the boundary of $\CFT$) should contain the image of a unique special region on $\operatorname{Cyl}$,
see Subsections \ref{subs.8.1} and \ref{subs.8.3}.
But in this case the point
$I_{\triangle ABC}=\Bigl(\cos \angle BAC, \cos \angle ABC, \cos \angle ACB\Bigr)$
is situated on $\BFT$ (the boundary of $\CFT$).
\smallskip

All these properties are illustrated in Fig. \ref{Image1},
where the boundary of $\CFT$ is drawn for an
acute-angled base $ABC$, a right-angled base $ABC$, and for an obtuse-angled base $ABC$.

\smallskip

Theorem~\ref{final.boundary} gives a complete description
the closure  $\operatorname{CFT}$ of  $F(\GT)$ for any given base $ABC$.
It is clear that
the interior $\operatorname{int} (\operatorname{CFT})$ is a subset of $F(\Pi^+)$. Using the obtained results, we can give an answer to Problem~\ref{prob.1}.
Note that a point of the boundary $\operatorname{bd} (\operatorname{CFT})$ can also be contained  in $F(\Pi^+)$.
This possibility is described in Theorem~\ref{theo.nondeg} and Corollary~\ref{cor.cyl.deg} as follows.

Let consider the set $\operatorname{SRT}$ of points $D \in \operatorname{Cyl}$ such that $D$ is either from some special region of $\operatorname{Cyl}$
or from the boundary of a special region in $\operatorname{Cyl}\bigcap \Pi^+$ (see Definition \ref{def.spec.reg}).
Then we have the following result:
\begin{equation}\label{eq.main}
\Sigma(\triangle ABC)=F(\Pi^+)=\operatorname{int} (\operatorname{CFT}) \bigcup F(\operatorname{SRT}).
\end{equation}

We hope that this (somewhat unusual) formula will be useful in studying related problems.
It should be noted that for a right-angled base and for an obtuse-angled base, the set $F(\operatorname{SRT})$
is homeomorphic to a triangle with a side removed, while for an acute-angled base,
the set $F(\operatorname{SRT})$ is homeomorphic to a disjoint union of three triangles, each of which is deprived of one of its sides (see Fig.~\ref{Part_image_Cyl}).
Recall that the sets $\operatorname{int} (\operatorname{CFT})$ and $\operatorname{bd} (\operatorname{CFT})$
are homeomorphic to an open three-dimensional ball and a two-dimensional sphere, respectively.

\vspace{15mm}

\vspace{5mm}

\end{document}